\documentclass[a4,11pt,usenames]{article}
\usepackage{amssymb}
\usepackage{mathrsfs}

\usepackage{amsmath}
\usepackage{amsthm}
\usepackage[all]{xy}
\usepackage{authblk}
\usepackage{accents}
\usepackage{imakeidx}

\usepackage[backend=biber,style=numeric]{biblatex}

\DeclareFontFamily{OT1}{pzc}{}
\DeclareFontShape{OT1}{pzc}{m}{it}{<-> s * [1.150] pzcmi7t}{}
\DeclareMathAlphabet{\mathpzc}{OT1}{pzc}{m}{it}

\usepackage[totalwidth=18cm,totalheight=24.2cm]{geometry}

\makeatletter
\let\mcnewpage=\newpage
\newcommand{\Trick}{
\renewcommand\newpage{%
        \if@firstcolumn
            \hrule width\linewidth height0pt
                \columnbreak
        \else
                \mcnewpage
        \fi
}
}
\makeatother

\makeatletter
\renewcommand\subsubsection{\@startsection{subsubsection}{3}{\z@}%
                                     {-3.25ex\@plus -1ex \@minus -.2ex}%
                                     {-1.5ex \@plus .2ex}%
                                     {\normalfont\normalsize\bfseries}}
\makeatother

\theoremstyle{plain}

\newtheorem{theorem}{Theorem}[section]
\newtheorem{proposition}[theorem]{Proposition}
\newtheorem{corollary}[theorem]{Corollary}
\newtheorem{lemma}[theorem]{Lemma}
\newtheorem{sublemma}[theorem]{Sublemma}

\newtheorem{maintheorem}{Theorem}

\newenvironment{mainthm}[1]
{\renewcommand{\themaintheorem}{#1}
\begin{maintheorem}}
{\end{maintheorem}}

\newtheorem{maincorollary}{Corollary}

\newenvironment{maincor}[1]
{\renewcommand{\themaincorollary}{#1}
\begin{maincorollary}}
{\end{maincorollary}}

\newtheorem{mainproposition}{Proposition}

\newenvironment{mainprop}[1]
{\renewcommand{\themainproposition}{#1}
\begin{mainproposition}}
{\end{mainproposition}}

\newtheorem{mainlm}{Lemma}

\newenvironment{mainlemma}[1]
{\renewcommand{\themainlm}{#1}
\begin{mainlm}}
{\end{mainlm}}

\theoremstyle{definition} 
\newtheorem{definition}[theorem]{Definition}

\newtheorem{example}[theorem]{Example}
\newtheorem{remark}[theorem]{Remark}

\newtheorem{question}[theorem]{Question}

\newtheorem*{notation}{Notation}

\newcommand{\thistheoremname}{}
\newtheorem*{genericthm*}{\thistheoremname}
\newenvironment{namedthm*}[1]
  {\renewcommand{\thistheoremname}{#1}%
   \begin{genericthm*}}
  {\end{genericthm*}}

\makeatletter
\let\@wraptoccontribs\wraptoccontribs
\makeatother

\def\dim{\mathrm{dim}}
\def\dis{\displaystyle}
\def\SU{\mathrm{SU}}
\def\INTE{\mathcal{I}}
\def\PSU{\mathrm{PSU}}
\def\SL{\mathrm{SL}}

\def\PSL{\mathrm{PSL}}
\def\SO{\mathrm{SO}}
\def\forg{\wh{\mathpzc{F}}}
\def\cla{\mathcal{C}}
\def\clal{\mathfrak{c}}
\def\HEM{H}

\def\odd{\mathrm{odd}}

\def\psl{\mathfrak{sl}}
\def\su{\mathfrak{su}}
\def\Ad{\mathrm{Ad}}

\def\SU{\mathrm{SU}}

\def\Bcal{\mathcal{B}}
\def\Zfrak{\mathfrak{Z}}
\def\hfrak{\mathfrak{h}}
\def\Nor{\mathrm{Nm}}
\def\comm{c}
\def\Comm{\bm{c}}
\def\PR{P}
\def\tr{\mathrm{tr}}

\def\ZZ{\mathbb{Z}}
\def\KK{\mathbb{K}}
\def\NN{\mathbb{N}}
\def\RR{\mathbb{R}}
\def\CC{\mathbb{C}}
\def\PP{\mathbb{P}}

\def\SPAN#1{\langle{#1}\rangle}
\def\deg{\mathrm{deg}}

\def\ol#1{\overline{#1}}

\def\bm#1{\text{\boldmath$#1$}}
\def\wti#1{\widetilde{#1}}

\def\wh#1{\widehat{#1}}
\def\pa{\partial} 
\def\e{\varepsilon}
\def\Gcal{\mathcal{G}}
\def\th{\vartheta}
\def\Hom{\mathpzc{Hom}}
\def\Rep{\mathpzc{Rep}}
\def\IM{\mathrm{Im}} 
\def\rar{\rightarrow}
\def\lra{\longrightarrow}

\def\arr#1#2{\stackrel{#1}{#2}}
\def\hra{\hookrightarrow}

\def\MCG{\mathrm{MCG}}

\def\Axis{A}

\def\Sph{\mathbb{S}^2}

\def\Kill{\mathcal{K}}

\def\stab{\mathrm{stab}}

\def\grp{\Pi}
\def\ANGL{\mathpzc{A}}
\def\dd{\bar{\delta}}
\def\Ll{\dd}
\def\ee{\mathrm{e}}
\def\EE{E}
\def\COAX{\mathrm{Coax}}

\def\base{\ast}

\makeindex
\usepackage[columns=3]{idxlayout}

\begin{document}

\renewcommand\indexname{Table of symbols}

\title{On decorated representation spaces associated to spherical surfaces}

\author[1]{Gabriele Mondello}
\affil[1]{{\small{``Sapienza'' Universit\`a di Roma, Department of Mathematics (mondello@mat.uniroma1.it)}}}

\author[2]{Dmitri Panov}
\affil[2]{{\small{King's College London, Department of Mathematics (dmitri.panov@kcl.ac.uk)}}}

\date{\vspace{-1.5cm}}

\maketitle

\smallskip 

\begin{center}
{\small{with an appendix by Daniil Mamaev}}
\end{center}

\medskip

\abstract{\noindent
We analyse local features of
the spaces of representations 
of the fundamental group of a punctured surface in
$\SU_2$ equipped with a decoration, namely a choice
of a logarithm of the representation at peripheral loops.
Such decorated representations naturally arise as monodromies
of spherical surfaces with conical points.
Among other things, in this paper we determine 
the smooth locus of such absolute and relative
decorated representation spaces: in particular, in the relative case (with few special exceptions) such smooth locus is dense, connected, and exactly consists of non-coaxial representations.
The present study sheds some light on the local structure
of the moduli space of spherical surfaces with conical points, 
which is locally modelled
on the above-mentioned decorated representation spaces.
}
%
%


\setlength{\parindent}{0pt}
\setlength{\parskip}{0.2\baselineskip}

\section{Introduction}

A spherical surface $(S,\bm{x},h)$ is a compact, connected, oriented surface $S$
endowed with a metric $h$ of constant curvature $1$ with conical singularities
at $\bm{x}=(x_1,\dots,x_n)$ of angles $2\pi\bm{\th}=(2\pi\th_1,\dots,2\pi\th_n)$.
In local polar coordinates $(r,\eta)$
such metric takes the form $dr^2+\sin^2(r)d\eta^2$
near every point of $\dot{S}=S\setminus\bm{x}$,
and the form $dr^2+\th^2_i\sin^2(r)d\eta^2$ near $x_i$ for every $i$.

Associated to the spherical metric $h$ on 
the punctured surface $\dot{S}$, we have
a monodromy homomorphism $\rho_h:\pi_1(\dot{S})\rar\SO_3(\RR)$.
Note that $\rho_h$ is well-defined only up to conjugation by elements of 
$\SO_3(\RR)$ but its conjugacy class $[\rho_h]$ is uniquely determined.

We remark that, though the monodromy homomorphism of a spherical metric naturally takes values in $\SO_3(\RR)\cong\PSU_2$, it always admits liftings to $\SU_2$.
As a representation in $\SU_2$ retains more information, 
in this paper we will focus on this case.


When no $\th_i$ is an integer, the deformation space of $(S,\bm{x},h)$
is modelled on the space of conjugacy classes of homomorphisms $\rho$
from $\pi_1(\dot{S})$ to $\SU_2$.
This is achieved by showing that the monodromy map,
that associates to a metric $h$ its monodromy representation $[\rho_h]$
(see \cite[Corollary 1(a)]{luo:monodromy}), is a local homeomorphism.

If $\th_i$ is an integer, the monodromy $[\rho_h]$
sends a peripheral loop of $\dot{S}$
that simply winds about the puncture $x_i$ to $\pm I\in\SU_2$.
Thus, $\rho_h$ fails to detect the exact position of the conical
point $x_i$: in more precise terms, there exist local surgeries around $x_i$
that move the conical point $x_i$ but keep the monodromy fixed.
Hence, passing from the deformation space of $(S,\bm{x},h)$ 
to the representation space results in a loss of information.
For this reason we introduce decorations.

Let $\Bcal$ be the subset of $\pi_1(\dot{S})$ consisting
of all classes of peripheral loops,
namely simple loops that winds about a puncture.
A {\it{decoration}} of a homomorphism $\rho$ is 
a map $\Axis:\Bcal\rar\su_2$ to the Lie algebra $\su_2$,
which is equivariant under the action
of $\pi_1(\dot{S})$ by conjugation on $\Bcal$ and via $\Ad_\rho$
on $\su_2$, and such that $\rho(\beta)=-\mathrm{exp}(2\pi\Axis(\beta))$ for all $\beta\in\Bcal$ (see Definition \ref{def:dec-hom}).
A {\it{decorated representation}} is a conjugacy class of pairs $(\rho,\Axis)$.

There is a very natural way to associate
a decorated representation to a spherical metric,
in such a way that the norm of $\Axis(\beta_i)$
is $\th_i$ for all peripheral loops $\beta_i$ that wind about $x_i$.

In a forthcoming paper \cite{MP:moduli} it will be shown that, in all cases,
the deformation space of spherical surfaces with conical points
is locally modelled on a space of decorated representations
via the decorated monodromy map. Similarly,
the locus of surfaces with conical points of assigned
angles is locally modelled on the subspace of representations
with decoration $\Axis$ of assigned norm.
\\


The aim of this paper is to study the local properties of the above-mentioned
decorated representation spaces, 
and in particular of the subsets of decorated representation
with finite stabilizers; in fact,
decorated monodromies of spherical metrics will have such properties,
as shown in Theorem \ref{mainthm:mon-sph}.

We will state our model results for the spaces of undecorated representations
(namely, Lemma \ref{mainlemma:alg}, Theorems \ref{mainthm:rep-undec}-\ref{mainthm:rep-undec-rel}-\ref{mainthm:rep-undec-nonsp}-\ref{mainthm:special-undec}
and Proposition \ref{mainprop:sympl}), and 
then we will present our results for the spaces of decorated representations
(Lemma \ref{mainlemma:anal}, Theorems \ref{mainthm:rep}-\ref{mainthm:rep-rel}-\ref{mainthm:rep-nonsp}-\ref{mainthm:special} 
and Corollary \ref{maincor:sympl-dec}).

Sometimes we will have to separately treat certain {\it{special}} cases,
which in fact only arise in genus $0$ and $1$ when 
all stabilizers (of the undecorated homomorphisms) are infinite.

In Lemma \ref{mainlemma:alg} we show that real representation spaces are semi-algebraic sets and complex representation spaces are complex algebraic sets.
Correspondingly, in Lemma \ref{mainlemma:anal} we will show that
decorated representation spaces are semi-analytic sets.

In Theorem \ref{mainthm:rep-undec} 
and Theorem \ref{mainthm:rep}
we study the smooth locus of
absolute representation spaces and absolute decorated representation spaces
respectively.
The relative representation spaces
are treated in Theorem \ref{mainthm:rep-undec-rel}
in the undecorated case, and in Theorem \ref{mainthm:rep-rel} in the decorated case.
Density and connectedness of the non-coaxial locus
in non-special cases are proven in Theorem
\ref{mainthm:rep-undec-nonsp} and Theorem \ref{mainthm:rep-nonsp};
the special cases are discussed in Theorem \ref{mainthm:special-undec}
and Theorem \ref{mainthm:special}.


In Proposition \ref{mainprop:sympl} we recall that, if no $\th_i$ is an integer,
relative representation spaces support a symplectic structure, 
namely Goldman symplectic structure.
Such form on the $\SU_2$-representation space
is the restriction of Goldman's complex symplectic form
on the $\SL_2(\CC)$-representation space. Since 
spherical metrics yield an underlying $\CC\PP^1$-structure, and since
$\PSL_2(\CC)$-representations
that are monodromies of $\CC\PP^1$-structures always admit a lifting to $\SL_2(\CC)$,
we will focus our discussion on the $(\SU_2,\SL_2(\CC))$-case.
Finally, in Corollary \ref{maincor:sympl-dec} we show how such symplectic
structure is induced on decorated representation spaces.
%

The above-mentioned main results of the present paper are stated in Section \ref{sec:main-results}.\\

In Appendix \ref{app:mamaev} Daniil Mamaev gives an elementary proof 
of a known existence criterion for closed polygons in $\mathbb{S}^3$.
Such criterion is relevant for the non-emptiness of our $\SU_2$ representations spaces in genus $0$ (see Section \ref{sec:conn0}).

\subsection{Setting}\label{sec:setting}

{\bf{Surface.}} In what follows $S$ will be a compact, connected, oriented surface of genus $g$\index{$g$, $n$}
and $\bm{x}=(x_1,\dots,x_n)$ an $n$-tuple of distinct points of $S$.
\index{$S$, $\bm{x}$, $\dot{S}$}
Unless differently specified, we will assume that the punctured surface $\dot{S}=S\setminus\bm{x}$
satisfies $\chi(\dot{S})<0$ and $n\geq 1$. 


{\bf{Fundamental group.}} We fix a basepoint $\base$ in $\dot{S}$ and
we denote by $\grp_{g,n}$\index{$\grp_{g,n}$, $\beta_i$, $\mu_j$, $\nu_j$}
 the fundamental group $\pi_1(\dot{S},\base)$.
We also pick a standard basis 
$\{{\mu}_1,{\nu}_1,{\mu}_2,\dots,{\nu}_g,{\beta}_1,\dots,{\beta}_n\}$
\index{$\mu_j$, $\nu_j$, $\beta_i$}
of $\grp_{g,n}$, which satisfies
the unique relation $[{\mu}_1,{\nu}_1]\cdots[{\mu}_g,{\nu}_g]
\cdot{\beta}_1\cdots{\beta}_n=e$.
Note that, as $n\geq 1$, such group is isomorphic to a free group with $2g+n-1$ generators.
We will denote by $\Bcal_i$\index{$\Bcal_i$, $\Bcal$}
the conjugacy class of ${\beta}_i$ and by $\Bcal$
the union $\bigcup_i \Bcal_i$.
Note that $\Bcal$ is the subset of $\pi_1(\dot{S},\base)$ consisting
of all classes of peripheral loops.

{\bf{Spherical structure with conical points.}} A metric $h$\index{$h$}
on $S$ 
is a {\it{spherical metric}} with angles $\bm{\th}=(\th_1,\dots,\th_n)$
\index{$\th_i$, $\bm{\th}$}
at $\bm{x}$
if $h$ has constant curvature $1$ on $\dot{S}$ and it has a conical singularity at $x_i$
of angle $2\pi\th_i$ for $i=1,\dots,n$. A {\it{spherical surface}} is a triple $(S,\bm{x},h)$ (see also \cite{eremenko:survey}).

{\bf{Killing form.}} We will endow the Lie algebra $\psl_2(\CC)$ with the Killing form $\Kill(X,Y):=-2\tr(XY)$,\index{$\Kill$, $|\cdot|$}
so that its restriction to $\su_2$ is positive-definite
and it induces a norm $\|X\|:=\sqrt{\Kill(X,X)}$.

{\bf{Adjoint action.}}
We denote by $\Sph$\index{$\Sph$} the unit $2$-sphere.
In this paper we will often identify it to the unit sphere in $\su_2$,
and will implicitly assume this identification without reminding.
The group $\SU_2$ isometrically acts on $\Sph$ via the adjoint action.
Such action factorizes through $\PSU_2$ and identifies such group to $\SO(\su_2)\cong\SO_3(\RR)$.
Upon identifying $\CC\PP^1$ with $\Sph$, the restriction of the 
$\PSL_2(\CC)$-action on $\CC\PP^1$ agrees with the $\SO_3(\RR)$-action
on $\Sph$ above mentioned.

{\bf{The map $\ee$.}}
We will often use the map
$
\ee:\su_2\lra\SU_2
$,\index{$\ee$} 
defined as $\ee(X):=-\exp(2\pi X)$.
Note that $\ee(X)$ is
conjugate to the diagonal matrix with eigenvalues
$\exp(i\pi(1\pm\|X\|))$.
The reason behind the above definition is that we want
$\ee(X)$ to be conjugate to
the $\SU_2$-monodromy 
of a spherical disk with one conical point of angle $2\pi\|X\|$.


{\bf{Conjugacy classes in $\SU_2$.}}
Let $\dd\in[0,1]$.\index{$\dd$}
The elements of $\SU_2$ at distance $2\pi\cdot \dd$
from $I$ form a closed algebraic subset and
a conjugacy class, which will be denoted by $\cla_{\dd}$,\index{$\cla_{\dd}$}
and consist of all matrices in $\SU_2$ with eigenvalues $e^{\pm i\pi\cdot\dd}$.
%
In particular, $\cla_0=\{I\}$, $\cla_1=\{-I\}$ and
$\cla_{\dd}$ is diffeomorphic to a $2$-sphere for all $\dd\in(0,1)$.
Given $\bm{\th}=(\th_1,\dots,\th_n)$, we will denote by
$\dd_i$\index{$\dd_i$}
the distance $d(\th_i,\ZZ_o)$
of $\th_i$ from the subset $\ZZ_o\subset\RR$\index{$\ZZ_o$}
of odd integers.
%
%
%
%

{\bf{Algebraic and analytic sets.}} 
In what follows, we will use
the term real/complex {\it{algebraic set}} (resp. {\it{analytic set}}) 
to denote a subset $Z$ of some $\KK^n$ with $\KK=\RR,\CC$
defined by finitely-many equations $f_i=0$ and inequalities $g_j\neq 0$,
with $f_i$ and $g_j$ polynomials (resp. analytic functions).
The tangent space $T_p Z$ to $Z\subset\KK^n$ at a point $p\in Z$ is the 
intersection of the hyperplanes in $\KK^n$ of equations $(df_i)_p=0$.
The point $p$ of $Z$ is {\it{smooth}} if $p$ belongs to a unique irreducible component $Z'$ of $Z$, and $T_p Z'=T_p Z$ has the same dimension as $Z'$.

A {\it{semi-algebraic set}} (resp.~{\it{semi-analytic set}})
is a subset of some $\RR^n$ locally obtained as a finite union
of loci defined by finitely-many equations $f_i=0$
and inequalities $g_j>0$, with $f_i,g_j$ polynomials (resp. real-analytic functions).
See, for instance, \cite{semianalytic}.

A map between semi-algebraic sets (resp.~semi-analytic sets)
is {\it{algebraic}} (resp.~{\it{analytic}}) if locally it is the restriction
of an algebraic (resp.~analytic) map between open subsets of Euclidean spaces.

A point $p$ of a semi-analytic set $Z\subset\RR^n$ is smooth if $Z$ agrees
with an analytic variety $W=\{f_i=0\,|\,i\in I\}$ in a neighbourhood of $p$, and $p$ is a smooth point of $W$.
%
The {\it{smooth locus}} of $Z$ is the subset of all smooth points in $Z$.

An irreducible (semi-)algebraic or (semi-)analytic set has {\it{pure dimension}}
if all of its irreducible components have the same dimension.

A {\it{closed algebraic subset}} of a semi-algebraic set is a subset defined by finitely-many polynomial equations. A {\it{closed analytic subset}} of a semi-analytic set is a subset locally defined by finitely-many analytic equations.
All such sets will be regarded with the {\it{classical topology}}, namely with the topology inherited by the ambient real/complex affine space.
Thus a subset is {\it{dense}} if it is such in the classical topology: for example,
a proper closed analytic subset of an irreducible semi-analytic set has dense complement.


\subsection{Main results}\label{sec:main-results}

\subsubsection{Homomorphism spaces and representation spaces.}\label{ssc:hom}
Here we introduce the representation spaces we are interested in.
Throughout the whole paper, we constantly identify
$\PSU_2$ with $\SO_3(\RR)$ and $\CC\PP^1$ with $\Sph$.

%
%
%

\begin{definition}[Homomorphism spaces]
The {\it{(absolute) homomorphism space}} $\Hom(\grp_{g,n},\SU_2)$\index{$\Hom$}
is the space of homomorphisms $\grp_{g,n}\rar \SU_2$. 
Given an {\it{angle vector}} 
$\bm{\th}=(\th_1,\dots,\th_n)\in\RR^n_{>0}$, 
the {\it{relative homomorphism space}}
$\Hom_{\bm{\th}}(\grp_{g,n},\SU_2)$\index{$\Hom_{\bm{\th}}$}
is the subspace
of $\Hom(\grp_{g,n},\SU_2)$ that consists of
homomorphisms that send $\Bcal_i$ to $\cla_{\dd_i}$
for every $i=1,\dots,n$. 
\end{definition}

Absolute and relative homomorphism spaces will be realized as 
real algebraic subsets of a Euclidean space, from which they will inherit
a classical topology and a Zariski topology (see Section \ref{sec:rep}).

Consider now the action of the group $\PSU_2$ by conjugation on 
$\Hom(\grp_{g,n},\SU_2)$. The {\it{stabilizer}} of a homomorphism
$\rho$ under such action is the subgroup 
$\stab(\rho)=\{g\in\PSU_2\,|\,g\rho g^{-1}=\rho\}$\index{$\stab(\rho)$}
of $\PSU_2$.

\begin{definition}[Representation spaces]
The {\it{(absolute) representation space}} $\Rep(\grp_{g,n},\SU_2)$\index{$\Rep$, $\Rep_{\bm{\th}}$}
is the topological quotient ot $\Hom(\grp_{g,n},\SU_2)$ by $\PSU_2$. 
The {\it{relative representation space}} $\Rep_{\bm{\th}}(\grp_{g,n},\SU_2)$
is the 
locus in it
corresponding to $\Hom_{\bm{\th}}(\grp_{g,n},\SU_2)$.
\end{definition}

The same definition of homomorphism and representation space
can be given replacing $\SU_2$ by $\SO_3(\RR)$.

\begin{remark}[Other names for $\Hom$ and $\Rep$ spaces]
The homomorphism space $\Hom$ was also called
``space of representations'' \cite{culler-shalen}, or
``representation variety'' \cite{lubotzky-magid} \cite{simpson:hodge},
or ``Betti representation space'' \cite{simpson:moduli}.
In \cite{culler-shalen} a point $[\rho]$ of the representation space 
was identified to its ``character'' $\chi_\rho:\grp_{g,n}\rar\CC$,
defined as $\chi_\rho(\gamma):=\mathrm{tr}(\rho(\gamma))$,
and $\Rep$ itself was then 
called ``space of characters''. 
Such $\Rep$ space was also called
``scheme of representations'' \cite{lubotzky-magid},
``Betti moduli space`` \cite{simpson:moduli},
``character variety'' \cite{goldman:geometric}.
\end{remark}

The decision of factoring homomorphism spaces of $\SU_2$-homomorphisms by the conjugacy action
of $\PSU_2$ can sound a bit non-standard, but we adopt it for the following reason.

\begin{remark}[Avoiding ubiquitous stabilizers]
Since $\pm I$ are in the centre of $\SU_2$,
they belong to the stabilizer of every point of $\Hom(\grp_{g,n},\SU_2)$.
In the context of monodromies of spherical metrics, such $\pm I$ do not
correspond to any piece of geometric information; rather, they are a side-effect
of lifting the monodromy from $\SO_3(\RR)\cong \PSU_2$ to $\SU_2$.
Moreover, taking the quotient 
of $\Hom(\grp_{g,n},\SU_2)$ or $\Hom_{\bm{\th}}(\grp_{g,n},\SU_2)$
by $\SU_2$ or by $\PSU_2$
makes no difference from the topological point of view.
For the above reasons, 
it will be more comfortable to avoid such ubiquitous stabilizers, and so taking
the quotient by $\PSU_2$ rather than by $\SU_2$.
\end{remark}

\begin{definition}[Central and coaxial homomorphisms]\label{def:coax}
A homomorphism $\rho$ in $\SU_2$ is {\it{central}}
if its image is contained in the center $Z(\SU_2)=\{\pm I\}$.
The central (resp.~non-central) locus in $\Hom(\grp_{g,n},\SU_2)$
is the locus of central (resp.~non-central) homomorphisms.
A homomorphism is {\it{coaxial}} if its image is contained in a $1$-parameter subgroup of $\SU_2$. The coaxial (resp.~non-coaxial) locus in $\Hom(\grp_{g,n},\SU_2)$
is the locus of coaxial (resp.~non-coaxial) homomorphisms.
We similarly call the corresponding locus in the represention space,
and in their relative versions.
\end{definition}

\begin{notation}
The non-coaxial locus in $\Hom(\grp_{g,n},\SU_2)$
is denoted by  $\Hom^{nc}(\grp_{g,n},\SU_2)$.\index{$\Hom^{nc}$, $\Rep^{nc}$}
Similarly for the represenation space and the relative versions.
\end{notation}

%

We stress that
a homomorphism $\rho$ in $\SU_2$ is
non-coaxial if and only if its $\PSU_2$-stabilizer is finite and,
in this case, such stabilizer is trivial
(see Lemma \ref{lemma:no-auto-undec}).
All of above definitions still make sense and the above statement
still holds true, replacing $\SU_2$ by $\SL_2(\CC)$, and taking
the quotient by the conjugacy action of $\PSL_2(\CC)$, with the following caveat.

The first issue concerns the definition of coaxiality.
Definition \ref{def:coax} 
must be modified for $\SL_2(\CC)$. In fact, a homomorphism
in $\SL_2(\CC)$ may have image contained inside the subgroup
\[
N:=\left\{\pm \left(\begin{array}{cc} 1 & t\\ 0 & 1\end{array}\right)\ :\ t\in\CC\right\},
\]
but not be contained inside a $1$-parameter subgroup of $\SL_2(\CC)$.
In this case, the stabilizer of such homomorphism
would not be discrete, since it would contain the image of $N$ inside $\PSL_2(\CC)$.
Hence, we call a homomorphism in $\SL_2(\CC)$ {\it{coaxial}} if 
the composition with the projection $\SL_2(\CC)\rar\PSL_2(\CC)$
has image contained in a $1$-parameter subgroup of $\PSL_2(\CC)$.

The second issue concerns relative homomorphism spaces.
We will say that $\rho\in\Hom_{\bm{\th}}(\grp_{g,n},\SL_2(\CC))$
if $B_i$ is conjugate to an element of $\cla_{\dd_i}$.

The third issue concerns quotients.
Since $\SU_2$ is compact, all orbits in $\Hom(\grp_{g,n},\SU_2)$ are closed and
the topological quotient $\Rep(\grp_{g,n},\SU_2)$ is Hausdorff.
On the other hand,
$\SL_2(\CC)$ is not compact and its orbits in $\Hom(\grp_{g,n},\SL_2(\CC))$ are not necessarily closed and so its topological quotient will not necessarily be Hausdorff: thus the symbol $\Rep(\grp_{g,n},\SL_2(\CC))$ will be used for the largest
Hausdorff quotient. Such space is also called
{\it{categorical quotient}} (as in \cite{brion}), or {\it{affine quotient}} (as in \cite{mukai})
as it will naturally be complex algebraic affine. Its points are in bijective correspondence with the closed $\SL_2(\CC)$-orbits.

In fact, it turns out that $\Rep(\grp_{g,n},\SL_2(\CC))$ is complex algebraic (see Theorem 1.1 of \cite{mumford:git}). 
As explained in Remark \ref{rmk:quotients}, on the other hand, $\Rep(\grp_{g,n},\SU_2)$ is semi-algebraic and $\Rep_{\bm{\th}}(\grp_{g,n},\SL_2(\CC))$ 
is a closed algebraic subset of $\Rep(\grp_{g,n},\SL_2(\CC))$.

\begin{remark}[Complex quotients and real quotients]\label{rmk:quotients}
Let $V$ be a vector space over $\KK=\RR,\CC$ with an algebraic
action of an algebraic reductive group $G$
over $\KK$. 
Let us remind the standard technique for realizing the quotient of $V$ by $G$.
First consider the finitely generated algebra $\KK[V]^G$ 
of $G$-invariant polynomial functions on $V$.
Then pick a finite set of generators $p_1,\dots,p_\ell$ for the $\KK$-algebra $\KK[V]^G$ and consider the algebraic map $p=(p_1,\dots,p_\ell):V\rar \KK^\ell$.
The image of $p$ is contained inside the zero locus $V(I)$
of the kernel ideal $I$ of the surjection $\KK[y_1,\dots,y_\ell]\rar \KK[V]^G$ that sends $y_j$ to $p_j$.
Moreover, the largest Hausdorff quotient $V/\!/G$ is homeomorphic to $p(V)$.
For $\KK=\CC$ we have $p(V)=V(I)$ and so $V/\!/G$ 
can be realized as a complex algebraic subset of $\CC^\ell$
(see Theorem 1.24 of \cite{brion}, or Section 5.1 of \cite{mukai}).
For $\KK=\RR$ the locus $p(V)$ is a semi-algebraic subset
contained inside $V(I)$ but need not be the whole $V(I)$; if $G$ is compact,
then the topological quotient $V/G$ is already Hausdorff
and so $p$ realizes a homeomorphism between
$V/G$ and the semi-algebraic set $p(V)$ of $\RR^\ell$
(see, for example, \cite{procesi-schwarz}).

Suppose now that $X$ is a $G$-invariant algebraic subset of $V$.
The largest Hausdorff quotient $X/\!/G$ can be identified to $p(X)$: this is
a complex algebraic subset of $\CC^\ell$ for $\KK=\CC$,
or a semi-algebraic subset of $\RR^\ell$ for $\KK=\RR$. If 
$\KK=\RR$ and $G$ is compact,
then the topological quotient $X/G$ is Hausdorff and $p$ induces
a homeomorphism between $X/G$ and the semi-algebraic subset $p(X)$.
\end{remark}

\begin{example}[Simplest quotients by compact real groups]
Consider the rotation action of $\SO_2(\RR)$ on $\RR^2$. 
If $s,t$ are the standard coordinates on $\RR^2$, then the algebra of invariant polynomials
is generated by $p(s,t)=s^2+t^2$ and the image of $p:\RR^2\rar\RR$ consists of the positive half-line $\RR_{\geq 0}$, which is a semi-algebraic subset of $\RR$ and is identified to the quotient $\RR^2/\SO_2(\RR)$.
A similar example is obtained by acting on $\RR$ via the multiplication by $\{\pm 1\}$: if $t$ is the standard coordinate in $\RR$, then the algebra of invariant polynomials
is generated by $p(t)=t^2$ and $p:\RR\rar\RR$ has the positive half-line as image. Hence again $\RR/\{\pm 1\}=\RR_{\geq 0}$.
%
\end{example}

We will make use of the following two observations.

\begin{remark}[Quotient of a manifold by a compact group]\label{rmk:orbi}
The quotient of a manifold by a compact group acting with trivial stabilizers is 
a manifold
Moreover, if the manifold is oriented and the group is connected, then the quotient manifold is oriented too. 
\end{remark}

\begin{remark}[Quotient of a variety by a compact group]\label{rmk:X/G}
Consider the quotient of a real algebraic (or real analytic)
variety $X$ by a compact group $G$ that acts with connected stabilizers.
Since all orbits are closed, the tangent space $T_{[x]}(X/G)$ can be identied
to the normal space $N_{G\cdot x/X}$ to the orbit $G\cdot x$ at the point $x$ (see \cite{janich} for more details).
Since $G\cdot x$ is locally isomorphic to $G/\mathrm{stab}(x)$,
we obtain
\[
\mathrm{dim}(T_{[x]}(X/G))=
\mathrm{dim}(T_x X)-\mathrm{dim}(G)+\mathrm{dim}(\mathrm{stab}(x)).
\]
Since $X$ is reduced and irreducible, so is $X/G$. Moreover both spaces have
a Zariski-open dense subset of smooth points.
It follows that the smooth locus of $X/G$ consists of points $[x]$
at which $\mathrm{dim}(T_x X)+\mathrm{dim}(\mathrm{stab}(x))$ achieves it minimum.
\end{remark}

We will see that homomorphism spaces 
$\Hom(\grp_{g,n},\SU_2)$ and $\Hom_{\bm{\th}}(\grp_{g,n},\SU_2)$
are real algebraic,
since they are defined by 
certain real algebraic equations on $\SU_2^{2g+n}$.
Thus their dimension as real algebraic sets is well-defined.
Moreover it is possible to compute
their Zariski tangent space, namely the intersection of the kernels
of the differentials of the above-mentioned equations.
As a consequence, one can detect their smoothness through the implicit function theorem
(see, for instance, \cite[Theorem 4.12]{lee:book}).
Furthermore, the associated representation spaces 
$\Rep(\grp_{g,n},\SU_2)$ and $\Rep_{\bm{\th}}(\grp_{g,n},\SU_2)$
are semi-algebraic
and the conjugacy action of $\SU_2$ on non-coaxial homomorphisms 
has trivial stabilizer (see Lemma \ref{lemma:no-auto-undec}).

%
%
%
%
%
%
%
%
%

Since certain representation spaces
$\Rep_{\bm{\th}}(\grp_{g,n},\SU_2)$ entirely consist of coaxial representations 
(in other words, classes of homomorphisms
with infinite stabilizer), such spaces have special features.
This phenomenon 
only occurs for the following class of triples $(g,n,\bm{\th})$.

\begin{definition}[Special triples]\label{def:special}
Consider triples $(g,n,\bm{\th})$ with $g\in\ZZ_{\geq 0}$, $n\in\ZZ_{>0}$
and $\bm{\th}\in\RR^n_{>0}$.
A triple $(g,n,\bm{\th})$ is {\it{special}} 
if 
\begin{itemize}
\item[(i)]
$g=1$ and $\bm{\th}\in\ZZ^n$ with $\sum_{i=1}^n (\th_i-1)\in 2\ZZ$,
\item[(ii)]
$g=0$ and $d_1(\bm{\th}-\bm{1},\ZZ_o^n)\leq 1$,
\end{itemize}
where $\bm{1}=(1,1,\dots,1)$\index{$\bm{1}$} 
and $d_1$\index{$d_1$}
is the distance in $\RR^n$
induced by the norm $\|\bm{p}\|_1:=\sum_{i=1}^n |p_i|$,
and $\ZZ^n_o$\index{$\ZZ^n_o$} 
is the subset of $\bm{m}\in\ZZ^n$ such that
$\|\bm{m}\|_1$ is odd.
\end{definition}

Now we are ready to state a few results about representation spaces
that will guide our investigation of decorated representation
spaces, which will be defined in Section \ref{sec:intro-dec}
where we will also present our main results.

In this first lemma we investigate the semi-algebraic structure of
representation spaces.

\begin{mainlemma}{I}[Semi-algebraicity of representation spaces]\label{mainlemma:alg}
The following properties hold.
\begin{itemize}
\item[(o)]
$\Hom(\grp_{g,n},\SU_2)$ is a real algebraic set and $\Hom(\grp_{g,n},\SL_2(\CC))$
is a complex algebraic set.
Moreover,
$\Rep(\grp_{g,n},\SU_2)$ is a semi-algebraic set and
$\Rep(\grp_{g,n},\SL_2(\CC))$ is a complex algebraic set.
\item[(i)]
The maps $\Hom(\grp_{g,n},\SU_2)\rar\Hom(\grp_{g,n},\SO_3(\RR))$
and
$\Rep(\grp_{g,n},\SU_2)\rar\Rep(\grp_{g,n},\SO_3(\RR))$ are real algebraic, local homeomorphisms.
\item[(ii)]
The coaxial loci in $\Hom(\grp_{g,n},\SU_2)$ 
and $\Hom(\grp_{g,n},\SL_2(\CC))$ are closed algebraic
subsets.
\item[(iii)]
$\Hom_{\bm{\th}}(\grp_{g,n},\SU_2)\subset\Hom(\grp_{g,n},\SU_2)$
and $\Hom_{\bm{\th}}(\grp_{g,n},\SL_2(\CC))\subset\Hom(\grp_{g,n},\SL_2(\CC))$
are closed algebraic subsets. 
\item[(iv)]
The coaxial locus is closed algebraic inside $\Hom_{\bm{\th}}(\grp_{g,n},\SU_2)$
and inside $\Hom_{\bm{\th}}(\grp_{g,n},\SL_2(\CC))$.
\end{itemize}
Statements (ii-iii-iv) also hold if $\Hom$ spaces are replaced by the corresponding $\Rep$ spaces.
\end{mainlemma}

We remark that the above claims are also true
if we replace $(\SU_2,\SL_2(\CC))$ by $(\SO_3(\RR),\PSL_2(\CC))$.

We stress that 
most of Lemma \ref{mainlemma:alg} is already well-known.
In particular, the semi-algebraic nature of real representation spaces
is already mentioned in \cite[page 67]{johnson-millson} 
and extensively discussed in \cite{huebschmann}.
An explicit example of such non-algebraic (though semi-algebraic) $\SU_2$-representation space
is described in Example \ref{example:non-algebraic}.


In the following theorem we investigate absolute representation spaces.

\begin{mainthm}{II}[Absolute representation spaces]\label{mainthm:rep-undec}
Fix $g\geq 0$ and $n>0$ with $2g-2+n>0$.
\begin{itemize}
\item[(i)]
Non-coaxial homomorphisms in $\SU_2$
have trivial stabilizer.
\item[(ii)]
The space $\Hom(\grp_{g,n},\SU_2)$ is isomorphic to $\SU_2^{2g+n-1}$.
Inside it,
\begin{itemize}
\item[(ii-a)]
the coaxial locus is irreducible, of dimension $2g+n+1$;
\item[(ii-b)]
the non-coaxial locus is an oriented manifold of dimension $6g-3+3n$.
\end{itemize}
As a consequence, inside $\Rep(\grp_{g,n},\SU_2)$
\begin{itemize}
\item[(ii-c)]
the coaxial locus is irreducible, of dimension $2g+n-1$;
\item[(ii-d)]
the non-coaxial locus is an oriented manifold of {\emph{real}} dimension $6g-6+3n$.
\end{itemize}
\item[(iii)]
The coaxial locus in $\Hom(\grp_{g,n},\SU_2)$ is connected,
the non-coaxial locus is dense and connected.
The same holds in $\Rep(\grp_{g,n},\SU_2)$.
\end{itemize}
In the $\SL_2(\CC)$ case, claims (i-ii-iii) hold replacing 
$\SU_2$ by $\SL_2(\CC)$ and ``real dimension'' by ``complex dimension''.
\end{mainthm}

In the following result we consider properties of relative representation spaces.

\begin{mainthm}{III}[Relative representation spaces]\label{mainthm:rep-undec-rel}
Fix $g\geq 0$ and $\bm{\th}=(\th_1,\dots,\th_n)$ and let
$k$ be the number of integral entries of $\bm{\th}$.

Then inside $\Hom_{\bm{\th}}(\grp_{g,n},\SU_2)$
\begin{itemize}
\item[(i)]
the coaxial locus has pure dimension $2g+2$, and it is non-empty if and only if 
$\sum_i(\pm(\th_i-1))\in 2\ZZ$
for some choice of the signs;
\item[(ii)]
the non-coaxial locus is an oriented manifold of {\emph{real}} dimension $6g-3+2(n-k)$.
\end{itemize}
Hence, inside $\Rep_{\bm{\th}}(\grp_{g,n},\SU_2)$
\begin{itemize}
\item[(iii)]
the coaxial locus has pure dimension $2g$;
\item[(iv)]
the non-coaxial locus is an oriented manifold of {\emph{real}} dimension $6g-6+2(n-k)$.
\end{itemize}
The above claims hold replacing 
$\SU_2$ by $\SL_2(\CC)$ and ``real dimension'' by ``complex dimension''.
\end{mainthm}

Theorem \ref{mainthm:rep-undec-rel}(ii)
was already obtained by \cite{igusa:smoothness} and \cite{narasimhan-seshadri64}
for $n=0$, and in \cite{mehta-seshadri} for $n>0$.
For general $\bm{\th}$
smoothness of the whole $\Rep_{\bm{\th}}(\grp_{g,n},\SU_2)$
was also obtained in 
\cite[Proposition 3.3.1]{audin:lectures} and in
\cite[Corollary 9.9]{boalch:braiding}.
In this paper the proof of Theorem \ref{mainthm:rep-undec-rel}(ii) relies on
a computation, whose case $n=0$ was already performed
in \cite{goldman:symplectic}, and whose
case $n=1$ was explicitly treated in \cite{jeffrey:extended-mathann94}.

Now we separately treat non-special and special cases.

\begin{mainthm}{IV}[Non-special relative representation spaces]\label{mainthm:rep-undec-nonsp}
Assume that $(g,n,\bm{\th})$ is non-special
and let $k$ be the number of integral entries of $\bm{\th}=(\th_1,\dots,\th_n)$.
Then 
\begin{itemize}
\item[(i)]
$\Hom_{\bm{\th}}(\grp_{g,n},\SU_2)$ has pure dimension $6g-3+2(n-k)$
and $\Rep_{\bm{\th}}(\grp_{g,n},\SU_2)$ has pure dimension $6g-6+2(n-k)$.
\end{itemize}
Moreover, inside $\Hom_{\bm{\th}}(\grp_{g,n},\SU_2)$ and $\Rep_{\bm{\th}}(\grp_{g,n},\SU_2)$
\begin{itemize}
\item[(ii)]
their non-coaxial locus coincides with the smooth locus;
\item[(iii)]
their non-coaxial locus is non-empty and dense;
\item[(iv)]
their non-coaxial locus is connected.
\end{itemize}
Claims (i) and (iii) still hold replacing 
$\SU_2$ by $\SL_2(\CC)$ and ``real dimension'' by ``complex dimension''.
\end{mainthm}

The situation in special cases is very explicit.

\begin{mainthm}{V}[Special relative representation spaces]\label{mainthm:special-undec}
Assume that $(g,n,\bm{\th})$ is special and let $k$ be the number of integral entries of $\bm{\th}$. Then all representations are coaxial and the following hold.
\begin{itemize}
\item[(i)]
Let $g=1$.
Then $\Rep_{\bm{\th}}(\grp_{1,n},\SU_2)$ is homeomorphic to the $2$-sphere: four points on such sphere correspond to central representations, all the other points correspond to coaxial non-central representations.
\item[(ii)]
Let $g=0$. If $k=n-1$ or $d_1(\bm{\th}-\bm{1},\ZZ^n_o)<1$,
then $\Rep_{\bm{\th}}(\grp_{0,n},\SU_2)$ is empty.
\item[(iii)]
Assume $g=0$, $k\neq n-1$ and $d_1(\bm{\th}-\bm{1},\ZZ^n_o)=1$.
\begin{itemize}
\item[(iii-a)]
If $k=n$, then $\Hom_{\bm{\th}}(\grp_{0,n},\SU_2)$ is isomorphic to 
a point.
\item[(iii-b)]
If $k\leq n-2$, then
$\Hom_{\bm{\th}}(\grp_{0,n},\SU_2)$ is isomorphic to $\Sph$.
\item[(iii-c)]
$\Hom_{\bm{\th}}(\grp_{0,n},\SU_2)$ consists
of one conjugacy class, and
the natural
structure of algebraic scheme on $\Hom_{\bm{\th}}(\grp_{0,n},\SU_2)$
is reduced if and only if $k=n$ or $k=n-2$.
\end{itemize}
\end{itemize}
\end{mainthm}

Now we turn our attention to symplectic structures.
Non-coaxial loci of representation spaces in $\SU_2$ or in $\SL_2(\CC)$
support Goldman's Poisson structure,
whose definition will be recalled in Section \ref{sec:symplectic}.

\begin{mainprop}{VI}[Symplectic structure on relative representation spaces]\label{mainprop:sympl}
Let $(g,n,\bm{\th})$ be non-special.
\begin{itemize}
\item[(i)]
The natural map $\Rep^{nc}_{\bm{\th}}(\grp_{g,n},\SU_2)\rar \Rep^{nc}_{\bm{\th}}(\grp_{g,n},\SL_2(\CC))$
is an embedding of real-analytic manifolds.
\item[(ii)]
Goldman's Poisson structures on $\Rep^{nc}_{\bm{\th}}(\grp_{g,n},\SU_2)$ 
and on $\Rep^{nc}_{\bm{\th}}(\grp_{g,n},\SL_2(\CC))$
are respectively real symplectic and complex symplectic.
Moreover, such symplectic structures are compatible with the embedding in (i).
\end{itemize}
\end{mainprop}


The above Proposition \ref{mainprop:sympl} is not original:
see \cite{atiyah-bott} \cite{goldman:symplectic} \cite{karshon}
for the case of closed surfaces,
\cite{witten:gauge-cmp91} and \cite{jeffrey:extended-mathann94} for the case $n=1$, and
\cite{fock-rosly} 
\cite{guruprasad} \cite{alekseev-malkin:symplectic}
and \cite{mondello:poisson} for the case $n\geq 1$.

\subsubsection{Decorated representation spaces.}\label{sec:intro-dec}
View $\SU_2$ as acting on $\CC\PP^1\cong\Sph$.
If no $\th_i$ is integer, then a representation $[\rho]\in \Rep_{\bm{\th}}(\grp_{g,n},\SU_2)$
sends every $\beta\in\Bcal$ to a non-central element, 
which thus determines a unique axis of rotation in $\Sph$.
On the other hand, if some $\th_i$ is integer, then $[\rho]$
sends every $\beta'_i\in\Bcal_i$ to $\pm I$ and so
$\rho(\beta'_i)$ does not determine an axis of rotation. 
In order to add the extra piece of information
consisting of axes of rotation of $\rho(\beta)$ for all $\beta\in\Bcal$, we introduce decorations.

\begin{remark}
Decorated representations can be seen as holonomies of flat connections
with regular singularities at the punctures, the decoration corresponding to the residues of the connection. 
Spaces of flat connections with regular singularities were considered in
\cite{jeffrey:extended-mathann94},
\cite{guruprasad}, \cite{audin:lectures},
but also in \cite{FoGo}.
Flat connections with irregular singularities were considered
already in \cite{JMU} and \cite{boalch:isomonodromic}.
\end{remark}

\begin{definition}[Decorated homomorphisms]\label{def:dec-hom}
Let $\rho:\grp_{g,n}\rar \SU_2$ be a homomorphism. A {\it{decoration}} for $\rho$
is a map $\Axis:\Bcal\rar\su_2\setminus\{0\}$\index{$\Axis$}
such that
\begin{itemize}
\item[(a)]
$\Axis(\gamma\beta\gamma^{-1})=\rho(\gamma)\Axis(\beta)\rho(\gamma)^{-1}$
for all $\beta\in\Bcal$ and $\gamma\in\grp_{g,n}$
\item[(b)]
$\rho(\beta)=\ee(\Axis(\beta))$ for all $\beta\in\Bcal$.
\end{itemize}
We call such couple $(\rho,\Axis)$ a {\it{decorated homomorphism}}.
\end{definition}

Note that, by condition (a), a decoration is determined by its values
at $\beta_1,\dots,\beta_n$.

If $(\rho,\Axis)$ is a decorated homomorphism, then the associated
{\it{normalized decoration}} is
$\hat{\Axis}(\beta):=\frac{\Axis(\beta)}{\|\Axis(\beta)\|}$\index{$\hat{\Axis}$}
can be thought of as a map $\hat{\Axis}:\Bcal\rar\Sph$.
Note that $\hat{\Axis}(\beta)$ is a point of $\Sph$ fixed by $\rho(\beta)$.
Thus, using the decoration $\Axis$, we can think of $\rho(\beta)$ as a rotation
of angle $2\pi\|\Axis(\beta)\|$ with centre $\hat{\Axis}(\beta)$.

We will see in Section \ref{sec:intro-mon} that the monodromy of a spherical surface
$(S,\bm{x},h)$ can be naturally endowed with a decoration, which can be essentially identified to the restriction of the developing map to the boundary points
of the completion of the universal cover of $\dot{S}$. Condition (a) for decorated monodromies is then a consequence of the equivariance 
of the developing map.

Now we introduce the spaces of decorated representations that we want to study,
whose topology will be described in Section \ref{sec:decorated}.

\begin{definition}[Space of decorated representations]
The space of {\it{decorated homomorphisms}} 
$\wh{\Hom}(\grp_{g,n},\SU_2)$\index{$\wh{\Hom}$, $\wh{\Hom}_{\bm{\th}}$}
is the space of couples
$(\rho,\Axis)$, where $\rho:\grp_{g,n}\rar\SU_2$ is a homomorphism
and $\Axis$ is a decoration for $\rho$.
The subset of $(\rho,\Axis)$ in $\wh{\Hom}(\grp_{g,n},\SU_2)$ such that
$\|\Axis(\beta_i)\|=\th_i$ for all $i=1,\dots,n$ is denoted by $\wh{\Hom}_{\bm{\th}}(\grp_{g,n},\SU_2)$.
The {\it{decorated representation space}} $\wh{\Rep}(\grp_{g,n},\SU_2)$
\index{$\wh{\Rep}$, $\wh{\Rep}_{\bm{\th}}$}
is the space of 
$\PSU_2$-conjugacy classes of decorated homomorphisms. We denote by
$\wh{\Rep}_{\bm{\th}}(\grp_{g,n},\SU_2)$ the locus of pairs $[\rho,\Axis]$ such that
$\|\Axis(\beta_i)\|=\th_i$ for all $i=1,\dots,n$.
\end{definition}


We say that a decorated homomorphism $(\rho,\Axis)$ is {\it{non-coaxial}} if $\rho$ is.
The corresponding loci in the spaces of decorated 
homomorphisms (resp. representations)
are denoted by 
$\wh{\Hom}^{nc}(\grp_{g,n},\SU_2)$ and $\wh{\Hom}^{nc}_{\bm{\th}}(\grp_{g,n},\SU_2)$.
(resp. by
$\wh{\Rep}^{nc}(\grp_{g,n},\SU_2)$ and $\wh{\Rep}^{nc}_{\bm{\th}}(\grp_{g,n},\SU_2)$).
\index{$\wh{\Hom}^{nc}$, $\wh{\Rep}^{nc}$}
However,
for decorated representations the notion of being ``non-elementary''
will be also important for us: monodromies of spherical metrics
will always have such property, as shown in Theorem \ref{mainthm:mon-sph}(i) below.

\begin{definition}[Elementary decorated homomorphisms]
A decorated homomorphism $(\rho,\Axis)$ is {\it{elementary}} if there exists a 
$1$-dimensional subalgebra $\hfrak$ of $\su_2$
that contains the image of $\Axis$ and such that 
the $1$-parameter subgroup $\exp(\hfrak)\subset\SU_2$
contains the image of $\rho$.
The elementary (resp.~non-elementary) locus in $\wh{\Hom}(\grp_{g,n},\SU_2)$
is the locus of elementary (resp.~non-elementary) homomorphisms.
We similarly call the corresponding locus in the decorated represention space,
and in their relative versions.
\end{definition}

\begin{notation}
The non-elementary locus in $\wh{\Hom}(\grp_{g,n},\SU_2)$
is denoted by $\wh{\Hom}^{ne}(\grp_{g,n},\SU_2)$.\index{$\wh{\Hom}^{ne}$}
Similarly for the decorated representation space and the relative versions.
\end{notation}

Note that an elementary decorated homomorphism $(\rho,\Axis)\in\wh{\Hom}_{\bm{\th}}(\grp_{g,n},\SU_2)$
is coaxial. The converse is not true in general, though it is true if no $\th_i$ is integral,
as shown in Lemma \ref{mainlemma:anal}(a) below.

In Theorem \ref{mainthm:rep}(ii-b) we will show that the singular locus of $\wh{\Hom}(\grp_{g,n},\SU_2)$ consists of the subset $\wti{\Sigma}$
that we here introduce.
We will also see in Theorem \ref{mainthm:mon-sph}(iii) that spherical surfaces 
with decorated
monodromy in $\wti{\Sigma}$ are of a very simple type (see Definition \ref{def:purely-hem}).

\begin{definition}[The $\Sigma$ locus]\label{def:sigma}
The subset $\wti{\Sigma}$\index{$\wti{\Sigma}$, $\Sigma$}
of $\wh{\Hom}(\grp_{g,n},\SU_2)$
consists of classes $(\rho,\Axis)$ of decorated homomorphisms such that
\begin{itemize}
\item[(a)]
$\rho$ is coaxial
\item[(b)]
$\|\Axis(\beta)\|\in \ZZ$ for all $\beta\in\Bcal$
\item[(c)]
some two-dimensional subspace of $\su_2$ preserved by $\IM(\rho)$
contains the image of $\Axis$.
\end{itemize}
We denote by $\Sigma$ the image of such locus inside $\wh{\Rep}(\grp_{g,n},\SU_2)$.
\end{definition}

Let us comment on the above definition. Condition (b) implies that $\wti{\Sigma}$ 
is contained inside the union of all
$\wh{\Hom}_{\bm{\th}}(\grp_{g,n},\SU_2)$ with $\bm{\th}\in\ZZ^n$.
To decypher condition (c),
we note that decorated homomorphisms occurring in $\wti{\Sigma}$ can be central 
or non-central. This permits us to view $\wti{\Sigma}$ as
the union of two more understandable disjoint subloci $\wti{\Sigma}_0$ and $\wti{\Sigma}_1$,\index{$\wti{\Sigma}_0$, $\wti{\Sigma}_1$}
defined as follows:
\begin{itemize}
\item
$\wti{\Sigma}_0$ is the locus of $(\rho,\Axis)$ with {\it{central}} $\rho$
such that the image of $\Axis$ does not span $\su_2$,
and 
\item
$\wti{\Sigma}_1$ is the locus of $(\rho,\Axis)$ with {\it{non-central}} $\rho$
such that
the homomorphism $\rho$ takes values in a $1$-parameter subgroup
of $H$ with Lie algebra $\hfrak$ and the decoration $\Axis$
satisfies
$\Axis(\beta)\in \hfrak^\perp$ and $\|\Axis(\beta)\|\in \ZZ$ for all $\beta\in\Bcal$.
\end{itemize}
We denote by $\Sigma_0$ and $\Sigma_1$ the corresponding loci
in $\wh{\Rep}(\grp_{g,n},\SU_2)$.

%

The inclusions between the above defined loci inside the space of decorated representations
are summarized in the following diagram.
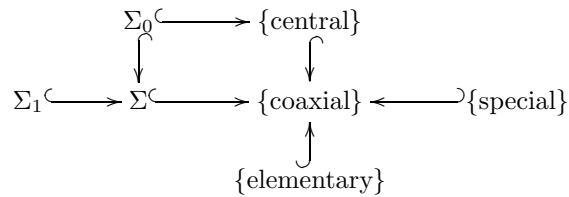
\[
\xymatrix@R=0.5cm{
& \Sigma_0 \ar@{^(->}[r]  \ar@{^(->}[d] & \{\text{central}\}\ar@{^(->}[d]\\
\Sigma_1  \ar@{^(->}[r] & \Sigma  \ar@{^(->}[r] & \{\text{coaxial}\} & \{\text{special}\} \ar@{_(->}[l]\\
&& \{\text{elementary}\}\ar@{^(->}[u]
}
\]

\begin{remark}
The above definitions of decorated homomorphism and representation
still make sense when replacing $\SU_2$ by $\SO_3(\RR)$.
\end{remark}

In Lemma \ref{mainlemma:anal}, Theorem \ref{mainthm:rep},
Theorem \ref{mainthm:rep-nonsp} and
Theorem \ref{mainthm:special} below we describe the local structure of the decorated representation spaces in $\SU_2$.
We show that the spaces of decorated homomorphisms are real-analytic
and decorated representation spaces are semi-analytic.
Since the conjugacy action of $\PSU_2$ on non-elementary decorated homomorphisms 
is proper (Corollary \ref{cor:proper}) with trivial stabilizer
(Lemma \ref{lemma:no-auto}),
we determine the locus of decorated representation spaces
that admits a natural manifold structure, separately treating
the non-special and the special case.\\

%
%


\begin{mainlemma}{$\hat{\text{I}}$}[Semi-analyticity of decorated representation spaces]\label{mainlemma:anal}
Let $g\geq 0$ and $n>0$.
%
\begin{itemize}
\item[(o)]
$\wh{\Hom}(\grp_{g,n},\SU_2)$ is a real-analytic set and
$\wh{\Rep}(\grp_{g,n},\SU_2)$ is a semi-analytic set.
\item[(i)]
The natural map
$\wh{\Hom}(\grp_{g,n},\SU_2)\rar\wh{\Hom}(\grp_{g,n},\SO_3(\RR))$ 
is a real-analytic local homeomorphism.
\item[(ii)]
Inside $\wh{\Hom}(\grp_{g,n},\SU_2)$
the elementary, coaxial and $\wti{\Sigma}$ loci
are closed analytic.
If $g=0$, then $\wti{\Sigma}_1$ is empty; if $g\geq 1$,
then $\wti{\Sigma}_1$ is open and dense inside $\wti{\Sigma}$.
\item[(iii)]
The subset
$\wh{\Hom}_{\bm{\th}}(\grp_{g,n},\SU_2)$ 
of $\wh{\Hom}(\grp_{g,n},\SU_2)$
is closed analytic.
\item[(iv)]
The coaxial and the elementary loci in $\wh{\Hom}_{\bm{\th}}(\grp_{g,n},\SU_2)$ are closed analytic.
\end{itemize}

\medskip

Moreover, if no $\th_i$ is integral, then
\begin{itemize}
\item[(a)]
a decorated representation
is elementary if and only if it is coaxial;
\item[(b)]
$\wh{\Hom}^{nc}_{\bm{\th}}(\grp_{g,n},\SU_2)\rar
\Hom^{nc}_{\bm{\th}}(\grp_{g,n},\SU_2)$
is a real-analytic isomorphism.
\end{itemize}
Statements analogous to (i-ii-iii-iv) and (b) hold for the corresponding
decorated representation spaces.
\end{mainlemma}

%
%
%
%
%
%
%
Note that $\wh{\Rep}_{\bm{\th}}(\grp_{g,n},\SU_2)$
and $\wh{\Rep}(\grp_{g,n},\SU_2)$ can indeed be not real analytic,
but only semi-analytic (see Example \ref{example:non-analytic}).

The following result mirrors the undecorated
case treated in Theorem \ref{mainthm:rep-undec}.

\begin{mainthm}{$\widehat{\text{II}}$}[Absolute decorated representation spaces]\label{mainthm:rep}
Fix $g$ and $n$ such that $2g-2+n>0$.
\begin{itemize}
\item[(i)]
Non-elementary decorated homomorphisms in $\SU_2$
have trivial stabilizers.
\item[(ii)]
Inside $\wh{\Hom}(\grp_{g,n},\SU_2)$ 
\begin{itemize}
\item[(ii-a)]
the elementary locus has pure dimension $2g+n+1$;
\item[(ii-b)]
the singular locus coincides with $\wti{\Sigma}$
and has pure dimension $2g+2n+2$;
\item[(ii-c)]
the components of the coaxial locus have dimensions in 
$[n+1,2n]$ if $g=0$,
or in $[2g+n+1,2g+2n-1]\cup\{2g+2n+2\}$ if $g\geq 1$;
\item[(ii-d)]
$\wh{\Hom}^{ne}(\grp_{g,n},\SU_2)\setminus\wti{\Sigma}$
is an oriented manifold of dimension $6g-3+3n$.
\end{itemize}
As a consequence, inside $\wh{\Rep}(\grp_{g,n},\SU_2)$
\begin{itemize}
\item[(ii-e)]
the elementary locus has pure dimension $2g+n-1$;
\item[(ii-f)]
the non-elementary singular locus $\Sigma^{ne}$ has pure dimension $2g+2n-1$;
\item[(ii-g)]
$\wh{\Rep}^{ne}(\grp_{g,n},\SU_2)\setminus\Sigma$
is an oriented manifold of dimension $6g-6+3n$.
\end{itemize}
\item[(iii)]
Inside $\wh{\Hom}(\grp_{g,n},\SU_2)$,
\begin{itemize}
\item[(iii-a)]
the non-elementary and the non-coaxial loci are dense,
and so $\wh{\Hom}(\grp_{g,n},\SU_2)$ has pure dimension $6g-3+3n$;
\item[(iii-b)]
the non-elementary locus is connected, and
non-coaxial locus is connected if and only if $(g,n)\neq (0,3),(1,1)$.
\end{itemize}
Claims analogous to (iii-a) and (iii-b) hold for $\wh{\Rep}(\grp_{g,n},\SU_2)$, which has pure dimension $6g-6+3n$.
\end{itemize}
\end{mainthm}

Now we consider relative decorated spaces.
As in Theorem \ref{mainthm:rep-undec-rel},
a major role in the below result will be played by the computation
of Zariski tangent spaces.

\begin{mainthm}{$\widehat{\text{III}}$}[Relative decorated representation spaces]\label{mainthm:rep-rel}
Let $\bm{\th}=(\th_1,\dots,\th_n)$, and $k$ be the number of integer entries of $\bm{\th}$.
Inside
$\wh{\Hom}_{\bm{\th}}(\grp_{g,n},\SU_2)$ 
\begin{itemize}
\item[(i)]
the elementary and the coaxial loci are non-empty
if and only if $\sum_i(\pm(\th_i-1))\in 2\ZZ$ for some choice of the signs:
in this case, the elementary locus is connected, irreducible, of dimension $2g+2$, and
the coaxial locus is connected, irreducible, of dimension $2n$ if $(g,n)=(0,k)$, and $2g+2k+2$ otherwise;
\item[(ii)]
the non-coaxial locus is an oriented manifold of dimension $6g-3+2n$.
\end{itemize}
Inside $\wh{\Rep}_{\bm{\th}}(\grp_{g,n},\SU_2)$
\begin{itemize}
\item[(iii)]
the elementary locus is connected, irreducible, of dimension $2g$; 
\item[(iv)]
the coaxial locus is connected, irreducible, of dimension
$2n-3$ if $(g,n)=(0,k)$,
$2g$ if $k=0$, and $2g+2k-1$ otherwise; 
\item[(v)]
the non-coaxial locus is an oriented manifold of dimension $6g-6+2n$.
\end{itemize}
Moreover, inside $\wh{\Hom}_{\bm{\th}}(\grp_{g,n},\SU_2)$ and
$\wh{\Rep}_{\bm{\th}}(\grp_{g,n},\SU_2)$
\begin{itemize}
\item[(vi)]
the non-elementary locus is dense and connected.
\end{itemize}
\end{mainthm}

Note that the above Theorem \ref{mainthm:rep-rel}(ii) 
deals with the singularities of 
the relative homomorphisms space, whereas 
Theorem \ref{mainthm:rep}(ii)
deals with the absolute homomorphisms space.
Thus there might be singular
points in $\wh{\Hom}^{ne}_{\bm{\th}}(\grp_{g,n},\SU_2)$
that are smooth in $\wh{\Hom}^{ne}(\grp_{g,n},\SU_2)$.
When this is the case, the same happens for the corresponding
points in the representation spaces.\\

In the following result we deal with the relative spaces in the non-special case.

\begin{mainthm}{$\widehat{\text{IV}}$}[Non-special relative decorated representation spaces]\label{mainthm:rep-nonsp}
Let $(g,n,\bm{\th})$ be non-special and
let $k$ be the number of integer entries of $\bm{\th}$. 
Then 
\begin{itemize}
\item[(i)]
$\wh{\Hom}_{\bm{\th}}(\grp_{g,n},\SU_2)$ has pure dimension $6g-3+2n$
and $\wh{\Rep}_{\bm{\th}}(\grp_{g,n},\SU_2)$
has pure dimension $6g-6+2n$.
\end{itemize}
Moreover, inside
$\wh{\Hom}_{\bm{\th}}(\grp_{g,n},\SU_2)$ and $\wh{\Rep}_{\bm{\th}}(\grp_{g,n},\SU_2)$
\begin{itemize}
\item[(ii)]
their non-coaxial locus coincides with the smooth locus;
\item[(iii)]
their non-coaxial locus is non-empty and dense;
\item[(iv)]
their non-coaxial locus is connected.
\end{itemize}
%
%
%
%
%
\end{mainthm}

In special decorated $\SU_2$-representation spaces non-coaxial loci are empty.
Nevertheless, it is still possible to define natural
manifold structures on certain open subsets.

\begin{notation}[Coaxial subsets]\index{$\COAX$}
For every integer $m\geq 1$ denote by $\COAX((\Sph)^m)$ the subset of $(\Sph)^m$
consisting of $m$-tuple of points that sit on the same line.
\end{notation}

\begin{mainthm}{$\widehat{\text{V}}$}[Special relative decorated representation spaces]\label{mainthm:special}
Assume that $(g,n,\bm{\th})$ is special and let $k$ be the number of integral entries of $\bm{\th}$.
Then all decorated homomorphisms are coaxial, and the following hold.
\begin{itemize}
\item[(i)]
Assume $g=1$ and $k=n$ with $\sum_i(\th_i-1)$ even.
Inside $\wh{\Hom}_{\bm{\th}}(\grp_{1,n},\SU_2)$ 
\begin{itemize}
\item[(i-a)]
the central locus is isomorphic to $\{\pm I\}^2\times (\Sph)^n$
and the central elementary locus corresponds to $\{\pm I\}^2\times \COAX((\Sph)^n)$;
\item[(i-b)]
the non-central locus is isomorphic to an $(\Sph)^{n+1}$-bundle over $S^2\setminus\{\text{$4$ points}\}$ and the non-central elementary locus 
is a subbundle with fiber isomorphic to $\COAX((\Sph)^{n+1})$.
\end{itemize}
Hence the non-elementary non-central locus
in $\wh{\Hom}_{\bm{\th}}(\grp_{1,n},\SU_2)$
is a connected oriented manifold of dimension $2n+4$,
and the corresponding non-elementary non-central locus in $\wh{\Rep}_{\bm{\th}}(\grp_{1,n},\SU_2)$
is a connected oriented manifold of dimension $2n+1$.
\item[(ii)]
Assume $g=0$, and $k=n-1$ or $d_1(\bm{\th}-\bm{1},\ZZ^n_o)<1$.
Then $\wh{\Rep}_{\bm{\th}}(\grp_{0,n},\SU_2)$ is empty.
\item[(iii)]
Let $g=0$ and $d_1(\bm{\th}-\bm{1},\ZZ^n_o)=1$.
\begin{itemize}
\item[(iii-a)]
If $k=0$, then $\wh{\Hom}_{\bm{\th}}(\grp_{0,n},\SU_2)$
is isomorphic to $\Sph$ and consist of
one class of elementary decorated representations.
\item[(iii-b)]
If $1\leq k\leq n-2$, then $\wh{\Hom}_{\bm{\th}}(\grp_{0,n},\SU_2)$ is isomorphic
to $(\Sph)^{1+k}$; the elementary locus corresponds to $\COAX((\Sph)^{1+k})$
and so $\wh{\Rep}^{ne}_{\bm{\th}}(\grp_{0,n},\SU_2)$ is 
a connected oriented manifold of dimension $2k-1$.
\item[(iii-c)]
If $k=n\geq 3$, then all homomorphisms
are central and $\wh{\Hom}_{\bm{\th}}(\grp_{0,n},\SU_2)$
is isomorphic to $(\Sph)^n$; 
the elementary locus corresponds to $\COAX((\Sph)^n)$,
and so $\wh{\Rep}^{ne}_{\bm{\th}}(\grp_{0,n},\SU_2)$ is 
a connected, oriented manifold of dimension $2n-3$.
\item[(iii-d)]
$\wh{\Hom}_{\bm{\th}}(\grp_{0,n},\SU_2)$ consists of a single conjugacy class
and its natural structure of analytic space
is reduced if and only if $k=n$ or $k=n-2$.
\end{itemize}
\end{itemize}
\end{mainthm}

\begin{remark}[About reduced analytic structures]
The space $\wh{\Hom}_{\bm{\th}}(\grp_{g,n},\SU_2)$ is naturally described
as the zero locus of finitely many analytic equations inside $\SU_2^{2g}\times \su_2^{\oplus n}$.
For $k<n-2$,  Theorem \ref{mainthm:special}(iii-d) is saying that
at each point of $\wh{\Hom}_{\bm{\th}}(\grp_{0,n},\SU_2)$,
the kernels of the differentials of such equations intersect in a 
vector subspace of dimension larger than $2k-1$.
\end{remark}

An example of decorated representations as in
Theorem \ref{mainthm:special}(i-b) is given by
decorated monodromies of
spherical tori with one conical point
of angle $(2m+1)2\pi$ (with $m\in\NN$),
which are non-elementary non-central (see \cite{EMP}).

The costraints coming from $\SU_2$-representations in genus $0$
as in Theorem \ref{mainthm:special}(ii)-(iii)
were studied in \cite{MP1} and \cite{eremenko:coaxial}.
In \cite{zhu} Zhu has produced a $1$-dimensional moduli space of
spherical surfaces of genus $0$ with $n=4$ conical points: the decorated
monodromies of such spherical surfaces belong to the space
treated in Theorem \ref{mainthm:special}(iii-b) with $k=1$.\\

The existence of a Goldman symplectic form
on decorated representation spaces is
a direct consequence of Lemma \ref{mainlemma:anal}(b)
and Proposition \ref{mainprop:sympl}.

\begin{maincor}{$\widehat{\text{VI}}$}[Symplectic structure on decorated relative representation spaces]\label{maincor:sympl-dec}
Consider the map $\wh{\Rep}^{nc}_{\bm{\th}}(\grp_{g,n},\SU_2)
\rar \Rep^{nc}_{\bm{\th}}(\grp_{g,n},\SU_2)$ that forgets the decoration.
If no $\th_i$ is integral, then
the Goldman symplectic form pulls back to a real symplectic form on
$\wh{\Rep}^{nc}_{\bm{\th}}(\grp_{g,n},\SU_2)$.
\end{maincor}


We finally mention that the mapping class group $\MCG_{g,n}$,
namely the group of orientation-preserving self-homeomorphisms of $(S,\bm{x})$
up to isotopy, naturally acts on all representation spaces introduced above (with or without decoration). Moreover, such action
preserves the non-coaxial locus, the non-elementary locus,
the $\Sigma$ locus, and the symplectic structures.

\subsubsection{Decorated monodromy of a spherical metric.}\label{sec:intro-mon}
We recall that a metric of constant curvature $1$ is locally isometric to a portion of the standard $\Sph$,
which can be identified to the unit sphere in $\su_2$. Thus, simply-connected spherical surfaces
admit a global developing map to $\Sph$, which is a local isometry.

{\bf{Monodromy representation, universal cover and completion.}}
Consider now a spherical surface $(S,\bm{x},h)$ with angles $2\pi\bm{\th}$ at $\bm{x}$,
and lift the spherical metric to the universal cover $\wti{\dot{S}}$
\index{$\wti{\dot{S}}$, $\tilde{\base}$}
of $\dot{S}$.
Fix a basepoint $\base\in\dot{S}$ and a preimage $\tilde{\base}\in\wti{\dot{S}}$, so that
$\pi_1(\dot{S},\base)$ is identified to 
the group of deck transformations of $\wti{\dot{S}}\rar\dot{S}$.
Since $\wti{\dot{S}}$ is simply-connected, the above observation ensures that there exists a locally isometric
developing map $\iota:\wti{\dot{S}}\rar\Sph$.\index{$\iota$}
Moreover such map is equivariant
with respect to $\pi_1(\dot{S},\base)$, which acts on $\wti{\dot{S}}$ via deck transformations and on $\Sph$
via a {\it{monodromy homomorphism}} $\ol{\rho}:\pi_1(\dot{S},\base)\rar\SO_3(\RR)$.
We incidentally remark that the overline notation $\ol{\rho}$ is used for homomorphisms to $\SO_3(\RR)$, to differentiate them from their lifts that take values in $\SU_2$.


Note that the metric completion $\wh{S}$\index{$\wh{S}$, $\pa\wh{S}$}
of $(\wti{\dot{S}},\tilde{h})$ is independent of the chosen $h$
and in fact it could be defined in purely topological terms: here we will just say that the added
points $\pa\wh{S}:=\wh{S}\setminus \wti{\dot{S}}$ correspond to loops in $\pi_1(\dot{S},\base)$ that simply wind about some puncture.
It is easy to check that the developing map uniquely extends to 
$\hat{\iota}:\wh{S}\rar\Sph$.\index{$\hat{\iota}$}

We recall that we have fixed an isomorphism $\pi_1(\dot{S},\base)\cong\grp_{g,n}$: hence, $\pa\wh{S}$ can be identified to $\Bcal$.

\begin{definition}[Decorated monodromy]
Let $h$ be a spherical metric on $(S,\bm{x})$, whose monodromy homomorphism
can thus be viewed as a $\ol{\rho}:\grp_{g,n}\rar\SO_3(\RR)$.\index{$\ol{\rho}$}
The {\it{decorated monodromy homomorphism}} of $(S,\bm{x},h)$
is the couple $(\ol{\rho},\Axis)$, where
$\Axis$ is defined by 
$\Axis(\beta_i):=\th_i\cdot\hat{\iota}(\beta_i)$ for all $i$.
Its $\SO_3(\RR)$-conjugacy class $[\ol{\rho},\Axis]$ is uniquely defined and
is the {\it{decorated monodromy representation}} associated to $(S,\bm{x},h)$.
\end{definition}

Given a homomorphism $\rho$ to $\PSL_2(\CC)$,
it is clear what is meant by an $\SL_2(\CC)$-lifting of $\rho$.
It is well-known that all monodromies of $\CC\PP^1$-structures admit a lift to $\SL_2(\CC)$ (see, for example,
\cite[Lemma 1.3.1]{gkm}).
We will now introduce the notion of lift for decorated homomorphisms.

\begin{definition}[$\SU_2$-lifting of decorated homomorphism]
An {\it{$\SU_2$-lifting}} of a decorated homomorphism $(\ol{\rho},\Axis)\in\wh{\Hom}_{\bm{\th}}(\grp_{g,n},\SO_3(\RR))$
is a $(\rho,\Axis)\in \wh{\Hom}_{\bm{\th}}(\grp_{g,n},\SU_2)$,
where $\rho$ is an $\SU_2$-lifting of $\ol{\rho}$
and $\Ad_{\rho(\beta_j)}(\Axis(\beta_j))=e^{\pi i(\th_j-1)}\Axis(\beta_j)$
for all $j$.
We will also say that $[\rho,\Axis]$ is an $\SU_2$-lifting of $[\ol{\rho},\Axis]$.
\end{definition}

Finally, we will see that there is a special class of spherical surfaces
whose decorated monodromy lands in the singular locus of the representation space.

\begin{definition}[Purely hemispherical surfaces]\label{def:purely-hem}
A {\it{standard hemisphere}} is the closed portion of $\Sph$ bounded by a maximal circle. A {\it{purely hemispherical surface}} is a spherical surface with conical points
obtained from a collection of closed standard hemispheres by identifications
along their boundaries.
\end{definition}

We observe that purely hemispherical surfaces are a special case
of hemispherical surfaces studied in \cite{SCLX} and \cite{GT}.

%
%

The following result collects some properties of decorated monodromy representations of a spherical surface and provides further motivation for the investigation carried on in the present paper.

\begin{mainthm}{$\widehat{\text{VII}}$}[Decorated monodromy of a spherical surface]\label{mainthm:mon-sph}
Let $(\ol{\rho},\Axis)$ be a decorated monodromy homomorphism of a spherical metric on $(S,\bm{x})$ and assume $\chi(\dot{S})<0$. Then
\begin{itemize}
\item[(o)]
the couple $(\ol{\rho},\Axis)$ admits $2^{2g}$ $\SU_2$-liftings.
More precisely, there exists a bijective corrspondence between
$\SU_2$-liftings of $(\ol{\rho},\Axis)$ and spin structures on $S$.
\end{itemize}
Fix now one $\SU_2$-lifting $(\rho,\Axis)$. Then
\begin{itemize}
\item[(i)]
the couple $(\rho,\Axis)$ is non-elementary and $\hat{\Axis}$ achieves at least three values;
\item[(ii)]
if no $\th_i$ is integer, then $\rho$ is non-coaxial;
\item[(iii)]
if $(\rho,\Axis)$ belongs to the locus $\wti{\Sigma}$, then the spherical surface
is purely hemispherical;
\item[(iv)]
if $(\rho,\Axis)$ belongs to $\wti{\Sigma}_0$, then the spherical surface is
a cover of $\Sph$ with branch points on a great circle of $\Sph$.
\end{itemize}
\end{mainthm}

It would be interesting to know exactly which decorated representations
arise as decorated monodromies of spherical surfaces with conical points
of assigned angles. 

\begin{question}[Geometrization of decorated representations]\label{q:dec}
Given $\bm{\th}$, which non-elementary
decorated representation in $\wh{\Hom}^{ne}_{\bm{\th}}(\grp_{g,n},\SU_2)$
is an $\SU_2$-lifting of the
decorated monodromy of some spherical metric on $(S,\bm{x})$?
\end{question}

We recall that
Goldman-Xia \cite{goldman-xia:ergodicity}
proved that the mapping class group ergodically acts on the space
of relative (non-decorated) $\SU_2$-representation.
Consider now the space of spherical metrics on a fixed surface
with conical points of angles $2\pi\bm{\th}$ and suppose that it is non-empty.
The image of the monodromy map associated to such metrics is a non-empty,
mapping-class group invariant,
open subset of the relative representation space (see \cite{MP:moduli}).
Hence, the locus of spherical monodromy representations is a dense open
subset of full measure inside the space of relative representations.

We mention that,
in \cite[Corollary D]{FG}, Faraco-Gupta determine
which $\SO_3(\RR)$-representation
arises as monodromy of some spherical surface with conical points.
However such result only partially answers Question \ref{q:dec}
for two reasons.
First, Corollary D of \cite{FG}
only deals with non-decorated representations.
Moreover, the angles of the conical points
of the surface whose monodromy realizes the given
representation are only determined up to adding multiples of $2\pi$.

\subsection{Acknowledgements}
We thank Philip Boalch for useful and precise remarks.
The first-named author was partially supported by ``Moduli and Lie theory'' PRIN 2017 research grant and by
INdAM research group GNSAGA. The second-named author was supported by EPSRC grant EP/S035788/1.


\section{Conventions and centralizers}

In this short section we collect notations and conventions employed  throughout the paper.

\subsection{Conventions}\label{sec:conventions}

{\bf{M\"obius transformations.}}
We let $\SL_2(\CC)$ act on $\CC\PP^1$ via projective transformation
and we identify $\CC$ with the open subset $\{[X_0:X_1]\in\CC\PP^1|X_1\neq 0\}$
via the map $\CC\ni z\to [z:1]\in\CC\PP^1$.
This way $\SL_2(\CC)$ acts on $\CC$ via M\"obius transformations as
\[
\left(\begin{array}{cc}
a & b\\
c & d
\end{array}\right)\cdot z=\frac{az+b}{cz+d}.
\]

%


{\bf{Exponential and volume form.}}
For every $0\neq X\in\su_2$, the action of $\ee(X)$ on $\Sph$ fixes $\hat{X}:=X/\|X\|$ and acts as a rotation of angle $2\pi\|X\|$ on $T_{\hat{X}}\Sph$.
In fact, we record the following easy observation without proof.

\begin{lemma}\label{lemma:exponential-iso}
For every $\th>0$ non-integer, the restriction
of $\ee:\su_2\rar\SU_2$ to the sphere
of radius $\th$
is an isomorphism onto its image.
\end{lemma}

The volume $3$-form on $\su_2$ associated to $\Kill$ is
$\mathrm{dvol}_{\Kill}:=Y^*\wedge Z^*\wedge [Y,Z]^*$,\index{$\mathrm{dvol}_{\Kill}$, $\omega_{\Kill}$}
where $Y,Z$ are orthogonal vectors of $\su_2$ of norm $1$
and $X^*=\Kill(X,\cdot)$.
If $v$ is the outgoing radial vector field on $\su_2$ whose
value at $X\in\su_2$ is exactly $X$, we set
$\omega_{\Kill}:=\iota_v(\mathrm{dvol}_{\Kill})$, which is
$2$-homogeneous and restricts to the standard
area $2$-form on the unit sphere.
We will say that $0\neq X\in\su_2$ is {\it{integral}}
if $\ee(X)=\pm I$, namely if $\|X\|\in\ZZ$.

{\bf{Tangent space.}}
Given $r\in \SU_2$,
we will identify $T_r \SU_2$ with $T_I \SU_2\cong\su_2$ via the differential
$T_I \SU_2\rar T_r \SU_2$
of the right-multiplication $\SU_2\ni h\mapsto h\cdot r\in \SU_2$ by $r$. More explicitly,
\[
\xymatrix@R=0in{
\su_2 \ar[rr] && T_r \SU_2\\
\dot{V} \ar@{|->}[rr] && [\exp(t\dot{V})r]
}
\]
where $[\exp(t\dot{V})r]$ is the tangent vector determined at $t=0$ by the path $t\mapsto \exp(t\dot{V})r=(I+t\dot{V}+o(t))r$.
An analogous identification can be employed for $\SO_3(\RR)$, $\PSL_2(\CC)$ or $\SL_2(\CC)$.

{\bf{Distance from the identity.}}
Consider the map
$D_I:\SU_2\lra [0,1]$\index{$D_I$}
defined as $D_I(B):=\frac{1}{\pi}\mathrm{arccos}(\tr(B)/2)$.

\begin{lemma}[Distance from the identity]\label{lemma:D_I}
The map $D_I$ is continuous. Moreover, 
\begin{itemize}
\item[(i)]
$D_I$ is analytic and submersive away from $\pm I$;
\item[(ii)]
$2\pi\cdot D_I(B)$ is the distance between $I$ and $B$ with respect to the metric
induced on $\SU_2$ by $\Kill$;
\item[(iii)]
$D_I(\ee(X))=d(\|X\|,\ZZ_o)$
for all $X\in\su_2$, where $d(\cdot,\ZZ_o)$ is the distance
in $\RR$ from the subset of odd integer points.
\end{itemize}
\end{lemma}
\begin{proof}
Continuity is obvious, since $\mathrm{arccos}:[0,1]\rar[0,\pi]$ is continuous.

(i) follows from the fact that 
the restriction of $\mathrm{arccos}$ to $(0,1)$ is analytic and has nonvanishing
derivative.

(ii)
Let $X_0=\frac{1}{2}\begin{pmatrix} i & 0\\ 0 & -i\end{pmatrix}\in\Sph$
and consider the geodesic path $B(t):=\mathrm{exp}(2\pi t X_0)$ on $\SU_2$ for $t\in[0,1]$
Then $B(0)=I$, $B(1)=-I$ and $D_I(B(t))=t$.
On the other hand, $\dot{B}(t)=2\pi X_0 \cdot B(t)$, and so $\|\dot{B}(t)\|=2\pi$,
which implies that the distance induced by $\Kill$ on $\SU_2$ between $I$ and $B(t)$
is $2\pi t$.
The conclusion follows, since all geodesic paths on $\SU_2$ originating from $I$ with speed $2\pi$ are conjugate to $B(t)$.

(iii)
Again, by conjugacy invariance it is enough to prove it for $X=tX_0$ with $t\in\RR$.
Since both hand sides are invariant under $t\mapsto t+2$, we can assume that
$t\in[0,2]$.
Since $\ee(X)=\exp(\pi (1+t)X_0)$, we have $D_I(\ee(X))=|1-t|$.
Finally, note that $|1-t|=|1-\|X\||$ is exactly the distance of $t\in[0,2]$ from
the closest odd number.
\end{proof}

\subsection{Centralizers}
The following notion of centralizer will be useful when considering the conjugacy action
of $\PSU_2$ on a product of copies of $\SU_2$ and $\su_2$.

\begin{definition}[Centralizer and infinitesimal centralizer]
Given a collection $\mathcal{C}$ of elements of $\SU_2$, we denote by
$Z(\mathcal{C})\subset \SU_2$\index{$Z(\mathcal{C})$}
the {\it{centralizer}} of $\mathcal{C}$,
namely the subset of elements of $\SU_2$ that commute with every
element of $\mathcal{C}$, and by $\Zfrak(\mathcal{C})\subset\su_2$
\index{$\Zfrak(\mathcal{C})$}
the {\it{infinitesimal centralizer}} of $\mathcal{C}$, namely the Lie algebra of
$Z(\mathcal{C})$.
\end{definition}

Note that $Z(\mathcal{C})$ always contains $Z(\SU_2)=\{\pm I\}$.
If $\rho$ is a homomorphism that takes values in $\SU_2$,
then the centralizer and the infinitesimal centralizer of $\rho$, denoted by $Z(\rho)$
\index{$Z(\rho)$, $\Zfrak(\rho)$}
and by $\Zfrak(\rho)$ respectively, are just the centralizer 
and the infinitesimal centralizer of the image of $\rho$.

The above definition of (infinitesimal) centralizer can be given
for subsets of $\SL_2(\CC)$ or homomorphisms in $\SL_2(\CC)$.

If $(\rho,\Axis)$ is a decorated homomorphism in $\SU_2$, then $Z(\rho,\Axis)\subseteq Z(\rho)$\index{$Z(\rho,\Axis)$}
denotes the subgroup of elements
$g\in Z(\rho)$ such that $\Ad_g$ fixes $\IM(\Axis)$, and
$\Zfrak(\rho,\Axis)$ is the Lie algebra of $Z(\rho,\Axis)$.


\begin{lemma}[Non-coaxial homomorphism has trivial centralizer]\label{lemma:no-auto-undec}
For every $\rho$ in $\Hom(\grp_{g,n},\SL_2(\CC))$ the following hold.
\begin{itemize}
\item[(i)]
The centralizer $Z(\rho)\subseteq\SL_2(\CC)$ of $\rho$ contains $\pm I$;
the stabilizer of $\rho$ is $Z(\rho)/\{\pm I\}\subseteq\PSL_2(\CC)$.
\item[(ii)]
$\rho$ non-coaxial $\iff$ $Z(\rho)$ finite $\iff$ $Z(\rho)=\{\pm I\}$.
\item[(iii)]
$\rho$ is coaxial non-central $\iff$ $Z(\rho)/\{\pm I\}$ is a $1$-parameter subgroup of $\PSL_2(\CC)$;
$\rho$ is central $\iff$ $Z(\rho)=\SL_2(\CC)$.
\end{itemize}
\end{lemma}
\begin{proof}
(i) Clearly $\pm I\in Z(\rho)$.
By definition, the stabilizer of $\rho$
under the action of $\SL_2(\CC)$ by conjugation is $Z(\rho)$.
Hence, the stabilizer of $\rho$ under the action of $\PSL_2(\CC)$ by conjugation
is $Z(\rho)/\{\pm I\}$.

(ii)
We will show that $Z(\rho)=\{\pm I\}$ $\implies$
$Z(\rho)$ finite $\implies$ $\rho$ non-coaxial
$\implies$ $Z(\rho)=\{\pm I\}$.

The first implication is obvious.

Let us prove the second implication, namely
$Z(\rho)$ finite $\implies$ $\rho$ non-coaxial, by contradiction.
If $\rho$ were coaxial, then
$\mathrm{Im}(\rho)$ would be contained in a $1$-parameter
subgroup $H$. It would then follow that
$H\subseteq Z(\rho)$, and so $Z(\rho)$ would be infinite.

As for the third implication, $\rho$ non-coaxial $\implies$ $Z(\rho)=\{\pm I\}$,
let $g\in\SL_2(\CC)$ such that $g\rho g^{-1}=\rho$. We want to show that $g=\pm I$.

Since $\rho$ is non-coaxial, its image in $\PSL_2(\CC)$ is not contained in a $1$-parameter subgroup. It follows that there are two matrices
$M,N$ in the image of $\rho$ that do not commute: in particular, $M,N\neq \pm I$.
Since $gMg^{-1}=M$ and $gNg^{-1}=N$,
and every eigenspace of $M$ and $N$ has dimension $1$,
every eigenvector  for $M$ or $N$ must be an eigenvector for $g$ too.

Suppose first that $M$ or $N$ is diagonalizable. Without loss of generality, we can assume that $M$ is diagonalizable and that $\CC^2$ decomposes as $\CC^2=V_1\oplus V_2$ into eigenspaces of dimension $1$ for $M$. Then $g$ must be diagonalizable and $V_1,V_2$ must be made of eigenvectors for $g$ too.
Since $N$ does not commute with $M$, it follows that $g$ must be $\pm I$.

Suppose now that both $M$ and $N$ are not diagonalizable,
and call $V_1$ the only eigenspace of $M$ and
$V_2$ is the only eigenspace of $N$.
We claim that $\CC^2=V_1\oplus V_2$,
and $g$ must be diagonalizable and $V_1,V_2$ must consist of eigenvectors for $g$,
from which it follows that $g=\pm I$.

By contradiction, suppose $V_1=V_2$ and
pick a basis of $\CC^2$ with the first vector in $V_1=V_2$.
With respect to such a basis,
the endomorphisms $M,N$ of $\CC^2$
would be represented by two matrices of type
$\pm \left(\begin{smallmatrix}
1 & \ast\\
0 & 1\\
\end{smallmatrix}\right)$, and so $M,N$ would commute.

(iii) Clearly $\rho$ is central, and so takes values $\pm I$, if and only if $Z(\rho)=\SL_2(\CC)$.

If $Z(\rho)/\{\pm I\}$ is a $1$-parameter subgroup of $\PSL_2(\CC)$, then $\rho$ must be coaxial by (ii) but cannot be central.

Vice versa, suppose that $\rho$ is coaxial but not central, and so its image
is contained in a $1$-parameter subgroup  $H$ of $\SL_2(\CC)$. We want to show that
$Z(\rho)=\pm H$.

Let $X\in\psl_2(\CC)$ that generate $H$ and such that $\exp(X)\neq \pm I$ is an element in the image of $\rho$.
Clearly, $\pm H$ is contained inside $Z(\rho)$. 

If $\exp(X)$ is diagonalizable, then it has distinct eigenvalues; 
it follows that every element of $Z(\rho)$ different from $\pm I$ must have the same eigenvectors as $\exp(X)$, and so $Z(\rho)$ is contained in $\pm H$. 

If $\exp(X)$ is not diagonalizable, then (up to conjugation) we can assume that
$X$ is strictly upper-triangular. It can then be checked by hand that
$gXg^{-1}=X$ if and only if $g=\pm\exp(tX)$ for some $t\in\CC$.
The conclusion follows.
\end{proof}

Here is an immediate consequence.

\begin{corollary}\label{cor:proper-undec}
The group $\PSL_2(\CC)$ acts properly and discontinuously
on the non-coaxial locus $\Hom^{nc}(\grp_{g,n},\SL_2(\CC))$.
As a consequence, $\PSU_2$ acts properly and discontinuously
on $\Hom^{nc}(\grp_{g,n},\SU_2)$.
\end{corollary}
\begin{proof}
Recall that a homomorphism in $\SL_2(\CC)$ is non-coaxial if its image
in $\PSL_2(\CC)$ is not contained in a $1$-parameter subgroup.
Thus, non-coaxial homomorphisms in $\SL_2(\CC)$ are Zariski-dense.
It follows that
the action of $\PSL_2(\CC)$ on $\Hom^{nc}(\grp_{g,n},\SL_2(\CC))$
is proper by \cite[Section 5.3.4]{labourie:notes} 
and it is free by Lemma \ref{lemma:no-auto-undec}.
The second claim immediately follows.
\end{proof}

A similar conclusion holds for decorated homomorphisms
in $\SU_2$.

\begin{lemma}[Non-elementary decorated homomorphism has trivial centralizer]\label{lemma:no-auto}
For every $(\rho,\Axis)$ in $\wh{\Hom}(\grp_{g,n},\SU_2)$ the following hold:
\begin{itemize}
\item[(i)]
the centralizer $Z(\rho,\Axis)\subseteq\SU_2$ contains $\pm I$;
the stabilizer of $(\rho,\Axis)$ is $Z(\rho,\Axis)/\{\pm I\}\subseteq\PSU_2$.
\item[(ii)]
$(\rho,\Axis)$ non-elementary $\iff$ $Z(\rho,\Axis)$ finite $\iff$ $Z(\rho,\Axis)=\{\pm I\}$.
\end{itemize}
\end{lemma}
\begin{proof}
(i) is analogous to Lemma \ref{lemma:no-auto-undec}(i).

(ii)
We will show that $Z(\rho,\Axis)=\{\pm I\}$ $\implies$
$Z(\rho,\Axis)$ finite $\implies$ $(\rho,\Axis)$ non-elementary
$\implies$ $Z(\rho,\Axis)=\{\pm I\}$.

The first implication again is obvious and
the second implication, namely $Z(\rho,\Axis)$ finite $\implies$ $(\rho,\Axis)$ non-elementary,
is entirely analogous to the second implication
in the proof of Lemma \ref{lemma:no-auto-undec}(ii).

As for the third implication,
$(\rho,\Axis)$ non-elementary
$\implies$ $Z(\rho,\Axis)=\{\pm I\}$,
let $g\in\SU_2$ such that $g(\rho,\Axis)g^{-1}=(\rho,\Axis)$. We want to show that $g=\pm I$.

If $\rho$ is non-coaxial, then the conclusion follows from 
Lemma \ref{lemma:no-auto-undec}.
If $\rho$ is central, then by definition $\Axis$ must achieve two linearly independent values  $M,N\in\su_2$
and $g$ must commute with both of them. 
As in Lemma \ref{lemma:no-auto-undec}, this implies that $g=\pm I$.

Finally, if $\rho$ is coaxial but non-central with image contained in the $1$-parameter subgroup $H=\exp(\hfrak)$, then $\rho$ achieves a value $M\in H$ different from $\pm I$.
Thus, $gM g^{-1}=M$ implies that $g$ must be contained in $H$. 
Since $(\rho,\Axis)$ is not elementary, $\Axis$ achieves a value
$N\in\su_2$ not in $\hfrak$. This implies that $g=\pm I$.
\end{proof}

Again, here is an immediate consequence.

\begin{corollary}\label{cor:proper}
The group $\PSU_2$ acts properly and discontinuously
on the non-elementary locus $\Hom^{ne}(\grp_{g,n},\SU_2)$.
\end{corollary}
\begin{proof}
Properness of the action is immediate because $\SU_2$ is compact.
Freeness follows from Lemma \ref{lemma:no-auto}.
\end{proof}

\section{Representation spaces}\label{sec:rep}

We recall from the introduction that
$\grp_{g,n}$ is the group generated by 
$\{{\mu}_1,{\nu}_1,{\mu}_2,\dots,{\nu}_g,{\beta}_1,\dots,{\beta}_n\}$
and with the unique relation $[{\mu}_1,{\nu}_1]\cdots[{\mu}_g,{\nu}_g]\cdot{\beta}_1\cdots{\beta}_n=e$,
and that $\Bcal_i$ is the conjugacy class of ${\beta}_i$.

We will use the symbol $V$ to denote
$V=\mathcal{M}_{2,2}(\CC)$. 
Note that $\SU_2$ is a real algebraic subset 
and $\SL_2(\CC)$ is a complex algebraic subset
of the vector space $V$.


\subsection{Topology and semi-algebraic structure}\label{ssc:algebraic}

The goal of this section is to prove Lemma \ref{mainlemma:alg}.
In order to do so, and to endow the absolute $\SU_2$-homomorphism space
with an algebraic structure,
we follow a classical idea and we
embed the absolute $\SU_2$-homomorphism space inside
a smooth subset $\Gcal$ of $V^{\oplus (2g+n)}$, so that its image is described
by an algebraic equation $R=I$. The action of $\PSU_2$ on the homomorphism
space will be induced by a natural action on $V^{\oplus(2g+n)}$
that preserves $\Gcal$ and is compatible with the map $R$.
The relative case and the case of $\SL_2(\CC)$ are dealt with in
an analogous fashion.\\

\subsubsection{The embedding $\bm{\lambda}$.}\label{ssc:lambda}
Inside $V^{\oplus(2g+n)}$ consider the smooth real algebraic subset
$\Gcal:=\SU_2^{2g+n}$\index{$\Gcal$, $\lambda$}
of dimension $6g+3n$, whose points
$(M_1,N_1,\dots,M_n,N_n,B_1,\dots,B_n)$ will be often denoted by
$(\bm{M},\bm{N},\bm{B})$ for sake of brevity.

Note that the map
\[
\lambda:\xymatrix@R=0in{
\Hom(\grp_{g,n},\SU_2)\ar[r] & \Gcal\\
\rho\ar@{|->}[r] &
(\rho({\mu}_1),\rho({\nu}_1),\dots,\rho({\mu}_g),\rho({\nu}_g),\rho({\beta}_1),\dots,\rho({\beta}_n))
}\index{$\lambda$}
\]
is injective.
Thus, the set $\Hom(\grp_{g,n},\SU_2)$ can be identified to the closed algebraic subset
of $\Gcal$ consisting of those $(\bm{M},\bm{N},\bm{B})$ such that
\begin{equation}\label{eq:I}
[M_1,N_1]\cdots[M_g,N_g]\cdot B_1\cdots B_n=I.
\end{equation}

\begin{remark}
It can be shown that the algebraic structure induced on
$\Hom(\grp_{g,n},\SU_2)$ is independent of the choice
of the generators $\{\mu_i,\lambda_i,\beta_j\}$ of $\grp_{g,n}$.
\end{remark}


Fix $\bm{\th}=(\th_1,\dots,\th_n)\in\RR^n_{>0}$
and let $k$ be the number of integral entries of $\bm{\th}$.
Call $\Gcal_{\bm{\th}}$\index{$\Gcal_{\bm{\th}}$}
the set of all elements
$(\bm{M},\bm{N},\bm{B})$ in $\Gcal$
such that $B_i\in \cla_{\dd_i}$ for $i=1,\dots,n$.
Clearly $\Gcal_{\bm{\th}}$ is a closed algebraic subset of $\Gcal$,
which is smooth of dimension $6g+2(n-k)$.

\subsubsection{The map $\bm{R}$.}\label{ssc:R}
Let now $R$\index{$R$, $R_{\bm{\th}}$}
be the real analytic map defined as
\[
R: 
\xymatrix@R=0in{
\Gcal\ar[rr] && \SU_2\\
(\bm{M},\bm{N},\bm{B})\ar@{|->}[rr] &&\prod_j [M_j,N_j]\prod_i B_i
}
\]
and denote by $R_{\bm{\th}}:\Gcal_{\bm{\th}}\rar \SU_2$ the restriction of $R$ to $\Gcal_{\bm{\th}}$.
Via the embedding $\lambda$,
the space $\Hom(\grp_{g,n},\SU_2)$ identifies to $R^{-1}(I)$
and $\Hom_{\bm{\th}}(\grp_{g,n},\SU_2)$ identifies to 
$R_{\bm{\th}}^{-1}(I)$. 

\begin{lemma}[Number of equations cutting the image of $\lambda$]\label{lemma:number-eq}
The space $\lambda(\Hom(\grp_{g,n},\SU_2)$ inside $\Gcal$ is described by $3$ real algebraic equaitons.
If $k$ is the number of integer entries in $\bm{\th}$, then
the image of $\lambda(\Hom_{\bm{\th}}(\grp_{g,n},\SU_2)$ inside $\Gcal$ is described by $n+2k+3$ real algebraic equations.
\end{lemma}
\begin{proof}
The only equation cutting $\lambda(\Hom(\grp_{g,n},\SU_2)$ is $R=I$, which amounts to $3$ scalar equations.

Consider now $\lambda(\Hom_{\bm{\th}}(\grp_{g,n},\SU_2)$.
For each $i$ with non-integer $\th_i$, the condition $B_i\in\cla_{\dd_i}$ translates into
the equation $\tr(B_i)=2\cos(\pi\dd_i)$. For each $j$ with integer $\th_j$,
the condition $B_j\in\cla_{\dd_j}$ translates into $B_j=I$ or $B_j=-I$, which is equivalent to three scalar equations. Moreover $R_{\bm{\th}}=I$ contributes with another triple of scalar equations.
To sum up, we obtain $n-k$ equations for the non-integer $\th_i$,
$3k$ equations for the integer $\th_j$, and $3$ more equations for $R_{\bm{\th}}=I$:
in total, we obtain $(n-k)+3k+3=n+2k+3$.
\end{proof}

The above Lemma \ref{lemma:number-eq} endows both spaces with a natural structure
of algebraic schemes.

\subsubsection{The conjugacy action.}\label{ssc:conj}
The group $\PSU_2$ acts on $V^{\oplus (2g+n)}$ componentwise
via conjugation, preserving $\Gcal$ and $\Gcal_{\bm{\th}}$,
and $R$ is $\PSU_2$-equivariant.
Hence such action induces the usual conjugacy action
on the homomorphism spaces via $\lambda$.

We remark that the absolute and relative representation spaces
have a natural structure of algebraic schemes too, being described by algebraic
equations (induced by $R=I$ and by $R_{\bm{\th}}=I$ respectively)
inside $\Gcal/\PSU_2$ and $\Gcal_{\bm{\th}}/\PSU_2$ respectively.

\subsubsection{The complex case.}\label{ssc:complex-case}
As mentioned at the beginning of 
Section \ref{ssc:algebraic},
in order to endow the absolute and relative $\SL_2(\CC)$-homomorphism
spaces with an algebraic structure,
we perform the above constructions
just replacing the group $\SU_2$ by $\SL_2(\CC)$.
More explicitly, we proceed as follows.

Denote by $\Gcal_\CC$\index{$\Gcal_\CC$, $\lambda_\CC$}
the smooth complex algebraic subset $\SL_2(\CC)^{2g+n}\subset V^{\oplus(2g+n)}$
of complex dimension $6g+3n$ and let
$\lambda_\CC:\Hom(\grp_{g,n},\SL_2(\CC))\hra\Gcal_\CC$
be the injective map defined analogously to $\lambda$.
Note now that, if $B\in\SU_2$, then
its $\SL_2(\CC)$-conjugacy class
is smooth of complex dimension $2$ if $\tr(B)\neq \pm 2$,
and of dimension $0$ if $\tr(B)=\pm 2$.
Hence,
$\Gcal_{\bm{\th},\CC}$ is a closed complex algebraic subset of 
$\Gcal_\CC$,
which is smooth of complex dimension $6g+2(n-k)$.
Moreover, via $\lambda_\CC$,
the space $\Hom(\grp_{g,n},\SL_2(\CC))$ identifies to $R_\CC^{-1}(I)$
and $\Hom_{\bm{\th}}(\grp_{g,n},\SL_2(\CC))$ to $R_{\bm{\th},\CC}^{-1}(I)$,
where $R_\CC:\Gcal_\CC\rar \SL_2(\CC)$ and $R_{\bm{\th},\CC}:\Gcal_{\bm{\th},\CC}\rar \SL_2(\CC)$ are defined in the obvious way
and are $\PSL_2(\CC)$-equivariant.

As in the $\SU_2$ case with Lemma \ref{lemma:number-eq}, we have the following.

\begin{lemma}[Number of equations cutting the image of $\lambda$]\label{lemma:number-eq-C}
The space $\lambda_{\CC}(\Hom(\grp_{g,n},\SL_2(\CC))$ inside $\Gcal_{\CC}$ is described by $3$ complex algebraic equations.
If $k$ is the number of integer entries in $\bm{\th}$, then
the image of $\lambda_{\CC}(\Hom_{\bm{\th}}(\grp_{g,n},\SL_2(\CC))$ inside $\Gcal_{\CC}$ is described by $n+2k+3$ complex algebraic equations.
\end{lemma}

\subsubsection{Algebraic and semi-algebraic structure.}

We can now turn to the main result of this section.

\begin{proof}[Proof of Lemma \ref{mainlemma:alg}]
%

%
(o) By the constructions performed in Sections \ref{ssc:lambda}-\ref{ssc:R}-\ref{ssc:conj} 
the space $\Hom(\grp_{g,n},\SU_2)$ 
is homeomorphic via $\lambda$ to
a real algebraic subset of the vector space
$V^{\oplus(2g+n)}$, on which $\PSU_2$ acts algebraically:
such $\lambda$ induces on $\Hom(\grp_{g,n},\SU_2)$
a real algebraic structure.
As discussed in Remark \ref{rmk:quotients}, if $p_1,\dots,p_\ell$
generate the algebra of $\PSU_2$-invariant polynomial functions
on $V^{\oplus(2g+n)}$, then
$\Rep(\grp_{g,n},\SU_2)$ can be realized
as a semi-algebraic subset of $\RR^\ell$.
As similarly discussed in Section \ref{ssc:complex-case}, 
$\Hom(\grp_{g,n},\SL_2(\CC))$ embeds
via $\lambda_\CC$ as a complex algebraic
subset of $V^{\oplus(2g+n)}$, on which $\PSL_2(\CC)$ acts algebraically.
As in Remark \ref{rmk:quotients}, if $q_1,\dots,q_s$ generate
the algebra of $\PSL_2(\CC)$-invariant polynomial functions on $V^{\oplus(2g+n)}$, 
then
$\Rep(\grp_{g,n},\SL_2(\CC))$ can be 
realized as a complex algebraic subset of $\CC^s$.

(i) Recall that $\grp_{g,n}$ is a free group and that
the natural map $\SU_2\rar\SO_3(\RR)$ is an algebraic double cover.
It follows that
the map $\Hom(\grp_{g,n},\SU_2)\rar\Hom(\grp_{g,n},\SO_3(\RR))$ is 
algebraic, since it is the restriction of the natural map
$(\SU_2)^{2g+n}\rar \SO_3(\RR)^{2g+n}$, and that
is a finite cover. Hence it is a local homeomorphism.
Taking the quotient by $\SO_3(\RR)$, we obtain a cover
between the corresponding representation spaces,
which can be easily seen, following the construction
in (o), to be an algebraic map.

(ii) Let $\rho\in\Hom(\grp_{g,n},\SU_2)$ with 
$\lambda(\rho)=(\bm{M},\bm{N},\bm{B})$.
It is easy to see that $\rho$ is coaxial if and only if
the sum of the images of $\mathrm{Ad}_{M_j}-I$, $\mathrm{Ad}_{N_j}-I$, $\mathrm{Ad}_{B_i}-I$ do not span the whole $\su_2$.
Since the latter property can be expressed through algebraic equations,
the coaxial locus inside $\Hom(\grp_{g,n},\SU_2)$ is closed algebraic.
The $\SL_2(\CC)$ case is similar.
%

(iii) 
The locus $\Hom_{\bm{\th}}(\grp_{g,n},\SU_2)$ inside 
$\Hom(\grp_{g,n},\SU_2)$ is described by the $\PSU_2$-invariant
algebraic equations $\tr(B_i)=-2\cos(\pi\th_i)$ for $i=1,\dots,n$,
and so it is a closed $\PSU_2$-invariant algebraic subset.
As a consequence,
$\Rep_{\bm{\th}}(\grp_{g,n},\SU_2)$ is a closed algebraic subset of
$\Rep(\grp_{g,n},\SU_2)$.
The proof for $G_\CC=\SL_2(\CC)$ is analogous, with the only caveat that the equation $\tr(B_i)=-2\cos(\pi\th_i)$ must be replaced by
$B_i=(-I)^{\th_i-1}$ if $\th_i\in\ZZ$.
(iv)
The coaxial locus in $\Hom_{\bm{\th}}(\grp_{g,n},\SU_2)$ is closed algebraic
by (ii) and (iii); hence, 
so is the coaxial locus in $\Rep_{\bm{\th}}(\grp_{g,n},\SU_2)$.
\end{proof}

In the following example we show that $\SU_2$-representation spaces can be
non-algebraic.

\begin{example}[Semi-algebraic nature of $\SU_2$-representation space]\label{example:non-algebraic}
Let $g=1$, $n=2$ and $\th_1=\th_2=t$ for some fixed $t\in (0,1)$.
Via $\lambda$, the homomorphism space $\Hom_{\bm{\th}}(\grp_{1,2},\SU_2)$ 
is identified to the algebraic subset of $(M_1,N_1,B_1,B_2)\in \SU_2^4$ such that
$[M_1,N_1]=B_1 B_2$ and $\tr(B_1)=\tr(B_2)=2\cos\pi(t-1)$.
Consider the algebraic subset $\mathpzc{X}$ of $\Hom_{\bm{\th}}(\grp_{1,2},\SU_2)$ 
defined by $\tr(N_1)=\tr(B_1 B_2)=2$, or equivalently by $N_1=I$ and $B_1=(B_2)^{-1}$.
We claim that $\mathpzc{X}/\SU_2$ is not algebraic.
As a consequence, $\Rep_{\bm{\th}}(\grp_{1,2},\SU_2)$ itself is not algebraic but only semi-algebraic,
according to Lemma \ref{mainlemma:alg}(o). For the same reason,
$\Rep(\grp_{1,2},\SU_2)$ is not algebraic either.

In order to prove the claim, consider the subset $\mathpzc{X}'$ of $\mathpzc{X}$ defined by
$B_1=B$ with $B=\begin{pmatrix}e^{i\pi(t-1)} & 0\\ 0 & e^{-i\pi(t-1)}\end{pmatrix}$. Since the stabilizer of $B$ is the subgroup $H$ of diagonal matrices in $\SU_2$,
the natural inclusion $\mathpzc{X}'\hookrightarrow \mathpzc{X}$ induces an isomorphism
$\mathpzc{X}'/H\cong \mathpzc{X}/\SU_2$. 
Clearly, $\SU_2\ni M\mapsto (M,I,B,B^{-1})\in \mathpzc{X}'$ is an isomorphism
and so we only have to show that $\SU_2/H$ is not algebraic.

A quick way to proceed is to note that 
$\SU_2/H$ is homeomorphic to a closed $2$-disk
and to recall that an affine, complete, real algebraic variety
has a non-zero $\ZZ/2$-valued fundamental class
(see \cite[Proposition 11.3.1]{BCR}). Since this is not the case for the closed $2$-disk, $\SU_2/H$ cannot be algebraic.

For a direct approach, recall that $\SU_2$-invariant functions
on the homomorphism space are polynomials in the trace
functions $\rho\mapsto\mathrm{tr}(\rho(\gamma))$ \cite[1.31]{lubotzky-magid}.
Hence, we consider the $H$-invariant
functions $f,g:\SU_2\rar\RR$ defined as
$f(M)=\tr(M)$ and $g(M)=\tr(MB)$.
We assert the $f,g$ generate the algebra of $H$-invariant functions on $\SU_2$.
It follows that $(f,g):\SU_2\rar\RR^2$ induces an isomorphism of $\SU_2/H$
onto its image, which is contained inside the box $[-2,2]^2$.
Since $\SU_2/H$ has dimension $2$, it follows that it cannot be algebraic.

To prove the assertion, let $M=\begin{pmatrix} m_{11} & -\ol{m}_{21} \\ m_{21} & \ol{m}_{11}\end{pmatrix}$ be an element of $\SU_2$. The $H$-orbit of $M$
consists of matrices of type $\begin{pmatrix} m_{11} & -\ol{m}_{21}e^{-is} \\ m_{21} e^{is} & \ \ol{m}_{11}\end{pmatrix}$
for all $s\in\RR$. Now, the algebra of functions on $\SU_2$
is generated by $\mathrm{Re}(m_{11}),\mathrm{Im}(m_{11}),\mathrm{Re}(m_{21}),\mathrm{Im}(m_{21})$. Since $m_{21}\ol{m}_{21}=1-m_{11}\ol{m}_{11}$,
it is easy to see that the subalgebra of $H$-invariant functions is generated by
$\mathrm{Re}(m_{11}),\mathrm{Im}(m_{11})$.
Since $f(M)=2\mathrm{Re}(m_{11})$ and $g(M)=2\mathrm{Re}(m_{11})\cos\pi(t-1)-2\mathrm{Im}(m_{11})\sin\pi(t-1)$, and since $\sin\pi(t-1)\neq 0$,
the algebra of $H$-invariant functions on $\SU_2$ is generated by $f$ and $g$.
\end{example}

We can also easily understand some basic properties of the coaxial locus.

\begin{proposition}[Coaxial locus]\label{prop:coaxial}
Let $\bm{\th}\in\RR^n_{>0}$ and $2g-2+n>0$.
\begin{itemize}
\item[(i)]
The coaxial locus in $\Hom(\grp_{g,n},\SU_2)$ is a closed, connected, algebraic subset
of pure dimension $2g+n+1$. Hence the coaxial locus in $\Rep(\grp_{g,n},\SU_2)$
is a closed, connected, irreducible, algebraic subset of dimension $2g+n-1$.
\item[(ii)]
The coaxial locus in $\Hom_{\bm{\th}}(\grp_{g,n},\SU_2)$ is non-empty if and only if
there exist $\e_1,\dots,\e_n\in\{\pm 1\}$ such that $\sum_i \e_i(\th_i-1)\in 2\ZZ$.
In this case, it is a closed, algebraic subset
of pure dimension $2g+2$. Hence the coaxial locus in $\Rep_{\bm{\th}}(\grp_{g,n},\SU_2)$
is a closed, algebraic subset of pure dimension $2g$.
\end{itemize}
Claim (ii) still holds if we replace $\SU_2$ by $\SL_2(\CC)$ and we interpret the dimensions as complex dimensions.
\end{proposition}
\begin{proof}
(i) Since $\grp_{g,n}$ is free on $\mu_1,\dots,\beta_{n-1}$, we can identify
$\Hom(\grp_{g,n},\SU_2)$ to $\SU_2^{2g+n-1}$.
Hence, the coaxial locus in $\SU_2^{2g+n-1}$ is closed by Lemma \ref{mainlemma:alg}(ii).
Upon seeing $\Sph$ as the unit sphere in $\su_2$,
a $1$-parameter subgroup of $\SU_2$ can be described as $\exp(\RR\cdot X)$ for a certain $X\in\Sph$. As a consequence, the coaxial locus in $\SU_2^{2g+n-1}$ is the image of
the map
\[
\xymatrix@R=0in{
\Sph\times\RR^{2g+n-1}\ar[rr] && \SU_2^{2g+n-1}\\
(X,s_1,t_1,\dots,s_g,t_g,\th_1,\dots,\th_{n-1})
\ar@{|->}[rr] && \bm{\ee}(s_1 X,\dots,\th_{n-1}X)
}
\]
where $\bm{\ee}$\index{$\bm{\ee}$}
is meant to operate componentwise.
Such map has connected, irreducible domain of dimension $2g+n+1$
and its fibers are discrete over non-central $(2g+n-1)$-tuples
and $2$-dimensional over $\{\pm I\}^{2g+n-1}$. 
As a consequence, the coaxial locus in $\Hom(\grp_{g,n},\SU_2)$
is closed algebraic, irreducible, connected, of dimension $2g+n+1$.
Moreover the stabilizer
of a non-central coaxial $(2g+n-1)$-tuple has dimension $1$.
Hence the coaxial locus in $\Rep(\grp_{g,n},\SU_2)$
is closed algebraic, irreducible, connected, of dimension $(2g+n+1)-(3-1)=2g+n-1$.

(ii) Identify $\Hom_{\bm{\th}}(\grp_{g,n},\SU_2)$ to its image via $\lambda$.
Suppose that an element $(M_1,N_1,\dots,B_1,\dots,B_n)$ belongs to the coaxial
locus of $\Hom_{\bm{\th}}(\grp_{g,n},\SU_2)$. Then there exists a line $\hfrak\subset\su_2$
such that $M_1,\dots,B_n$ belong to $H=\exp(\hfrak)$. Moreover, if $X\in\Sph\cap\hfrak$,
then there exist $s_i,t_i\in\RR$ and $\e_1,\dots,\e_n\in\{\pm 1\}$
such that $(M_1,\dots,B_n)=\bm{\ee}(s_1 X,t_1 X,\dots, t_g X,\e_1 \th_1 X,\dots, \e_n\th_n X)$.
This implies that $\ee(\e_1\th_1 X)\cdots \ee(\e_n\th_n X)=I$, namely that
$\sum_j \e_j(\th_j-1)$ is an even integer.
Vice versa, suppose that there exist $\e_1,\dots,\e_n\in\{\pm 1\}$ such that $\sum_j \e_j(\th_j-1)$.
For every such $\bm{\e}=(\e_1,\dots,\e_n)$, consider the map
\[
\xymatrix@R=0in{
\Sph\times\RR^{2g}\ar[rr] && \Hom_{\bm{\th}}(\grp_{g,n},\SU_2)\\
(X,s_1,t_1,\dots,s_g,t_g)
\ar@{|->}[rr] && \bm{\ee}(s_1 X,\dots,t_g X,\e_1\th_1 X,\dots,\e_n\th_n X)
}
\]
and denote its image by $\Hom^{\bm{\e}}_{\bm{\th}}(\grp_{g,n},\SU_2)$.
Such map has connected, irreducible domain of dimension $2g+2$
and its fibers are discrete over non-central homomorphisms
and $2$-dimensional over central ones.
As a consequence, $\Hom^{\bm{\e}}_{\bm{\th}}(\grp_{g,n},\SU_2)$ is connected, irreducible,
of dimension $2g+2$. 
Now, the union of all such $\Hom^{\bm{\e}}_{\bm{\th}}(\grp_{g,n},\SU_2)$ is exactly the coaxial locus,
which has thus pure dimension $2g+2$. Moreover, it is closed algebraic by Lemma \ref{mainlemma:alg}(ii-iii). The conclusion for $\Rep_{\bm{\th}}(\grp_{g,n},\SU_2)$ easily follows.

The proof in the $\SL_2(\CC)$ case is almost identical: it is enough to replace
$\RR^{2g}$ by $\CC^{2g}$ and to note that
$\hfrak$ must be a complex $1$-dimensional subspace of $\psl_2(\CC)$ that consists of diagonalizable matrices with imaginary eigenvalues, and that $X\in\hfrak$ must satisfy $\Kill(X,X)=1$.
\end{proof}

We can now establish the main properties of absolute representation spaces.

\begin{proof}[Proof of Theorem \ref{mainthm:rep-undec}]
(i) is essentially Lemma \ref{lemma:no-auto-undec}.

Let us first prove (ii-iii) in the $\SU_2$ case.

(ii)
Since $n>0$, the group $\grp_{g,n}$ is free on $2g+n-1\geq 2$ generators.
It follows that $\Hom(\grp_{g,n},\SU_2)\cong \SU_2^{2g+n-1}$,
and so it is a smooth algebraic set of dimension $6g+3n-3$.

(ii-a) The irreducibility of the coaxial locus in $\Hom(\grp_{g,n},\SU_2)$
is proven in Proposition \ref{prop:coaxial}(i), where its dimension $2g+n+1$ is also computed.

(ii-b) It is enough to note that the non-coaxial locus
is open by Lemma \ref{mainlemma:alg}(ii).

(ii-c) It is easy to see that the central locus is clearly $0$-dimensional
and that the general coaxial homomorphism is non-central, and so has $1$-dimensional stabilizer. It follows from (ii-a) that the coaxial locus in $\Rep(\grp_{g,n},\SU_2)$
has dimension $(2g+n+1)-(3-1)$.

(ii-d) follows from (ii-b), using
Lemma \ref{lemma:no-auto-undec} and Remark \ref{rmk:orbi}.

(iii) Connectedness of the coaxial locus is proven in Proposition \ref{prop:coaxial}(i). Connectedness and density of the non-coaxial locus is proven in Proposition \ref{prop:connected-abs} below.

Finally, the proof of (ii-iii) in the $\SL_2(\CC)$ case is analogous, provided
one notices that the action of $\PSL_2(\CC)$
on $\Hom^{nc}(\grp_{g,n},\SL_2(\CC))$
is proper (see, for example, \cite[Section 5.3.4]{labourie:notes}).
\end{proof}

\subsubsection{Explicit description of $\Hom$ in simple cases.}\label{sec:simple}

Before investigating the properties of relative homomorphism spaces,
we wish to analyze two peculiar cases, in which we 
are able to explicitly describe all homomorphisms up to conjugation.
Note that these are exactly the cases of simple triples $(g,n,\bm{\th})$,
namely those for which $k\geq 2g-2+n$, 
where is the number of integer entries of $\bm{\th}$.

We first treat simple cases of genus $0$.

\begin{proposition}[Surfaces of genus $0$ with at most $2$ non-integral angles]\label{prop:genus0<3}
Assume $g=0$ and let $k$ be the number of integer entries of
$\bm{\th}\in\RR^n_{>0}$.
\begin{itemize}
\item[(o)]
If $n-k=0$, then $\Hom_{\bm{\th}}(\grp_{0,n},\SU_2)$ is non-empty
if and only if $\sum_i (\th_i-1)$ is even: in this case 
$\Hom_{\bm{\th}}(\grp_{0,n},\SU_2)$
consists of a single point.
\item[(i)]
If $n-k=1$, then $\Hom_{\bm{\th}}(\grp_{0,n},\SU_2)$ is empty.
\item[(ii)]
Suppose $n-k=2$ and $\th_1,\th_2\notin\ZZ$.
Then $\Hom_{\bm{\th}}(\grp_{0,n},\SU_2)$ is non-empty if and only if $\dd_1=\dd_2$:
in this case $\Hom_{\bm{\th}}(\grp_{0,n},\SU_2)$ can be identified to the conjugacy
class $\cla_{\dd_1}$.
\end{itemize}
\end{proposition}
\begin{proof}
Identify $\Hom_{\bm{\th}}(\grp_{0,n},\SU_2)$ with the locus of $(B_1,\dots,B_n)$
in $\cla_{\dd_1}\times\dots\times\cla_{\dd_n}$ such that $B_1\cdots B_n=I$.

(o) This is immediate, since $B_i$ must be equal to $(-I)^{\th_i-1}$.

(i) Suppose $\th_1\notin\ZZ$. This is again immediate, since $B_2,\dots,B_n\in\{\pm I\}$, but $B_1$ can be equal to $\pm I$.

(ii) Note that $B_i=(-I)^{\th_i-1}$ for $i\geq 3$ and that $B_1\in\cla_{\dd_1}$.
It follows that $B_2$ is uniquely determined by $B_1$.
Hence the map
\[
\xymatrix@R=0in{
\Hom_{\bm{\th}}(\grp_{0,n},\SU_2)\ar[rr] && \cla_{\dd_1}\\
(B_1,\dots,B_n) \ar@{|->}[rr] && B_1
}
\]
is an isomorphism.
\end{proof}

Now we consider the simple cases in genus $1$.

\begin{proposition}[Surfaces of genus $1$ with integral angles]\label{prop:genus1}
Suppose that $g=1$ and $\bm{\th}\in\ZZ^n$.
Then
\begin{itemize}
\item[(i)]
if $\sum_i (\th_i-1)$ is odd, then $(1,n,\bm{\th})$ is not special
and every $\rho\in\Hom_{\bm{\th}}(\grp_{1,n},\SU_2)$ is conjugate
to the non-coaxial homomorphism $\rho'$ determined by
\[
\lambda(\rho')=
\left(\left(\begin{matrix}
i & 0\\ 0 & -i
\end{matrix}\right),
\left(\begin{matrix}
0 & -1\\ 1 & 0
\end{matrix}\right),
(-I)^{\th_1-1},\dots,(-I)^{\th_n-1}\right).
\]
It follows that $\Hom_{\bm{\th}}(\grp_{1,n},\SU_2)$
is isomorphic to $\PSU_2$ and that $\Rep_{\bm{\th}}(\grp_{1,n},\SU_2)$
consists of a single point.
\item[(ii)]
if $\sum_i (\th_i-1)$ is even, then $(1,n,\bm{\th})$ is special
and every $\rho\in\Hom_{\bm{\th}}(\grp_{1,n},\SU_2)$ is conjugate
to the coaxial homomorphism $\rho'_{s,t}$ determined by
\[
\lambda(\rho'_{s,t})=
\left(\left(\begin{matrix}
e^{i(s-1)\pi} & 0\\ 0 & e^{-i(s-1)\pi}
\end{matrix}\right),
\left(\begin{matrix}
e^{i(t-1)\pi} & 0\\ 0 & e^{-i(t-1)\pi}
\end{matrix}\right),
(-I)^{\th_1-1},\dots,(-I)^{\th_n-1}\right),
\]
for some $t,s\in \RR/2\ZZ$.
It follows that $\Hom_{\bm{\th}}(\grp_{1,n},\SU_2)$
has four points corresponding to central homomorphisms, and the complement
is diffeomorphic to an $\Sph$-bundle over $S^2\setminus\{\text{4 points}\}$.
As a consequence, $\Rep_{\bm{\th}}(\grp_{1,n},\SU_2)$ is homeomorphic to $S^2$
and four of its points correspond to central representations.
\end{itemize}
\end{proposition}
\begin{proof}
Recall that every $\rho\in\Hom_{\bm{\th}}(\grp_{1,n},\SU_2)$
must satisfy $\rho(\beta_i)=(-I)^{\th_i-1}$.

(i) In this case $\rho(\beta_1)\cdots\rho(\beta_n)=-I$.
It was proven in \cite[Corollary A.3]{EMP} that $\rho$ is conjugate to $\rho'$. Hence, $\Rep_{\bm{\th}}(\grp_{1,n},\SU_2)$
consists of a single point.
Since $\rho'$ is non-coaxial,
by Lemma \ref{lemma:no-auto-undec} the stabilizer of $\rho'$ is trivial,
and so the map $\PSU_2\ni g\mapsto g\rho' g^{-1}\in \Hom_{\bm{\th}}(\grp_{1,n},\SU_2)$
is an isomorphism.

(ii) In this case $\rho(\beta_1)\cdots\rho(\beta_n)=I$,
and so $\rho(\mu_1),\rho(\nu_1)$ commute and can be simultaneously diagonalized.
It is immediate that $\rho$ must be conjugate to some $\rho'_{s,t}$.

As for the description of $\Hom_{\bm{\th}}(\grp_{1,n},\SU_2)$, 
note first that the four central homomorphisms
are parametrized by the values $\pm I$ taken at $\mu_1,\nu_1$.
Now, identify $\Sph$ with the unit sphere in $\su_2$
and let $\mathcal{X}:=
\Sph\times (\RR^2\setminus\ZZ^2)/(2\ZZ)^2$.
Note that $\{\pm 1\}$ acts on $\mathcal{X}$ by multiplication,
namely $(-1)\cdot (\hat{X},[s,t]):=(-\hat{X},[-s,-t])$
for every element $(\hat{X},[s,t])$ of $\mathcal{X}$,
and consider the map
\[
\xymatrix@R=0in{
\mathcal{X} \ar[rr]&& \Hom_{\bm{\th}}(\grp_{1,n},\SU_2)\\
(\hat{X},[s,t])\ar@{|->}[rr] && \bm{\ee}(s\hat{X},t\hat{Y})
}
\]
Such map factors through the quotient
$\mathcal{X}/\{\pm 1\}$, and in fact it 
sends $\mathcal{X}/\{\pm 1\}$ isomorphically onto
the locus of non-central homomorphisms.
Note that $(\RR^2\setminus\ZZ^2)/(2\ZZ)^2$ is diffeomorphic to
a $2$-torus with four points removed
and that its quotient by $\{\pm 1\}$ is diffeomorphic to
a $2$-sphere with four points removed, which we denote by
$\dot{\Sigma}_{0,4}$.
It follows that the projection
$\mathcal{X}/\{\pm 1\}\rar \dot{\Sigma}_{0,4}$ is an $\Sph$-bundle.
Finally, fix $\hat{X}_D:=\mathrm{diag}(i,-i)$ and consider
the closed locus in $\Hom_{\bm{\th}}(\grp_{1,n},\SU_2)$
corresponding to $(\ee(s\hat{X}_D),\ee(t\hat{X}_D))=\rho'_{s,t}$, which is diffeomorphic
to $\RR^2/(2\ZZ)^2$. Its stabilizer under the conjugacy action
of $\SU_2$ is generated by the subgroup of diagonal matrices, and by
$\left(\begin{smallmatrix}0 & -1\\ 1 & 0\end{smallmatrix}\right)$
that acts by multiplication by $-1$ on $\RR^2/(2\ZZ)^2$.
Hence, $\Rep_{\bm{\th}}(\grp_{1,n},\SU_2)$ is homeomorphic
to $(\RR^2/(2\ZZ)^2)/\{\pm 1\}$, which is homeomorphic to $S^2$;
moreover the four points $[0,0],[1,0],[0,1],[1,1]$ correspond
to central representations.
\end{proof}

\subsection{Connectedness and monotone connectedness}\label{sec:m-c}

In this short section we recall two criteria for the connectedness
of the total space and of the fibers of a fibration.
We also introduce monotone connected pairs,
that will be useful to study relative homomorphism spaces in genus $0$.

The following statement is rather standard: a proof is included for completeness.

\begin{lemma}[Surjective maps and connectedness]\label{lemma:connected}
Let $f:X\rar Y$ be a continuous, surjective map between locally finite CW-complexes.
\begin{itemize}
\item[(i)]
Suppose that $f$ is proper or open, and that $f$ has connected fibers.\\
Then $X$ is connected $\iff$ $Y$ is connected.
\item[(ii)]
Suppose that $X$ is connected, $Y$ is simply-connected and $f$ is a fiber bundle.\\
Then all fibers of $f$ are connected.
\end{itemize}
\end{lemma}
\begin{proof}
(i) Since $f$ is surjective, $X$ connected $\implies$ $Y$ connected.
Vice versa, suppose that $Y$ is connected and let $U,V\subseteq X$
be disjoint, open subsets such that $U\cup V=X$ (which implies that $U,V$ are also closed). 
We want to show that either $U$ or $V$ is empty.

Since the fibers of $f$ are connected, each fiber of $f$ is completely contained either in $U$ or in $V$. It follows that $f(U)\cap f(V)=\emptyset$.
Since $f$ is surjective, $f(U)\cup f(V)=Y$. We separately consider the
following two cases.

Suppose first that $f$ is open. Then $f(U)$ and $f(V)$ are open.
Since $Y$ is connected, either $f(U)$ or $f(V)$ is empty.
As a consequence, either $U$ or $V$ must be empty, and so $X$ is connected.

Suppose now that $f$ is proper. 
Then $f$ is closed because $Y$ is locally compact and Hausdorff.
It follows that $f(U)$ and $f(V)$ are closed.
Since $Y$ is connected, either $f(U)$ or $f(V)$ must be empty.
Again, either $U$ or $V$ must be empty and so $X$ is connected.
%
%
%

(ii) Pick $x\in X$, $y=f(x)\in Y$ and let $F=f^{-1}(y)$.
Then the exact sequence in homotopy for the fibration $(F,x)\rar (X,x)\rar (Y,y)$
gives $\pi_1(Y,y)\rar \pi_0(F,x)\rar \pi_0(X,x)$.
Since $\pi_1(Y,y)$ and $\pi_0(X,x)$ are trivial by hypothesis, $\pi_0(F,x)$ is too
and so $F$ is connected.
\end{proof}

In order to establish the connectedness of 
the non-coaxial locus $\Hom_{\bm{\th}}(\grp_{0,n},\SU_2)$,
we will use the formalism of monotone-connected pairs here presented.

\begin{definition}[Monotone-connected pair]
Consider a CW-complex $X$ with a continuous function $f:X\rar\RR$. A path $\gamma:[0,1]\to X$ is called \emph{monotone} if the function $f(\gamma(t))$ is monotonic. We say that the pair $(X,f)$ is {\it monotone-connected} if any two points $x,y\in X$ can be joined by a monotone path.
\end{definition}

Note that, according to this definition, if $(X,f)$ is monotone-connected, then the level sets $\{f=c\}$ are connected. Our aim is to show the following.

\begin{proposition}[Criterion for monotone-connectedness]\label{analyticmonotone} Let $M$ be a connected real-analytic variety and $f$ be an analytic function on $M$. Then $(M,f)$ is monotone connected if and only if each level set $f^{-1}(c)$ is connected. 
\end{proposition}

We start with a simpler claim.

\begin{lemma}[Sufficient conditions for monotone-connectedness]\label{easymonotonecrit} Suppose $X$ is a compact CW-complex, $f$ is a continuous function on $X$, and each level set $\{f=c\}$ is connected. Then $(X,f)$ is monotone-connected if one of the following holds.
\begin{itemize}
\item[(i)]
The exists a monotone path $\gamma\subset X$ that connects a point where $f$ attains its minimum $\min(f)$ with a point where $f$ attains its maximums $\max(f)$.
%
%
\item[(ii)] 
For any $c\in [\min f, \max f]$ there is $\varepsilon>0$  and a monotone path $\gamma_{c,\varepsilon}\subset X$ on which $f$ attains all values from the interval $[\min f, \max f]\cap [c-\varepsilon,c+\varepsilon]$.
\end{itemize}
\end{lemma}
\begin{proof} 
(i) Since each level set of $f$ is connected, for any $x,y\in X$ we can choose a path
that first joins $x$ to a point of $\gamma$ inside the level set $\{f=f(x)\}$, then follows $\gamma$ and then connects to $y$ in the level set $\{f=f(y)\}$.

(ii) The interval $[\min f, \max f]$ is covered by sub-intervals for which monotone paths $\gamma_{c,\varepsilon}$ exist. We can choose a finite sub-cover and construct a monotone  $\gamma$ as in (1) that follows the corresponding finite collection of paths and jumps from one to another along level sets of $f$.
\end{proof} 

Now we can prove our criterion.

\begin{proof}[Proof of Proposition \ref{analyticmonotone}]
{\it ``Only if'' direction.} Assume by contrapositive that a level set $f^{-1}(c)$ is not connected. Then any two points $x,y\in f^{-1}(c)$ that lie in different connected components of $f^{-1}(c)$ cannot be connected by a monotone path in $M$. 

{\it ``If'' direction.} It will be enough to show that condition (ii) of Lemma \ref{easymonotonecrit} holds. 
We will prove a half of this statement: namely, that there exists a monotone path
$\gamma_{c,\varepsilon}^-$ that attains all values in the interval $[\min f, c]\cap[c-\varepsilon,c]$.  The other half is proven identically. 

We can assume $c>\min f$. The subset $M_{<c}\subset M$ where $f<c$ is semi-analytic. Since $M$ is connected, there is a point $x\in f^{-1}(c)$ that lies in the closure of $M_{<c}$. By the real-analytic curve selection lemma 
\cite[Paragraph~3]{milnor:singular}, there is a real-analytic map $\gamma: [0,1]\to M$, such that $\gamma(0)=x$, and $f(\gamma(t))<c$ for $t>0$. Since $f(\gamma(t))$ is an analytic function, it is monotonic for $t$ small enough. So we choose $\gamma_{c,\varepsilon}^-$ as a sub-path of $\gamma$.
\end{proof}

%
%
\subsection{First-order computations: commutator map and product map}\label{sec:comm-prod}

It order to analyze the maps $R$ and $R_{\bm{\th}}$,
we first compute the differentials of the commutator map and of the product map.

The {\it{commutator map}} is defined as follows\index{$\comm$, $\Comm$}
\[
\comm:
\xymatrix@R=0in{
\SU_2\times\SU_2\ar[rr] && \SU_2\\
(M,N)\ar@{|->}[rr] && [M,N]
}
\]
We also let $\Comm: (\SU_2\times\SU_2)^g\lra \SU_2$ be
defined as $\Comm(\bm{M},\bm{N}):=\comm(M_1,N_1)\cdots \comm(M_g,N_g)$.

The following computation is already contained in Example 3.7 in \cite{goldman:symplectic}.
We reproduce here for completeness.

\begin{lemma}[Differential of a product of commutators]\label{sub:c}
The image of the differential $d\Comm$
at the point $(\bm{M},\bm{N})$ is $\Zfrak({M}_1,\dots,{N}_g)^\perp$,
namely the orthogonal in $\su_2$ to the infinitesimal centralizer
of the set $\{M_1,\dots,N_g\}$.
\end{lemma}
\begin{proof}
Fix ${M}_1,\dots,{N}_g\in \SU_2$ and let
$\dot{V}_i,\dot{W}_i\in \su_2$ for $i=1,\dots,g$.
Consider the paths
\[
M_i(t)=\exp(t\dot{V}_i)M_i=(I+t\dot{V}_i+o(t)){M}_i,
\qquad N_i(t)=\exp(t\dot{W}_i){N}_i=(I+t\dot{W}_i+o(t)){N}_i
\]
in $\SU_2$. Under our conventions,
the tangent vector $\frac{d}{dt}M_i(t)|_{t=0}$ (resp. $\frac{d}{dt}N_i(t)|_{t=0}$) is identified
to the element $\dot{V}_i$ (resp. $\dot{W}_i$) of $\su_2$.
Inside the space of complex $2\times 2$ matrices, we now compute
\begin{align*}
\frac{d}{dt} \comm(M_1(t),N_1(t))\big|_{t=0}
&=
\dot{V}_1 {M}_1 {N}_1 {M}_1^{-1}{N}_1^{-1}+{M}_1\dot{W}_1{N}_1 {M}_1^{-1}{N}_1^{-1}-{M}_1{N}_1{M}_1^{-1}\dot{V}_1{N}_1^{-1}-
{M}_1 {N}_1{M}_1^{-1}{N}_1^{-1}\dot{W}_1=\\
&=\left( \dot{V}_1+{M}_1 \dot{W}_1 {M}_1^{-1}-({M}_1 {N}_1 {M}_1^{-1})\dot{V}_1({M}_1 {N}_1 {M}_1^{-1})^{-1}\right.\\
&\qquad\qquad\qquad\qquad\left.
-({M}_1{N}_1{M}_1^{-1}{N}_1^{-1})\dot{W}_1({M}_1{N}_1{M}_1^{-1}{N}_1^{-1})^{-1}\right)[{M}_1,{N}_1]=\\
&=\left( (I-\Ad_{{M}_1{N}_1^{-1}{M}_1^{-1}})\dot{V}_1+\Ad_{{M}_1}\circ (I-\Ad_{{N}_1 {M}_1 {N}_1^{-1}})\dot{W}_1\right)\comm({M}_1,{N}_1).
\end{align*}
Hence, $\IM(d\comm)_{({M}_1,{N}_1)}=\IM(I-\Ad_{{M}_1{N}_1^{-1}{M}_1^{-1}})+
\IM(\Ad_{{M}_1}\circ (I-\Ad_{{N}_1 {M}_1 {N}_1^{-1}}))$.
Note that
\[
\IM(I-\Ad_{{M}_1{N}_1^{-1}{M}_1^{-1}})=
\ker(I-\Ad_{{M}_1{N}_1^{-1}{M}_1^{-1}})^\perp=\Zfrak({M}_1 {N}_1^{-1}{M}_1^{-1})^\perp=\Ad_{M_1}(\Zfrak(N_1)^\perp).
\]
On the other hand,
\[
\IM(\Ad_{{M}_1}\circ (I-\Ad_{{N}_1 {M}_1 {N}_1^{-1}}))=
\Ad_{M_1}(\ker(I-\Ad_{{N}_1 {M}_1 {N}_1^{-1}})^\perp)=
\Ad_{M_1}(\Zfrak({N}_1 {M}_1 {N}_1^{-1})^\perp)
\]
and so
\begin{equation}\label{eq:dc1}
\IM(d\comm)_{({M}_1,{N}_1)}=
\Ad_{M_1}(\Zfrak({N}_1 {M}_1 N_1^{-1},N_1)^\perp)=
\Ad_{M_1}(\Zfrak(M_1,N_1))^\perp=\Zfrak(M_1,N_1)^\perp
\end{equation}
where the second equality depends on the fact
that the subgroup generated by $\{N_1 M_1N_1^{-1},\,N_1\}$
agrees with the one generated by $\{M_1,\,N_1\}$,
and the third equality depends on the fact that $\Ad_{M_1}$
is an isometry of $\su_2$ and fixes $\Zfrak(M_1,N_1)$.

Now denote $\comm(M_i(t),N_i(t))$ by $\comm_i$ and
$\Comm(M_1(t),\dots,N_g(t))$ simply by $\Comm$,
and let $\dot{\comm}_i$ and $\dot{\Comm}$ be their derivative at $t=0$.
Observe that
\[
\dot{\Comm} =(\dot{\comm}_1 \comm_1^{-1})\cdot \Comm+\comm_1(\dot{\comm}_2 \comm_2^{-1})\comm_1^{-1}\cdot \Comm+
\comm_1\comm_2(\dot{\comm}_3 \comm_3^{-1})\comm_2^{-1}\comm_1^{-1}\cdot \Comm+\dots=\left(\sum_{i=1}^g \Ad_{\hat{\comm}_{i-1}}(\dot{\comm}_i \comm_i^{-1})\right)\cdot \Comm
\]
where $\hat{\comm}_0=I$ and $\hat{\comm}_j=[M_1,N_1]\cdot[M_2,N_2]\cdots [M_j,N_j]$ for $j=1,\dots,g-1$.
Thus,
\begin{align*}
\IM(d\Comm_{(\bm{M},\bm{N})} ) &=\sum_{i=1}^g \Ad_{\hat{\comm}_{i-1}}(\IM(d\comm_i)_{(M_i,N_i)})=
\sum_{i=1}^g \Zfrak(\Ad_{\hat{\comm}_{i-1}}(M_i),\Ad_{\hat{\comm}_{i-1}}(N_i))^\perp=\\
&=\left(\bigcap_{i=1}^g \Zfrak\big(\Ad_{\hat{\comm}_{i-1}}(M_i),\Ad_{\hat{\comm}_{i-1}}(N_i)\big)\right)^\perp
\end{align*}
where the second equality relies on \eqref{eq:dc1}.
It follows that $\IM(d\Comm_{(\bm{M},\bm{N})})=\Zfrak(H)^\perp$,
where $H<\SU_2$ is the subgroup generated by
\[
\{{M}_1,{N}_1,\Ad_{[{M}_1,{N}_1]}{M}_2,\Ad_{[{M}_1,{N}_1]}{N}_2,\Ad_{[{M}_1,{N}_1][{M}_2,{N}_2]}{M}_3,\dots\}
\]
Such $H$
agrees with the subgroup generated by $\{\bm{M},\bm{N}\}$.
Hence $\IM(d\Comm_{(\bm{M},\bm{N})})=\Zfrak(\bm{M},\bm{N})^\perp$, as desired.
\end{proof}

A first consequence of the computation in Lemma \ref{sub:c}
is the following.

\begin{corollary}[Surjectivity and connectedness of the commutator map]\label{cor:comm}
The commutator map $\comm$ is proper, algebraic, surjective.
Moreover the following hold.
\begin{itemize}
\item[(i)]
$\comm^{-1}(I)$ is connected of dimension $4$
and $\comm^{-1}(I)\setminus (H\times H)$ is connected too for every $1$-parameter
subgroup $H$ of $\SU_2$.
\item[(ii)]
All the fibers of $\comm$ different from $\comm^{-1}(I)$
are smooth, connected of dimension $3$,
and they are all isomorphic to each other.
\item[(iii)]
For all $d\in[0,1]$ the preimage
$\comm^{-1}(\cla_d)$ of the conjugacy class $\cla_d$ is connected. 
\end{itemize}
\end{corollary}
\begin{proof}
Properness and algebraicity are obvious.
Consider now $B_0=\begin{pmatrix}0 & -1\\ 1 & 0\end{pmatrix}$
and $B=\begin{pmatrix}z & 0\\ 0 & z^{-1}\end{pmatrix}$,
where $|z|=1$.
A direct computation gives $[B,B_0]=B^2$. 

It follows that every diagonal matrix of $\SU_2$
is in the image of the commutator map.
Since such image is invariant under conjugation,
it must be the whole $\SU_2$ and so $\comm$ is surjective.

(i) Consider the map $w:\Sph\times[0,2)^2\rar\SU_2\times\SU_2$
that sends $(X,t,s)$ to $(\exp(2\pi tX),\exp(2\pi sX))$. Its image is
$\comm^{-1}(I)$. Moreover, the restriction of $w$
to $\Sph\times \left( (0,1)\cup(1,2)\right)^2$ is injective.
Hence, $\comm^{-1}(I)$ has dimension $4$.
If $H=\mathrm{exp}(\hfrak)$ for a certain
$1$-dimensional subspace $\hfrak\subset\su_2$, then
$\comm^{-1}(I)\setminus (H\times H)$ is the image of the restriction of $w$ to
$(\Sph\setminus\hfrak)\times(0,2)^2$, and so is connected.

(ii)
Consider the restriction of the commutator map to
$(\SU_2\times\SU_2)\setminus \comm^{-1}(I)\rar \SU_2\setminus\{I\}$.
Such restriction is proper, surjective and it is submersive by 
Lemma \ref{sub:c}. Hence such restriction is a fiber bundle
and so the fibers have dimension $6-3=3$.
Since $(\SU_2\times\SU_2)\setminus \comm^{-1}(I)$ is connected
and $\SU_2\setminus\{I\}$ is simply connected,
the fibers over an element different from $I$
are connected by Lemma \ref{lemma:connected}(ii).

(iii) The case $d=0$ follows from (i).
For $d\in(0,1]$
it follows from the discussion in (ii) that $\comm^{-1}(\cla_d)\rar\cla_d$ is a fiber bundle with connected base and fibers, and so it is connected
by Lemma \ref{lemma:connected}(i).
\end{proof}

We now consider the {\it{product map}}
\[
\PR:
\xymatrix@R=0in{
\SU_2^n\ar[rr] && \SU_2\\
(B_1,\dots,B_n)\ar@{|->}[rr] && B_1 B_2\cdots B_n.
}
\]
Given an angle vector $\bm{\th}$, we also denote by $\PR_{\bm{\th}}$
the restriction of $\PR$ to 
$\cla_{\bm{d}}:=\cla_{\dd_1}\times\dots\times \cla_{\dd_n}$.
Note that $T_{B_i} \cla_{\dd_i}=\Zfrak(B_i)^\perp$.

\begin{lemma}[Differential of $P$]\label{sub:P}
Let $m: \SU_2\times \SU_2\rar \SU_2$ be the multiplication map $m(Q_1,Q_2):=Q_1\cdot Q_2$. 
\begin{itemize}
\item[(i)]
Upon identifying the tangent spaces of $\SU_2$ with $\su_2$ as in Section \ref{sec:conventions}, we have
\[
dm_{(Q_1,Q_2)}(\dot{V}_1,\dot{V}_2)=\dot{V}_1+\Ad_{Q_1}(\dot{V}_2).
\]
\item[(ii)]
The differential of $P$ satisfies
\[
d\PR_{\bm{B}}(\bm{\dot{X}})=
\dot{X}_1+\sum_{i=2}^n 
\Ad_{B_1\cdots B_{i-1}}(\dot{X}_i).
\]
where $\bm{B}=(B_1,\dots,B_n)$ and $\bm{\dot{X}}=(\dot{X}_1,\dots,\dot{X}_n)$.
\item[(iii)]
If $\bm{B}\in \cla_{\bm{d}}$, then 
$\mathrm{Im}(d\PR)_{\bm{B}} =\su_2$ and
$\mathrm{Im}(d\PR_{\bm{\th}})_{\bm{B}} =\Zfrak(\bm{B})^\perp$.
\end{itemize}
\end{lemma}
\begin{proof}
(i) Consider the paths $Q_1(t)=\exp(t\dot{V}_1)Q_1$ and $Q_2(t)=\exp(t\dot{V}_2)Q_2$.
Then $Q_1(t)Q_2(t)=(I+t\dot{V}_1)Q_1(I+t\dot{V}_2)Q_2+o(t)=Q_1Q_2+t
(\dot{V}_1Q_1Q_2+Q_1\dot{V}_2Q_2)+o(t)=
Q_1Q_2+t(\dot{V}_1+Q_1\dot{V}_2 Q_1^{-1})Q_1Q_2+o(t)$ and the conclusion follows.

(ii) is obtained by iterating (i).

(iii) It immediately follows from (ii) that $d\PR_{\bm{B}}$ is surjective.
As for the image of $(d\PR_{\bm{\th}})_{\bm{B}}$, we have three cases.

If all $B_i=\pm I$, then clearly such image is $\{0\}$.

Suppose now that all $B_i$ belong to a $1$-parameter subgroup $H=\exp(\hfrak)$ but not all of them are $\pm I$.
Since $T_{B_i}\cla_{\dd_i}\subseteq \hfrak^\perp$, then
$\mathrm{Im}(d\PR_{\bm{\th}})_{\bm{B}}\subseteq \hfrak^\perp$.
If $k$ is the smallest index so that $B_k\neq \pm I$,
then $(d\PR_{\bm{\th}})_{\bm{B}}(0,\dots,0,\dot{X}_k,0,\dots,0)=\dot{X}_k$
for all $\dot{X}_k\in\hfrak^\perp$ by (ii), and so
$\mathrm{Im}(d\PR_{\bm{\th}})_{\bm{B}}=\hfrak^\perp$.

Finally, if $B_1,\dots,B_n$ are not contained in a $1$-parameter subgroup, then let $j$ be the smallest index so that $B_j\neq \pm I$
and let $k>j$ be the smallest index so that $\{B_j,B_k\}$ are not contained in the same $1$-parameter subgroup.
Then $(d\PR_{\bm{\th}})_{\bm{B}}(0,\dots,0,\dot{X}_j,0,\dots,0)=\dot{X}_j$
for all $\dot{X}_j\in\Zfrak(B_j)^\perp$, and so
$\mathrm{Im}(d\PR_{\bm{\th}})_{\bm{B}}\supseteq \Zfrak(B_j)^\perp$.
Let now $\dot{X}_k$ be in $\Zfrak(B_k)^\perp$ but not in
$\Zfrak(B_j)^\perp$. Then
$(d\PR_{\bm{\th}})_{\bm{B}}(0,\dots,0,\dot{X}_k,0,\dots,0)\notin
\Zfrak(B_j)^\perp$ by (ii), and so $\mathrm{Im}(d\PR_{\bm{\th}})_{\bm{B}}\neq \Zfrak(B_j)^\perp$, which implies that $(d\PR_{\bm{\th}})_{\bm{B}}$
is surjective.
\end{proof}


\subsection{First-order computations: the maps $R$ and $R_{\bm{\th}}$}\label{sec:first-order}

Let us explicitly deal with homomorphisms in $\SU_2$, as the $\SL_2(\CC)$ case will be analogous.
Recall the definitions of the smooth algebraic varieties $\Gcal$ and $\Gcal_{\bm{\th}}$ and of the algebraic maps $R$, $R_{\bm{\th}}$ introduced in Section \ref{ssc:algebraic}.

Since the image of $\Hom_{\bm{\th}}(\grp_{g,n},\SU_2)$
via $\lambda$ is defined by 
the equation $R_{\bm{\th}}=I$ 
inside $\Gcal_{\bm{\th}}$, by the implicit function theorem 
its smooth locus is detected by looking
at points at which $dR_{\bm{\th}}$ has locally constant rank.
Since in non-special cases $dR_{\bm{\th}}$ will have generically fully rank
(this will follows from Proposition \ref{prop:IFT-undec}(i)
and Theorem \ref{thm:density-non-dec}), a point $\rho\in\Hom_{\bm{\th}}(\grp_{g,n},\SU_2)$
will be smooth if and only if
$dR_{\bm{\th}}$ is surjective at $\lambda(\rho)$.
Analogous considerations hold for
$\Hom(\grp_{g,n},\SU_2)$, whose image via $\lambda$ 
is described by
the equation $R=I$ inside $\Gcal$.

For the above reason, we begin by analysing the ranks of the differentials of $R$ and $R_{\bm{\th}}$.

We remind that $\Zfrak(\rho)$ is the infinitesimal centralizer of the image of $\rho$.

\begin{proposition}[Differentials of $R$ and $R_{\bm{\th}}$]\label{prop:IFT-undec}
For every $\rho\in\Hom_{\bm{\th}}(\grp_{g,n},\SU_2)$,
the images of the differentials
of $R:\Gcal\rar\SU_2$ and $R_{\bm{\th}}:\Gcal_{\bm{\th}}\rar\SU_2$
at $\lambda(\rho)$ are
\begin{itemize}
\item[(i)]
$\mathrm{Im}(dR_{\bm{\th}})_{\lambda(\rho)}=\Zfrak(\rho)^\perp$;
\item[(ii)]
$dR$ is sujective at $\lambda(\rho)$.
\end{itemize}
%
\end{proposition}

Observe that, if we identify $\Hom(\grp_{g,n},\SU_2)$ with its image inside $\Gcal$ via $\lambda$,
then
\[
R(\bm{M},\bm{N},\bm{B})=\Comm(\bm{M},\bm{N})\cdot \PR(\bm{B}).
\]
Such observation will play an important role in the below computation, and motivates
why we analyzed the differentials of the maps $\Comm$ and $\PR$
in Section \ref{sec:comm-prod}.

\begin{proof}[Proof of Proposition \ref{prop:IFT-undec}]
%
First of all, note that
by Lemma \ref{sub:P}(i) the differential $dR$ consists of two summands
\[
dR_{(\bm{M},\bm{N},\bm{B})}(\bm{\dot{V}},\bm{\dot{W}},\bm{\dot{X}})=
d\Comm_{(\bm{M},\bm{N})}(\bm{\dot{V}},\bm{\dot{W}})+
\Ad_{\Comm(\bm{M},\bm{N})}
d\PR_{\bm{B}}(\bm{\dot{X}})
\]
and so
\begin{equation}\label{eq:dR-image-undec}
\IM(dR_{(\bm{M},\bm{N},\bm{B})})=\IM(d\Comm_{(\bm{M},\bm{N})})+
\Ad_{\Comm(\bm{M},\bm{N})}(\IM(d\PR_{\bm{B}})).
\end{equation}
Analogously for $dR_{\bm{\th}}$ we have
\begin{equation}\label{eq:dR_th-image-undec}
\IM(dR_{\bm{\th}})_{(\bm{M},\bm{N},\bm{B})}=\IM(d\Comm_{(\bm{M},\bm{N})})+
\Ad_{\Comm(\bm{M},\bm{N})}(\IM(d\PR_{\bm{\th}})_{\bm{B}}).
\end{equation}
Let now $(\bm{M},\bm{N},\bm{B})=\lambda(\rho)$.

(i) In order to compute the image of $dR_{\bm{\th}}$ at
$(\bm{M},\bm{N},\bm{B})$, we separately consider three cases.

If $\dim\,\Zfrak(\bm{M},\bm{N})=0$, then 
$\IM(d\Comm_{(\bm{M},\bm{N})})=\su_2$ by Lemma \ref{sub:c}.
Thus, $\IM(dR_{\bm{\th}})_{(\bm{M},\bm{N},\bm{B})}=\su_2$
by \eqref{eq:dR_th-image-undec}.

Suppose that $\Zfrak(\bm{M},\bm{N})=\mathfrak{h}$ is $1$-dimensional,
and let $H=\exp(\mathfrak{h})$. Then 
$\Comm(\bm{M},\bm{N})=I$ and
$\IM(d\Comm_{(\bm{M},\bm{N})})=\hfrak^\perp$.
By Lemma \ref{sub:P}(iii),
the image of $(d\PR_{\bm{\th}})_{\bm{B}}$ is $\Zfrak(\bm{B})^\perp$.
If $B_1,\dots,B_n\in H$, then $\Zfrak(\bm{B})\supseteq \hfrak$
and so $\IM(d\PR_{\bm{\th}})_{\bm{B}}\subseteq\hfrak^\perp$;
then it follows from \eqref{eq:dR_th-image-undec} that 
$\IM(dR_{\bm{\th}})_{(\bm{M},\bm{N},\bm{B})}=\hfrak^\perp$.
If some $B_i\notin H$, then $\Zfrak(\bm{B})\not\supseteq\hfrak$
and so $\IM(d\PR_{\bm{\th}})_{\bm{B}}\not\subseteq\hfrak^\perp$;
then it follows from \eqref{eq:dR_th-image-undec} that 
$\IM(dR_{\bm{\th}})_{(\bm{M},\bm{N},\bm{B})}=\su_2$.

Suppose now that $\dim\, \Zfrak(\bm{M},\bm{N})=3$, and so all $M_i,N_i=\pm I$.
Thus the image of
$(dR_{\bm{\th}})_{(\bm{M},\bm{N},\bm{B})}$ is equal to the image
of $(d\PR_{\bm{\th}})_{\bm{B}}$ by \eqref{eq:dR_th-image-undec}, and the result follows from
Lemma \ref{sub:P}(iii).

In all cases the image of $dR_{\bm{\th}}$ at $\lambda(\rho)$ agrees with $\Zfrak(\rho)^\perp$.\\

(ii)
Since $n>0$, the group $\grp_{g,n}$ is a free group on $\mu_1,\nu_1,\dots,\mu_g,\nu_g,\beta_1,\dots,\beta_{n-1}$.
It easily follows that $dR$ is surjective.
\end{proof}

Note that the statement of
Proposition \ref{prop:IFT-undec}
still holds if we replace $\SU_2$ by $\SL_2(\CC)$: the proof is identical.

\subsection{Tangent spaces to homomorphism spaces}

Using Proposition \ref{prop:IFT-undec} we can compute the dimensions of the tangent spaces to our homomorphism spaces.

\begin{corollary}[Tangent spaces to homomorphisms spaces]\label{cor:rep-undec}
Fix $\bm{\th}$ and let $k$ be the number of integer entries of $\bm{\th}$.
\begin{itemize}
\item[(i)]
The Zariski tangent spaces satisfy
\begin{align*}
\dim\, T_{\rho}\Hom_{\bm{\th}}(\grp_{g,n},\SU_2) & = 
6g-3+2(n-k)+\dim(Z(\rho))\\
\dim\, T_{\rho_\CC}\Hom_{\bm{\th}}(\grp_{g,n},\SL_2(\CC)) &= 
6g-3+2(n-k)+\dim(Z(\rho_\CC)).
\end{align*}
\item[(ii)]
The non-coaxial locus
$\Hom_{\bm{\th}}^{nc}(\grp_{g,n},\SU_2)$
is an oriented manifold of real dimension $6g-3+2(n-k)$,
and $\Hom_{\bm{\th}}^{nc}(\grp_{g,n},\SL_2(\CC))$
is a complex manifold of complex dimension $6g-3+2(n-k)$.
\item[(iii)]
The map $\ol{\bm{\Theta}}:\Hom^{nc}(\grp_{g,n},\SU_2)\rar [0,1]^n$\index{$\ol{\bm{\Theta}}$}
that sends $\rho$ to $(D_I(\rho(\beta_1)),\dots,D_I(\rho(\beta_n)))$ is continuous;
moreover, it is real-analytic and submersive (i.e.~has surjective differential)
over $(0,1)^n$.
\end{itemize}
\end{corollary}
\begin{proof}
(i)
We recall that $\lambda$ embeds $\Hom_{\bm{\th}}(\grp_{g,n},\SU_2)$
inside the smooth, oriented variety $\Gcal_{\bm{\th}}$
of dimension $6g+2(n-k)$, and that
$\lambda(\Hom_{\bm{\th}}(\grp_{g,n},\SU_2))$ is defined
by $R_{\bm{\th}}=I$ inside $\Gcal_{\bm{\th}}$.
Thus the tangent space to $\lambda(\Hom_{\bm{\th}}(\grp_{g,n},\SU_2))$
at $\lambda(\rho)=(\bm{M},\bm{N},\bm{B})$ is the kernel of
$(dR_{\bm{\th}})_{(\bm{M},\bm{N},\bm{B})}$.
Since the image of $(dR_{\bm{\th}})_{(\bm{M},\bm{N},\bm{B})}$
has dimension $\dim(\Zfrak(\rho)^\perp)=3-\dim(Z(\rho))$ 
by Proposition \ref{prop:IFT-undec}(i),
it follows that the tangent space to $\Hom_{\bm{\th}}(\grp_{g,n},\SU_2)$ at $\rho$
has dimension $6g+2(n-k)-3+\dim(Z(\rho))$.

(ii) By Proposition \ref{prop:IFT-undec}(i) the map $R_{\bm{\th}}$ is submersive
at non-coaxial representations, and so $\Hom^{nc}_{\bm{\th}}(\grp_{g,n},\SU_2)$ is a real manifold of dimension $6g-3+2(n-k)$.
Since $\SU_2$ and so $\Gcal$ are oriented,
$\Hom^{nc}_{\bm{\th}}(\grp_{g,n},\SU_2)$ is oriented.
The proof for $\SL_2(\CC)$ is analogous.

(iii) Continuity and analyticity properties of $\ol{\bm{\Theta}}$ follows from those of $D_I$ (see Lemma \ref{lemma:D_I}). So we focus on submersiveness.

Let $\rho$ be any homomorphism in $\Hom^{nc}_{\bm{\th}}(\grp_{g,n},\SU_2)$ with no integral $\th_i$.
To show that $\ol{\bm{\Theta}}$ is submersive at $\lambda(\rho)=
(\bm{M'},\bm{N'},\bm{B'})$
it is enough to show that the map
$\Gcal\rar \SU_2\times [0,1]^n$
that sends $(\bm{M},\bm{N},\bm{B})$ to $(R(\bm{M},\bm{N},\bm{B}),D_I(B_1),\dots,D_I(B_n))$ is smooth and submersive at $\lambda(\rho)$.
(We are just using that, given linear maps $f_k:W\rar W_k$ with $k=1,2$
such that $(f_1,f_2):W\rar W_1\times W_2$ is surjective, the restriction
of $f_2$ to $\ker(f_1)$ is surjective.)

By Lemma \ref{lemma:D_I}(i)
the map $(D_I)^n:\SU_2^n\mapsto [0,1]^n$ at $\bm{B'}$
is submersive at $\bm{B'}$, since $B'_i\neq \pm I$ for all $i$.
It follows that $\Gcal\rar [0,1]^n$ given by 
$(\bm{M},\bm{N},\bm{B})\mapsto(D_I(B_1),\dots,D_I(B_n))$ is 
submersive at $\lambda(\rho)$.
Since $R_{\bm{\th}}:\Gcal_{\bm{\th}}\rar\SU_2$ is submersive at $\lambda(\rho)$ by
Proposition \ref{prop:IFT-undec}(i), this  implies that the above map
$\Gcal\rar\SU_2\times[0,1]^n$ is submersive at $\lambda(\rho)$.
\end{proof}

Now we are ready to complete the investigation of the relative representation spaces.

\begin{proof}[Proof of Theorem \ref{mainthm:rep-undec-rel}]
(i) follows from Proposition \ref{prop:coaxial}(ii).

(ii) follows from Corollary \ref{cor:rep-undec}(ii).

(iii) Since the central locus is $0$-dimensional, the general coaxial homomorphism is non-central and so has $1$-dimensional stabilizer. By (i), it follows that
the coaxial locus in $\Rep_{\bm{\th}}(\grp_{g,n},\SU_2)$
has pure dimension $(2g+2)-(3-1)=2g$.

(iv)
follows from (ii) by Remark \ref{rmk:orbi}
and Lemma \ref{lemma:no-auto-undec}(ii).
\end{proof}

We postpone the deeper investigation of the non-coaxial locus of
non-special relative homomorphism spaces (namely, the proof of Theorem \ref{mainthm:rep-undec-nonsp}) to the end of Section \ref{sec:connectedness}.
Here we consider the special cases.

\begin{proof}[Proof of Theorem \ref{mainthm:special-undec}]
(i) is proven in Proposition \ref{prop:genus1}(ii).

(ii) The case $k=n-1$ is clear.
The case $d_1(\bm{\th}-\bm{1},\ZZ^n_o)<1$
follows from \cite[Theorem A]{MP1},
or from Proposition \ref{prop:conn-genus0}(o)
proven below.

(iii)
By Proposition \ref{prop:conn-genus0}(i) proven below,
the space $\Hom_{\bm{\th}}(\grp_{0,n},\SU_2)$ consists of a single conjugacy class.
Hence, $\Rep_{\bm{\th}}(\grp_{0,n},\SU_2)$ consists of a single point.

(iii-a) For $k=n$ the above-mentioned conjugacy class is the class of a central homomorphism, and so consists of a single point.

(iii-b) For $k\leq n-2$ the conjugacy class is the class of a non-central coaxial homomorphism, and so is isomorphic to $\Sph$.

(iii-c)
Since $\Hom_{\bm{\th}}(\grp_{0,n},\SU_2)$ consists of a single conjugacy class,
$\Rep_{\bm{\th}}(\grp_{0,n},\SU_2)$ consists of one point.

Moreover, note that the tangent space
to $\Hom_{\bm{\th}}(\grp_{0,n},\SU_2)$ at $\rho$
has dimension $t=-3+2(n-k)+\dim(Z(\rho))$ by Corollary \ref{cor:rep-undec}.
For $k=n$ we have $t=0$ and so the structure is reduced; 
for $k\leq n-2$ we have $t=2(n-k-1)$.
The scheme structure on $\Hom_{\bm{\th}}(\grp_{0,n},\SU_2)$
is reduced if and only if $t=2$, namely $k=n-2$.
\end{proof}


\subsection{Density and connectedness of the non-coaxial locus}\label{sec:connectedness}

In this section we prove the connectedness of absolute and relative
homomorphism and representation spaces, and of their non-coaxial loci.

\begin{proposition}[Connectedness of the absolute non-coaxial locus]\label{prop:connected-abs}
\begin{itemize}
\item[(i)]
The space $\Hom(\grp_{g,n},\SU_2)$ is smooth, connected, of dimension $6g+3n-3$.
The space $\Rep(\grp_{g,n},\SU_2)$ is connected.
\item[(ii)]
The non-coaxial locus in $\Hom(\grp_{g,n},\SU_2)$
and in $\Rep(\grp_{g,n},\SU_2)$ is dense and connected.
\end{itemize}
\end{proposition}
\begin{proof}
%
(i)
We have seen in Theorem \ref{mainthm:rep-undec}(ii) that
$\Hom(\grp_{g,n},\SU_2)$ is isomorphic to $\SU_2^{2g+n-1}$, and so it is smooth, connected, of dimension $6g+3n-3$.
It follows that the quotient $\Rep(\grp_{g,n},\SU_2)$ is connected.

(ii) 
Since $2g+n-1=1-\chi(\dot{S})\geq 2$ and
$\Hom(\grp_{g,n},\SU_2)$ is isomorphic to $\SU_2^{2g+n-1}$,
it easily follows that any homomorphism can be deformed to a non-coaxial one.
This proves the density of the non-coaxial locus in $\Hom(\grp_{g,n},\SU_2)$.
As for the connectedness, it is enough to note that
$\Hom(\grp_{g,n},\SU_2)$ is smooth, of dimension $6g+3n-3=2(2g-2+n)+(2g+n+1)$,
and that the coaxial locus has dimension $2g+n+1$ by 
by Proposition \ref{prop:coaxial}(i), 
and so codimension at least $2$ in $\Hom(\grp_{g,n},\SU_2)$.
The statement for the representation space immediately follows.
\end{proof}

As for the relative case, a pure codimension count is not enough because
the coaxial locus does not sit in the smooth locus of the relative homomorphism space.

\begin{theorem}[Connectedness of non-special $\Rep^{nc}_{\bm{\th}}$ space]\label{thm:connected-rel}
Let $(g,n,\bm{\th})$ be non-special.
Then  $\Hom_{\bm{\th}}(\grp_{g,n},\SU_2)$ and $\Rep_{\bm{\th}}(\grp_{g,n},\SU_2)$ are connected.
Moreover, the non-coaxial locus in $\Hom_{\bm{\th}}(\grp_{g,n},\SU_2)$ 
is open, dense, and connected, 
and so is the non-coaxial locus in $\Rep_{\bm{\th}}(\grp_{g,n},\SU_2)$.
\end{theorem}

We anticipate that the strategies for proving Theorem \ref{thm:connected-rel}
in genus zero and in positive genus are different.
In genus zero we will first prove that $\Hom_{\bm{\th}}(\grp_{g,n},\SU_2)$ is connected
and we will use such result to conclude that the non-coaxial locus is dense and connected.
In positive genus we will directly prove that the non-coaxial locus is dense and connected,
and obtain that $\Hom_{\bm{\th}}(\grp_{g,n},\SU_2)$ is connected as a consequence.\\

We will first deal with the case of genus $0$ and then with the case of positive genus.


\subsubsection{Density and connectedness of non-coaxial homomorphisms in genus zero.}\label{sec:conn0}

In this section we analyse properties of non-emptiness and connectedness
for $\Hom_{\bm{\th}}(\grp_{0,n},\SU_2)$ and for its non-coaxial locus.
In particular, we first establish when the relative homomorphism space is non-empty
and when its non-coaxial locus is nonempty.
Then we show that the relative homomorphism space is connected, and we deduce
that the non-coaxial locus is dense in nonspecial cases.
Finally, using the techniques of Section \ref{sec:m-c} we show that
the non-coaxial locus is connected.

\begin{notation}
In this section the angle vector $\bm{\th}=(\th_1,\dots,\th_n)$ will always be an $n$-tuple of positive real numbers, with $n\geq 2$. We will also use the associated
$\bm{\Ll}=(\Ll_1,\dots,\Ll_n)$ defined by $\Ll_i:=d(\th_i,\ZZ_o)\in[0,1]$.
\end{notation}

Choosing a basepoint $I$ in $\mathbb{S}^3$ determines
an identification $\SU_2\arr{\sim}{\lra}\mathbb{S}^3$ ,
which is actually a homothety of factor $1/2$
(since the points $I$ and $-I$ on $\SU_2$
lie at distance $2\pi$ for the metric induced by $\Kill$).
As discussed in Appendix \ref{app:intro}, 
there is an isomorphism
\[
\Hom_{\bm{\th}}(\grp_{0,n},\SU_2)\lra\mathpzc{Pol}(\bm{\Ll})
\]
between the relative homomorphism space and the space of closed spherical polygons 
$(p_1=I,p_2,\dots,p_n)$ in $\mathbb{S}^3$ with edge lengths $\pi\cdot (\Ll_1,\dots,\Ll_n)$, obtained by sending $\rho$ to the polygon with
$p_i=I\cdot\rho(\beta_1)\cdots\rho(\beta_{i-1})$.
Moreover such correspondence preserves coaxiality.

The first result of this section is the following.

\begin{proposition}[Non-emptiness of $\Hom_{\bm{\th}}$ for $g=0$]\label{prop:conn-genus0}
Let $\bm{\th}=(\th_1,\dots,\th_n)$ with $n\geq 2$.
\begin{itemize}
\item[(o)]
If $d_1(\bm{\th}-\bm{1},\ZZ^n_o)<1$, then
$\Hom_{\bm{\th}}(\grp_{0,n},\SU_2)$ is empty.
\item[(i)]
If $d_1(\bm{\th}-\bm{1},\ZZ^n_o)=1$, then
$\Hom_{\bm{\th}}(\grp_{0,n},\SU_2)$ consists of a single conjugacy class
of coaxial homomorphisms.
\item[(ii)]
If $d_1(\bm{\th}-\bm{1},\ZZ^n_o)>1$, then
the non-coaxial locus in $\Hom_{\bm{\th}}(\grp_{0,n},\SU_2)$
is non-empty.
\end{itemize}
\end{proposition}

The above result is also essentially proven in
\cite[Theorem A]{biswas} and \cite{Gal}, see also \cite[Theorem A]{MP1}.
Here we rely on Appendix \ref{app:mamaev} for a different and elementary proof,
that does not involve (semi)stable holomorphic bundles with parabolic structure.

\begin{proof}[Proof of Proposition \ref{prop:conn-genus0}]
By Theorem \ref{polygonsexistence} the space $\Hom_{\bm{\th}}(\grp_{0,n},\SU_2)$
non-empty if and only if $d_1(\bm{\th}-\bm{1},\ZZ^n_o)\geq 1$.
If $d_1(\bm{\th}-\bm{1},\ZZ^n_o)=1$, then all homomorphisms in
$\Hom_{\bm{\th}}(\grp_{0,n},\SU_2)$ belong to a single conjugacy class
by Corollary \ref{coaxialunique}(i).
If $d_1(\bm{\th}-\bm{1},\ZZ^n_o)>1$, then the non-coaxial locus
in $\Hom_{\bm{\th}}(\grp_{0,n},\SU_2)$ is non-empty by Corollary \ref{coaxialunique}(ii).
\end{proof}

The second result of this section is the following.

\begin{proposition}[Connectedness of $\Hom_{\bm{\th}}$ for $g=0$]\label{prop:conn04}
Let $n\geq 2$ and assume that $\bm{\th}=(\th_1,\dots,\th_n)$
satisfies $d_1(\bm{\th}-\bm{1},\ZZ^n_o)\geq 1$.
Then $\Hom_{\bm{\th}}(\grp_{0,n},\SU_2)$ is connected.
\end{proposition}

The case $d_1(\bm{\th}-\bm{1},\ZZ^n_o)=1$
follows from Proposition \ref{prop:conn-genus0}(i), thus so does the case $n=2$.
Hence we can assume $n\geq 3$.
Note that the statement is equivalent to saying that
$\mathpzc{Pol}(\bm{\Ll})$ is connected.

We start by reminding an observation about the length of sides of a convex triangles in $\mathbb{S}^2$.

\begin{lemma}[Existence and uniqueness of spherical triangles]\label{lemma:monlengh} 
Suppose $0\leq a\leq b\leq \pi$.
Then there is a continuous family of spherical triangles $A_tB_tC_t$ in $\mathbb{S}^2$
with $|B_tC_t|=a$, $|A_tC_t|=b$, indexed by 
$$t\in  I(a,b):=[b-a, \min(a+b, 2\pi-a-b)]\, ,$$ 
such that $|B_tC_t|=t$. Moreover, for every $t$,
the triangle $A_t B_t C_t$ is unique up to isometry of $\mathbb{S}^2$cdot ,
and it is contained inside a maximal circle 
if and only if $t$ is an endpoint of $I(a,b)$.
\end{lemma}

%
%
%

Before proceeding, we introduce certain subspaces of spherical polygons.

\begin{notation}
Given $U_n$ in $\cla_{\Ll_n}$, we will denote by
$\mathpzc{Pol}(\Ll_1,\dots,\Ll_{n-1},U_n)$
the subset of $\mathpzc{Pol}(\bm{\Ll})$ consisting of
polygons $(p_1=I,p_2,\dots,p_n)$ such that $p_n=U_n^{-1}$. 
%
\end{notation}

Spaces of polygons with one assigned edge will be useful in inductive proof.
On the other hand, their connectedness is equivalent to that of ordinary spaces of polygons.

\begin{sublemma}\label{sublemma:assigning}
The space $\mathpzc{Pol}(\bm{\Ll})$ is connected if and only if $\mathpzc{Pol}(\Ll_1,\Ll_2,\dots,\Ll_{n-1},U_n)$ is connected.
\end{sublemma}
\begin{proof}
We will show that both connectedness properties are equivalent to the connectedness
of $\mathpzc{Pol}(\bm{\Ll})/\SU_2$, where $\SU_2$ is acting 
via the isometric conjugacy action on $(\SU_2,I)\cong (\mathbb{S}^3,I)$.

The map $\mathpzc{Pol}(\bm{\Ll})\rar\mathpzc{Pol}(\bm{\Ll})/\SU_2$ is proper surjective, and its fibers are quotients of $\SU_2$ and so they are connected. 
By Lemma \ref{lemma:connected}(i) it follows that $\mathpzc{Pol}(\bm{\Ll})$ is connected if and only if $\mathpzc{Pol}(\bm{\Ll})/\SU_2$ is.

Similary, $\mathpzc{Pol}(\Ll_1,\Ll_2,\dots,U_n)\rar\mathpzc{Pol}(\bm{\Ll})/\SU_2$ is proper surjective, and its fibers are quotients of the centralizer of $U_n$
(which is either a $1$-parameter subgroup or the whole $\SU_2$) and 
so they are connected. 
By Lemma \ref{lemma:connected}(i) it follows that $\mathpzc{Pol}(\Ll_1,\Ll_2,\dots,U_n)$ is connected if and only if $\mathpzc{Pol}(\bm{\Ll})/\SU_2$ is.
\end{proof}

We can now prove the connectedness of the relative homomorphism spaces.

\begin{proof}[Proof of Proposition \ref{prop:conn04}]
We proceed by induction on $n\geq 3$. 

If $n=3$, then $\mathpzc{Pol}(\bm{\Ll})$ consists of a single conjugacy class
by Lemma \ref{lemma:monlengh}, and so it is connected.
Now we assume $n\geq 4$.

Consider the map
\[
\ell_{n-1,n}:\mathpzc{Pol}(\bm{\Ll})\lra [0,1]
\]
that sends $(p_1=I,\dots,p_n)$ to $\frac{1}{\pi}d_{\mathbb{S}^3}(p_{n-1},p_1)$,
where $d_{\mathbb{S}^3}$ is the usual distance on the unit $3$-sphere.
We want to show that the image and the fibers of $\ell_{n-1,n}$ are connected,
and then conclude that $\mathpzc{Pol}(\bm{\Ll})$ is connected by
Lemma \ref{lemma:connected}(i).

{\it{Connectedness of the image of $\ell_{n-1,n}$.}}\\
Let $\INTE_{n-1,n}$ be the image of the map $\cla_{\dd_{n-1}}\times\cla_{\dd_n}\rar[0,1]$
that sends $(U_{n-1},U_n)$ to $D_I(U_{n-1}U_n)$
and let $\INTE_{1,n-2}$ be the image of
$\cla_{\dd_1}\times\cdots\times\cla_{\dd_{n-2}}\rar[0,1]$
that sends $(U_1,\dots,U_{n-2})$ to $D_I(U_1\cdots U_{n-2})$.
The interval $\INTE_{n-1,n}$ consists of all the possible
lengths (divided by $\pi$) of the third edge of a spherical triangle whose first two edges have assigned lengths $\pi(\dd_{n-1},\dd_n)$.
Similarly, the interval $\INTE_{1,n-2}$ consists of all the possible
lengths (divided by $\pi$) of the $(n-1)$-st edge of a spherical $(n-1)$-gon whose first $n-2$ edges have assigned lengths $\pi(\dd_1,\dots,\dd_{n-2})$.
Hence, the image of $\ell_{n-1,n}$ is exactly the interval $\INTE_{1,n-2}\cap \INTE_{n-1,n}$.

{\it{Connectedness of the fibers of $\ell_{n-1,n}$.}}\\
Let $d\in \INTE_{1,n-2}\cap \INTE_{n-1,n}$ and fix $U\in\cla_d$.
Consider the subset $\mathpzc{Pol}_{n,U}(\bm{\Ll})$
of polygons $(p_1=I,p_2,\dots,p_n)$ in $\mathpzc{Pol}(\bm{\Ll})$ such that $p_{n-1}\cdot U=p_1$, and
note that
\[
\SU_2\cdot \mathpzc{Pol}_{n,U}(\bm{\Ll})=\ell_{n-1,n}^{-1}(d)
\]
where $\SU_2$ acts by conjugation as usual.
As $\SU_2$ is connected, it is enough to show that 
$\mathpzc{Pol}_{n,U}(\bm{\Ll})$ is connected.
Since the map
\[
\xymatrix@R=0in{
\mathpzc{Pol}(\dd_1,\dots,\dd_{n-2},U)\times \mathpzc{Pol}(\dd_{n-1},\dd_n,U^{-1})\ar[r] & \mathpzc{Pol}_{n,U}(\bm{\Ll})\\
((q_1=I,q_2,\dots,q_{n-2}),(r_1=I,r_2,r_3))
\ar@{|->}[r] &
(I,q_2,\dots,q_{n-2},r_2\cdot U^{-1})
}
\]
is manifestly an isomorphism, it is enough to show
that $\mathpzc{Pol}(\dd_1,\dots,\dd_{n-2},U)\times \mathpzc{Pol}(\dd_{n-1},\dd_n,U^{-1})$ is connected.
This follows from Sublemma \ref{sublemma:assigning} by
inductive hypothesis.
\end{proof}

Here is a consequence of the above connectedness result.

\begin{corollary}[Density of non-coaxial locus for non-special $g=0$]\label{cor:density-genus0}
Assume $d_1(\bm{\th}-\bm{1},\ZZ^n_o)>1$.
Inside $\Hom_{\bm{\th}}(\grp_{0,n},\SU_2)$,
the coaxial locus consists of finitely many conjugacy classes and
the non-coaxial locus is dense.
\end{corollary}
\begin{proof}
Clearly there are only finitely many polygons in $\mathpzc{Pol}(\bm{\Ll})$ 
that sit on a given maximal circle of $\mathbb{S}^3$: this implies the first claim.

As for the second claim, it is enough to show that every coaxial conjugacy class of polygons is in the closure of the non-coaxial locus.
By Proposition \ref{prop:conn-genus0}(ii) such non-coaxial locus is nonempty,
and so we can pick a non-coaxial polygon $P^{nc}$ there. 

Let now $P$ be any coaxial polygon.
By Proposition \ref{prop:conn04}, the space $\mathpzc{Pol}(\bm{\Ll})$ is connected
and so there exists a path 
$(P_t)_{t\in[0,1]}$ in $\mathpzc{Pol}(\bm{\Ll})$ such that 
$P_0=P$ and $P_1=P^{nc}$.
There exists $t_0\in[0,1)$ such that $t_0$ is the maximum $t$ for which
$P_t$ is conjugate to $P$. It follows that $P_{t_0}$ is in the closure of the non-coaxial locus, and so the whole conjugacy class of $P$ is.
\end{proof}

Given $\dd_1,\dots,\dd_k\in[0,1]$,
recall that $\cla_{\bm{\dd}}=\cla_{\dd_1}\times\cdots\times\cla_{\dd_k}$
and let
\[
\tau_k:\cla_{\bm{\dd}} \lra[-2,2]
\]
be defined as $\tau_k(U_1,\dots,U_k):=\tr(U_1\cdots U_k)$.

Another consequence of the connectedness of $\mathpzc{Pol}(\bm{\Ll})$
is the following.

\begin{corollary}[Monotone-connectedness of $\cla_{\bm{\dd}}$]\label{monconnect} 
The pair $(\cla_{\bm\dd}, \tau_k)$ is monotone-connected.
\end{corollary}
\begin{proof} 
The space $\cla_{\bm{\dd}}$ is a smooth, connected, real algebraic manifold,
and the function $\tau_k$ is real algebraic.
The level sets $\{\tau_k=c\}$ are connected, since they are isomorphic to spaces of spherical polygons $\mathpzc{Pol}(\dd_1,\ldots, \dd_k,\dd_{k+1})$
with $\dd_{k+1}:=\frac{1}{\pi}\arccos(\frac{c}{2})$. Hence the statement follows from Proposition \ref{analyticmonotone}.
\end{proof}


We now want to show the following.

\begin{proposition}[Connectedness of non-coaxial locus for $g=0$]\label{prop:g=0-noncoaxconnected}
Assume that $\bm{\th}$ satisfies $d_1(\bm{\th}-\bm{1},\ZZ^n_o)>1$.
Then the non-coaxial locus of $\Hom_{\bm{\th}}(\grp_{0,n},\SU_2)$ is connected. 
\end{proposition}

Before proving the above proposition, we need a technical lemma.

\begin{lemma}[Disalligning adjacent edges of a non-coaxial polygon]\label{deformtonon} 
Let $P$ be a non-coaxial polygon in $\mathpzc{Pol}^{nc}(\bm\dd)$. Suppose that $\dd_i,\dd_{i+1}\ne 0,1$ and the edges $e_i$, $e_{i+1}$ are aligned. Then there is a continuous deformation $P_t$ of $P_0=P$ such that, for $t>0$ small enough, the edges $e_i(t), e_{i+1}(t)$ are not aligned. 
\end{lemma}

\begin{proof} 
Let $m$ be the number of integral entries in $\bm{\dd}$.
Since $\mathpzc{Pol}^{nc}(\bm\dd)$ is identified 
to $\Hom^{nc}_{\bm{\th}}(\grp_{0,n},\SU_2)$, it is a manifold of dimension $2(n-m)-3$
by Corollary \ref{cor:rep-undec}(ii).

At the same time,
the locus of $\mathpzc{Pol}(\bm\dd)$ consisting of polygons for which $e_i$ and $e_{i+1}$ are aligned can be identified to a space  of polygons $\mathpzc{Pol}(\bm\dd')$ with $n-1$ edges (throwing away the $(i+1)$-st vertex). 
Indeed, $\mathpzc{Pol}^{nc}(\bm\dd')$ can be identified to a submanifold 
of $\mathpzc{Pol}^{nc}(\bm\dd)$ of codimension $2$ or $4$.
The conclusion clearly follows. 
\end{proof}

\begin{proof}[Proof of Proposition \ref{prop:g=0-noncoaxconnected}] 
We need to show that the locus
$\mathpzc{Pol}^{nc}(\bm{\Ll})$
of non-coaxial polygons inside $\mathpzc{Pol}(\bm{\Ll})$ is connected.

Let $P_0,P_1\in \mathpzc{Pol}^{nc}(\bm\dd)$ be two non-coaxial polygons. Suppose first that the edges $e_{n-1}$ and $e_n$ are not aligned neither in $P_0$ nor in $P_1$. We will construct a path in $\mathpzc{Pol}^{nc}(\bm\dd)$ that connects $P_0$ to $P_1$.

For $i=0,1$ let $P'_i$ be the broken geodesic obtained from $P_i$ by removing
the edges $e_{n-1},e_n$.
Using Corollary \ref{monconnect} we can find a continuous deformation
of broken geodesics $P'_t=(p_1(t)=I,p_2(t),\dots,p_{n-1}(t))$ with fixed edge lengths
between $P'_0$ to $P'_1$ with the additional property that
the distance between $p_{n-1}(t)$ and $p_1(t)$ changes monotonously. 
Then, using Lemma \ref{lemma:monlengh}
we can insert back edges $e_{n-1}(t)$ and $e_n(t)$ so that
the resulting family of polygons $P_t=(p_1(t)=I,p_2(t),\dots,p_n(t))$ varies continuously.
It also follows from Lemma \ref{lemma:monlengh} that edges $e_{n-1}(t)$ and $e_n(t)$ stay non-coaxial during such deformation, which proves the proposition in this case.

In case $e_{n-1}$ and $e_n$ are aligned in $P_0$ (or $P_1$), using Lemma \ref{deformtonon} we first construct a small deformation of $P_0$ inside $\mathpzc{Pol}^{nc}(\bm\dd)$ so that $e_{n-1}$ and $e_n$ are not aligned any more, and then we proceed as above.
\end{proof}


\subsubsection{Connectedness and density of non-coaxial locus in positive genus.}

Thoughout this section we focus on non-special cases in positive genus.
We begin by proving that the non-coaxial locus is dense.

\begin{theorem}[Density of non-coaxial locus in non-special cases]\label{thm:density-non-dec}
Suppose that $(g,n,\bm{\th})$ is not special.
Then the non-coaxial locus is non-empty and dense in $\Hom_{\bm{\th}}(\grp_{g,n},\SU_2)$
and in $\Hom_{\bm{\th}}(\grp_{g,n},\SL_2(\CC))$.
As a consequence $\Hom_{\bm{\th}}(\grp_{g,n},\SU_2)$ 
and $\Hom_{\bm{\th}}(\grp_{g,n},\SL_2(\CC))$
have pure dimension $6g-3+2(n-k)$ and $2(6g-3+2(n-k))$ respectively.
\end{theorem}
\begin{proof}
Note that the claimed pure-dimensionality follows from the density
of the non-coaxial locus, which is a manifold 
by Corollary \ref{cor:rep-undec}(ii).

We separately treat non-emptiness and density in the $\SL_2(\CC)$ and the $\SU_2$-case.\\

{\it{The $\SL_2(\CC)$ case.}}
Since $\Hom_{\bm{\th}}(\grp_{g,n},\SL_2(\CC))$ contains
$\Hom_{\bm{\th}}(\grp_{g,n},\SU_2)$, non-emptiness of the non-coaxial locus follows from
the $\SU_2$ case treated below.

As for density, let $k$ be the number of integral entries of $\bm{\th}$.
We have seen in Lemma \ref{lemma:number-eq-C}
 that $\Hom_{\bm{\th}}(\grp_{g,n},\SL_2(\CC))$ can be realized via the embedding $\lambda_\CC$ inside $\SL_2(\CC)^{2g+n}$, and that in $\SL_2(\CC)^{2g+n}$ it is cut by $n+2k+3$ algebraic equations.
Hence, each irreducible component of $\Hom_{\bm{\th}}(\grp_{g,n},\SL_2(\CC))$
must have dimension at least $3(2g+n)-(n+2k+3)=6g+2(n-k)-3$.
By Proposition \ref{prop:coaxial}(ii) the coaxial locus in $\Hom_{\bm{\th}}(\grp_{g,n},\SL_2(\CC))$ has dimension $2g+2$.
Since $2g+2<6g+2(n-k)-3$ in non-special cases, the non-coaxial locus is dense.\\

{\it{The $\SU_2$ case: non-emptiness.}}
We want to find $(M_1,N_1,\dots,M_g,N_g,B_1,\dots,B_n)$ in $\SU_2^{2g+n}$
that represents a non-coaxial point of $\Hom_{\bm{\th}}(\grp_{g,n},\SU_2)$.
We distinguish three cases.

{\it{Case $g\geq 2$.}}
Choose $B_i\in \cla_{\dd_i}$ for all $i$.
By Corollary \ref{cor:comm} the commutator map is surjective, and so 
there exist $M_1,N_1\in\SU_2$ such that
$[M_1,N_1]=(B_1\cdots B_n)^{-1}$. Note that it is always possible to choose them so that
$M_1\neq \pm I$.
Hence, we can pick $M_2=N_2\in\SU_2$ that does not belong to the same $1$-parameter subgroup as $M_1$. Finally, we set $M_i=N_i=I$ for $i>2$.
Then $(M_1,N_1,\dots,M_g,N_g,B_1,\dots,B_n)$ works.

{\it{Case $g=1$.}}
Recall that, by non-speciality,
there must be some non-integer $\th_i$.
As in the above case, pick $B_i\in\cla_{\dd_i}$ for all $i$.
If $B_1\cdots B_n\neq I$, then it is enough to pick
$M_1,N_1\in\SU_2$ such that
$[M_1,N_1]=(B_1\cdots B_n)^{-1}$.
If $B_1\cdots B_n=I$, then we pick $M_1=N_1\in\SU_2$
not in the same $1$-parameter subgroup as $B_i$.

{\it{Case $g=0$.}}
It follows from Proposition \ref{prop:conn-genus0}(ii).\\

{\it{The $\SU_2$ case: density.}}
Let $\rho$ be a coaxial homomorphism
in $\Hom_{\bm{\th}}(\grp_{g,n},\SU_2)$.
By definition, $\rho$ takes values in a $1$-parameter subgroup
$H$ (if $\rho$ is central, then pick any $H$).
Let $Y\in\Sph$ be an infinitesimal generator for $H$.
Hence,
\[
\lambda(\rho)=
\bm{\ee}(r_1 Y,s_1 Y,\dots,r_g Y,s_g Y,\epsilon_1\th_1 Y,\dots,\epsilon_n\th_n Y),
\]
where $\epsilon_1,\dots,\epsilon_n\in\{\pm 1\}$.
We want to construct {\it{a non-coaxial deformation of $\rho$}}, namely
a path $t\mapsto \rho_t\in
\Hom_{\bm{\th}}(\grp_{g,n},\SU_2)\subset \SU_2^{2g+n}$ such that
$\rho_0=\rho$ and $\rho_t$ is not coaxial for $t\in(0,\e)$
for some small $\e>0$.
In fact what we will construct is the path $t\mapsto (\bm{M}(t),\bm{N}(t),\bm{B}(t))$ inside $\SU_2^{2g+n}$ that corresponds to $\lambda(\rho_t)$.

Choose $Z\in\su_2$ not a multiple of $Y$.

{\it{Case $g\geq 2$.}}
If $\rho$ is central, namely 
all $\ee(r_i Y),\ee(s_i Y)$ are $\pm I$ and $\bm{\th}\in\ZZ^n$,
then a non-coaxial deformation of $\rho$
is obtained by moving the first four entries as follows
\[
t\mapsto \lambda(\rho)\cdot \exp(tY,tY,tZ,tZ,0,\dots,0).
\]
Indeed, for small $t\neq 0$ we have $\pm I\neq \ee(r_1 Y)\exp(tY)\in H$
and $\ee(r_2 Y)\exp(tZ)\notin H$.

If $\rho$ is not central then,
up to renumbering the generators of $\grp_{g,n}$,
we can assume that at least one among
$\ee(r_1 Y)$, $\ee(s_1 Y)$, $\ee(\epsilon_1\th_1 Y)$ is different from
$\pm I$.
A non-coaxial deformation is then given by
\[
t\mapsto
\bm{\ee}(r_1 Y,s_1 Y,r_2(Y+tZ),s_2(Y+tZ),r_3 Y, s_3 Y,\dots,\epsilon_n\th_n Y).
%
\]
Indeed, $\ee(r_2(Y+tZ))$ does not belong to $H$ for small $t\neq 0$.

{\it{Case $g=1$.}}
We claim that, if $\rho\in\Hom_{\bm{\th}}(\grp_{1,n},\SU_2)$ is coaxial,
then there must be a non-integral $\th_i$.
In fact, coaxiality implies that
$[\rho(\mu_1),\rho(\nu_1)]=I$ and then,
by Equation \eqref{eq:I},
that $\sum_i (\pm(\th_i-1))\in 2\ZZ$ for a suitable choice of the signs.
Since $(1,n,\bm{\th})$ is assumed to be non-special, it follows that
$\bm{\th}\notin\ZZ^n$.
Up to renumbering the punctures,
we can assume that $\th_1\notin\ZZ$.
Then
\[
t\mapsto
\bm{\ee}(r_1 (Y+tZ),s_1 (Y+tZ),\epsilon_1\th_1 Y,\dots,\epsilon_n\th_n Y)
\] 
corresponds to a non-coaxial deformation, since $\pm I\neq \ee(\epsilon_1\th_1 Y)\in H$ and $\ee(r_1(Y+tZ))\notin H$ for $0\neq t$ small.

{\it{Case $g=0$.}}
It was proven in Corollary \ref{cor:density-genus0}.
%
%
%
\end{proof}

Now we turn to connectedness of the non-coaxial locus inside
the  relative homomorphism space. 
A major role in the argument below
will be played by the commutator map $\comm:\SU_2\times\SU_2\rar\SU_2$
introduced in Section \ref{sec:comm-prod}.

%
%
%

\begin{proof}[Proof of Theorem \ref{thm:connected-rel}]
The case of genus $0$
was treated in Proposition \ref{prop:conn-genus0}(ii),
Corollary \ref{cor:density-genus0} and Proposition \ref{prop:g=0-noncoaxconnected}.
Now on we assume $g\geq 1$.

The non-coaxial locus in $\Hom_{\bm{\th}}(\grp_{g,n},\SU_2)$
is open by Lemma \ref{mainlemma:alg}(iv)
and is non-empty and dense by Theorem \ref{thm:density-non-dec}.

We are thus left to prove the connectedness of $\Hom^{nc}_{\bm{\th}}(\grp_{g,n},\SU_2)$.
Note that it will imply that the whole relative homomorphism space is connected.

Identify $\Hom_{\bm{\th}}(\grp_{g,n},\SU_2)$
to its image in $\SU_2^{2g}\times\cla_{\dd_1}\times\dots\times\cla_{\dd_n}$ via
$\lambda$, and
consider the map
\[
f:
\xymatrix@R=0in{
\Hom_{\bm{\th}}(\grp_{g,n},\SU_2)\ar[rr] && \SU_2^{2g-2}\times\cla_{\bm{\dd}}\\
(M_1,\dots,N_g,B_1,\dots,B_n)\ar@{|->}[rr] && (M_2,\dots,N_g,B_1,\dots,B_n).
}
\]
Let $Y$ be the closed subset of $\SU_2^{2g-2}\times\cla_{\bm{\dd}}$ defined as
\[
Y:=\{ (M_2,\dots,B_n)\in\SU_2^{2g-2}\times\cla_{\bm{\dd}} \ |
\ [M_2,N_2]\cdots[M_g,N_g]B_1\cdots B_n=I\}
\]
and denote by $Y^c$ the complement of $Y$ inside $\SU_2^{2g-2}\times\cla_{\bm{\dd}}$.
Every $(M_1,N_1,\dots,B_n)$ in $f^{-1}(Y^c)$ is non-coaxial,
because $[M_1,N_1]\neq I$.

{\it{Claim: the subset $Y$ has codimension at least $2$.}}
Observe first that $Y$ can be identified to
$\Hom_{\bm{\th}}(\grp_{g-1,n},\SU_2)$. 
Now we separately analyze two different cases.

Suppose first that $g=1$ and the triple $(0,n,\bm{\th})$ is special.
It follows that $d_1(\bm{\th}-\bm{1},\ZZ^n_o)\leq 1$.
The conclusion trivially holds if $d_1(\bm{\th}-\bm{1},\ZZ^n_o)<1$,
since in this case $Y\cong\Hom_{\bm{\th}}(\grp_{0,n},\SU_2)$
is empty by Proposition \ref{prop:conn-genus0}(o). Hence, we can assume  $d_1(\bm{\th}-\bm{1},\ZZ^n_o)=1$.
If $n-k=0$, then $\bm{\th}\in\ZZ^n$ and the conclusion follows from Proposition \ref{prop:genus1}.
If $n-k=1$, then $Y\cong\Hom_{\bm{\th}}(\grp_{0,n},\SU_2)$ is empty. For $n-k\geq 2$, then $\Hom_{\bm{\th}}(\grp_{0,n},\SU_2)$
has dimension $2$ by Proposition \ref{prop:conn-genus0}(i), whereas
$\cla_{\bm{\dd}}$ has dimension $2(n-k)\geq 4$: it follows that $Y$ has codimension at least $2$.

Suppose now that $g\geq 2$, or that $g=1$ and $(0,n,\bm{\th})$ is not special.
The space $\Hom_{\bm{\th}}(\grp_{g-1,n},\SU_2)$ has pure dimension $6(g-1)-3+2(n-k)$
by Theorem \ref{thm:density-non-dec}.
On the other hand, $\SU_2^{2g-2}\times\cla_{\bm{\dd}}$ has dimension $3(2g-2)+2(n-k)=6g-6+2(n-k)$. It follows that $Y$ has codimension at least $3$.\\

We now want to prove that $\Hom^{nc}_{\bm{\th}}(\grp_{g,n},\SU_2)$ is connected
by showing that
$f^{-1}(Y^c)$ is connected and is dense
inside $\Hom^{nc}_{\bm{\th}}(\grp_{g,n},\SU_2)$.

{\it{Connectedness of $f^{-1}(Y^c)$.}}
Since $\SU_2^{2g-2}\times\cla_{\bm{\dd}}$ is smooth and connected and $Y$ has codimension at least $2$ inside it, $Y^c$ is connected.
Clearly $f$ is proper. We want to show that $f$ is surjective with connected fibers,
so that the connectedness of $f^{-1}(Y^c)$ will follow by Lemma \ref{lemma:connected}(i).
Let
\[
p_1:
\xymatrix@R=0in{
\Hom_{\bm{\th}}(\grp_{g,n},\SU_2)\ar[rr] &&  \SU_2^2\\
(M_1,N_1,\dots,B_n)\ar@{|->}[rr] && (M_1,N_1)
}
\]
Observe that $p_1$ realizes an isomorphism from $f^{-1}(M_2,\dots,B_n)$ to
$\comm^{-1}(C)$ with $C=([M_2,N_2]\cdots[M_g,N_g]B_1\cdots B_n)^{-1}$. 
The conclusion follows, since the map $\comm$ is surjective and with connected fibers
by Corollary \ref{cor:comm}.

{\it{Density of $f^{-1}(Y^c)$.}}
Call $f^{nc}:\Hom^{nc}_{\bm{\th}}(\grp_{g,n},\SU_2)\rar \SU_2^{2g-2}\times\cla_{\bm{\dd}}$
the restriction of $f$. We want to show that $(f^{nc})^{-1}(Y)$ has dimension
strictly smaller than $\Hom^{nc}_{\bm{\th}}(\grp_{g,n},\SU_2)$ at every point:
it will follow that $(f^{nc})^{-1}(Y^c)=f^{-1}(Y^c)$ is dense.
For every $y\in Y$, the fiber $(f^{nc})^{-1}(y)$
has dimension $4$ by Corollary \ref{cor:comm}(i).
Moreover, by the above claim $Y$ has dimension at most $6(g-1)+2(n-k)-2$. 
It follows that $(f^{nc})^{-1}(Y)$ 
has dimension at most $6(g-1)+2(n-k)+2$.
On the other hand, $\Hom^{nc}_{\bm{\th}}(\grp_{g,n},\SU_2)$ is smooth of dimension
$6g-3+2(n-k)$ by Corollary \ref{cor:rep-undec}(ii). The conclusion follows.
\end{proof}

An easy consequence of the pure-dimensionality
of $\Hom_{\bm{\th}}(\grp_{g,n},\SU_2)$
is that the coaxial locus
is small in non-special cases.

\begin{corollary}[Codimension of the coaxial locus in non-special cases]\label{cor:codim-coax}
Let $(g,n,\bm{\th})$ be non-special.
Then the coaxial locus (if non-empty) inside $\Hom_{\bm{\th}}(\grp_{g,n},\SU_2)$
has codimension $4(g-1)+2(n-k)-1\geq 3$.
Hence, the coaxial locus (if non-empty) 
inside
$\Rep_{\bm{\th}}(\grp_{g,n},\SU_2)$ has pure codimension $4(g-1)+2(n-k)-2\geq 2$.
\end{corollary}
\begin{proof}
Consider first the relative homomorphism space.
By Proposition \ref{prop:coaxial}(ii), 
the coaxial locus in $\Hom_{\bm{\th}}(\grp_{g,n},\SU_2)$ has pure dimension $2g+2$,
whereas $\Hom_{\bm{\th}}(\grp_{g,n},\SU_2)$ has pure dimension
$3(2g-1)+2(n-k)$ by Theorem \ref{thm:density-non-dec}.
Denote by $\mathrm{cdim}:=(6g-3+2(n-k))-(2g+2)=4(g-1)+2(n-k)-1$.

The above considerations show that, if non-empty, the
the coaxial locus in $\Hom_{\bm{\th}}(\grp_{g,n},\SU_2)$ has
pure codimension $\mathrm{cdim}$.
We claim that $\mathrm{cdim}\geq 3$.

{\it{Case $g=0$.}}
We must have $n-k\geq 3$ and $d_1(\bm{\th}-\bm{1},\ZZ^n_o)>1$.
For $n-k=3$, Lemma \ref{lemma:monlengh} implies that the coaxial locus
is empty (see also the beginning of the proof of Proposition \ref{prop:conn04}).
For $n-k\leq 4$, we have $\mathrm{cdim}=2(n-k)-5\geq 3$.

{\it{Case $g=1$.}}
If all $\th_i$ are integer, namely if $n-k=0$, then
the sum $\sum(\th_i-1)$ must be odd, since $(1,n,\bm{\th})$
is assumed non-special. Hence an element of $\Hom_{\bm{\th}}(\grp_{1,n},\SU_2)$
identifies to $(M_1,N_1,B_1,\dots,B_n)$ with $B_i\in\cla_{\dd_i}$
and $[M_1,N_1]=-I$. This implies that the coaxial locus is empty in this case.
If $n-k=1$, then points of $\Hom_{\bm{\th}}(\grp_{1,n},\SU_2)$ must satisfy
$[M_1,N_1]\neq \pm I$ and so the coaxial locus is empty again.
For $n-k\geq 2$, we have $\mathrm{cdim}=2(n-k)-1\geq 3$.

{\it{Case $g\geq 2$.}}
We have $\mathrm{cdim}=4(g-1)+2(n-k)-1\geq 3$.

The statement for the representation space easily follows.
Indeed, recalling
that the central locus has dimension $0$ and that a coaxial non-central
homomorphism has $1$-dimensional stabilizer.
\end{proof}

As a consequence we have all ingredients to draw our conclusions
on the non-coaxial locus of non-special relative homomorphism and representation spaces.

\begin{proof}[Proof of Theorem \ref{mainthm:rep-undec-nonsp}]
Recall that smoothness of the non-coaxial locus is proven in
Theorem \ref{mainthm:rep-undec-rel}(ii-iv) together with the determination of its dimension. 

(iii) Non-emptiness and density of the non-coaxial locus 
is proven in Theorem \ref{thm:density-non-dec}.

(iv) Connectedness of the non-coaxial locus is proven in Theorem \ref{thm:connected-rel}. 

(i) Since the non-coaxial locus is smooth,
pure dimensionality of the relative homomorphism and representation spaces follows from (iii) and (iv).

(ii) Because of (i), 
Corollary \ref{cor:rep-undec}(i-ii) also shows that the smooth locus of $\Hom_{\bm{\th}}(\grp_{g,n},\SU_2)$ coincides with the non-coaxial locus. It follows that
$\Hom_{\bm{\th}}(\grp_{g,n},\SU_2)$ is reduced and irreducible.

Concerning the relative representation space, it is reduced and irreducible and
its smooth locus consists of $[\rho]$ at which
$\mathrm{dim}(T_\rho\Hom_{\bm{\th}}(\grp_{g,n},\SU_2))+\mathrm{dim}(Z(\rho))$ achieves its minimum
by Remark \ref{rmk:X/G}.
Such minimum is achieved at the non-coaxial locus by
Corollary \ref{cor:rep-undec}(ii), which is non-empty by (iii).
\end{proof}


\section{Decorated representation spaces}\label{sec:decorated}

We keep the same notation as in the beginning of Section \ref{sec:rep}.
In particular $V=\mathcal{M}_{2,2}(\CC)$.


\subsection{Topology and semi-analytic structure}\label{sec:analytic}

Analogously to Section \ref{ssc:algebraic}, 
in order to prove Lemma \ref{mainlemma:anal},
we endow the absolute decorated
homomorphism space with an analytic structure. The purpose is achieved
by embedding it inside a smooth algebraic variety $\wh{\Gcal}$,
so that its image is described by the equation $\wh{R}=I$.
Moreover, the conjugacy action by $\PSU_2$ is the restriction
of a natural action on $\wh{\Gcal}$, which is compatible with the map $\wh{R}$.
The relative case is similar.

\subsubsection{The embedding $\bm{\wh{\lambda}}$.}\label{ssc:lambda-hat}
Let $\wh{\Gcal}$\index{$\wh{\Gcal}$}
be the real algebraic subset
$\SU_2^{2g}\times \su_2^{\oplus n}$
of dimension $6g+3n$
of the real vector space $V^{\oplus 2g}\times\su_2^{\oplus n}$.
and let $\Nor$ (for {\it{norm}}) be the algebraic map defined as 
\[
\Nor:
\xymatrix@R=0in{
\wh{\Gcal}\ar[rr] && \RR_{\geq 0}^n\\
(M_1,N_1,\dots,X_1,\dots,X_n)\ar@{|->}[rr] && \left(\|X_1\|^2,\dots,\|X_n\|^2\right)
}
\]
We call $\wh{\Gcal}_{\bm{\th}}:=\Nor^{-1}(\th_1^2,\dots,\th_n^2)$\index{$\wh{\Gcal}_{\bm{\th}}$, $\wh{\Gcal}_+$}
for any $\bm{\th}\in\RR^n_{>0}$ and $\wh{\Gcal}_+:=\Nor^{-1}(\RR_{>0}^n)$.

Clearly, $\wh{\Gcal}_+$ is a smooth algebraic subset of dimension $6g+3n$.
Since $\Nor$ is submersive on $\wh{\Gcal}_+$,
it follows from the implicit function theorem that $\wh{\Gcal}_{\bm{\th}}$
is a smooth algebraic subset of dimension $6g+2n$.

\subsubsection{The map $\bm{\wh{R}}$.}\label{ssc:R-hat}
Let now $\wh{R}$\index{$\wh{R}$, $\wh{R}_{\bm{\th}}$}
be the real analytic map defined as
\[
\wh{R}: 
\xymatrix@R=0in{
\wh{\Gcal}\ar[rr] && \SU_2\\
(M_1,\dots,X_n)\ar@{|->}[rr] &&\prod_j [M_j,N_j]\prod_i \ee(X_i)
}
\]
and denote by $\wh{R}_{\bm{\th}}:\wh{\Gcal}_{\bm{\th}}\rar \SU_2$ the restriction of $\wh{R}$.

The injective map 
\[
\wh{\lambda}:\xymatrix@R=0in{
\wh{\Hom}(\grp_{g,n},\SU_2)\ar[r] & \wh{\Gcal}\\
(\rho,\Axis)\ar@{|->}[r] & (\rho({\mu}_1),\rho({\nu}_1),\dots,\rho({\mu}_g),\rho({\nu}_g),\Axis({\beta}_1),\dots,\Axis({\beta}_n))
}\index{$\wh{\lambda}$}
\]
identifies $\wh{\Hom}(\grp_{g,n},\SU_2)$ to
its image $\wh{R}^{-1}(I)\cap \wh{\Gcal}_+$,
and
$\wh{\Hom}_{\bm{\th}}(\grp_{g,n},\SU_2)$
to the subset $\wh{R}^{-1}(I)\cap\wh{\Gcal}_{\bm{\th}}=\wh{R}^{-1}_{\bm{\th}}(I)$.

\subsubsection{The map ${\bm{\Theta}}$.}\label{ssc:Theta}
Taking the componentwise square root of the map $\Nor$,
we obtain a map\index{$\bm{\Theta}$}
\[
\bm{\Theta}:\wh{\Hom}(\grp_{g,n},\SU_2)\lra\RR_+^n
\]
defined as $\bm{\Theta}(\rho,\Axis):=(\|\Axis(\beta_1)\|,\dots,\|\Axis(\beta_n)\|)$.
In this way,
$\wh{\Hom}_{\bm{\th}}(\grp_{g,n},\SU_2)$ identifies to $\bm{\Theta}^{-1}(\bm{\th})$.

\subsubsection{The conjugacy action.}\label{ssc:conj-dec}
Note that $\PSU_2$ acts on $V^{\oplus 2g}\times\su_2^{\oplus n}$ componentwise
via conjugaction and via adjunction.
Moreover $\wh{R}$ and $\Nor$ are $\PSU_2$-equivariant and the $\PSU_2$-action
on $\wh{R}^{-1}(I)\cap\wh{\Gcal}_+$ and $\wh{R}^{-1}(I)\cap \wh{\Gcal}_{\bm{\th}}$
agrees with the wished actions on the decorated homomorphism spaces.

\subsubsection{Analytic and semi-analytic structure.}\label{ssc:analytic}

We begin by addressing the first claim of Lemma \ref{mainlemma:anal},
namely the analyticity of the decorated homomorphism space
and the semi-analiticity of the corresponding decorated representation space.

\begin{proof}[Proof of Lemma \ref{mainlemma:anal}(o)] 
Analogously to Lemma \ref{mainlemma:alg}(o), 
the constructions performed in Sections \ref{ssc:lambda-hat}-\ref{ssc:R-hat}
show that the space $\wh{\Hom}(\grp_{g,n},\SU_2)$ 
is homeomorphic via $\wh{\lambda}$
to the subset of $\wh{\Gcal}_+$ described by the analytic equation $\wh{R}=I$.
Hence, $\wh{\Hom}(\grp_{g,n},\SU_2)$ is induced a real analytic structure.
Taking the quotient of $\wh{\Gcal}_+$
by the action of $\PSU_2$ described in Section \ref{ssc:conj-dec}, we obtain that 
$\wh{\Gcal}_+/\PSU_2$ is homeomorphic
to a semi-algebraic subset of a Euclidean space (the construction is similar to
the one described in Remark \ref{rmk:quotients}).
Moreover the locus $\{\wh{R}=I\}$ descends to a semi-analytic subset
of $\wh{\Gcal}_+/\PSU_2$ homeomorphic to $\wh{\Rep}(\grp_{g,n},\SU_2)$.
\end{proof}

Before completing the proof of Lemma \ref{mainlemma:anal},
we examine the behaviour of the map
\[
\forg:\wh{\Hom}(\grp_{g,n},\SU_2)\lra \Hom(\grp_{g,n},\SU_2)
\]
that forgets the decoration, defined as $\forg(\rho,\Axis):=\rho$,
and of its restriction 
\[
\forg_{\bm{\th}}:\wh{\Hom}_{\bm{\th}}(\grp_{g,n},\SU_2)\lra \Hom_{\bm{\th}}(\grp_{g,n},\SU_2)
\]
to the relative spaces.

\begin{lemma}[Forgetting the decoration]\label{lemma:forgetful}
The map $\forg$ is real-analytic.
Moreover, given $\bm{\th}\in\RR^n_{>0}$, the map
$\forg_{\bm{\th}}$ is a trivial bundle with fiber $(\Sph)^k$,
where $k$ is the number of integer entries of $\bm{\th}$.
\end{lemma}
\begin{proof}
The homomorphism space is real-algebraic by
Lemma \ref{mainlemma:alg}(o) and
the decorated homomorphism space is real-analytic
by Lemma \ref{mainlemma:anal}(o) proven above.
Hence, $\forg$ is real-analytic, 
being the restriction of the real-analytic map
$\wh{\Gcal}\rar\Gcal$ that sends
$(M_1,\dots,N_g,X_1,\dots,X_n)$ to $(M_1,\dots,N_g,\ee(X_1),\dots,\ee(X_n))$.

Up to rearranging the indices, we can assume that $\th_i\in\ZZ$
if and only if $i\leq k$.
Consider the map
$s:\wh{\Hom}_{\bm{\th}}(\grp_{g,n},\SU_2)\rar
\Hom_{\bm{\th}}(\grp_{g,n},\SU_2)\times (\Sph)^k$
defined as $s(\rho,\Axis):=(\rho,\hat{\Axis}(\beta_1),\dots,\hat{\Axis}(\beta_k))$.
The map $s$ is manifestly a real-analytic isomorphism and the
composition of $s$ and of the projection onto
$\Hom_{\bm{\th}}(\grp_{g,n},\SU_2)$ is exactly $\forg_{\bm{\th}}$: 
hence, $s$ gives the wished
trivialization of the fiber bundle $\forg_{\bm{\th}}$.
\end{proof}

Now we can complete the proof of our first main statement
on decorated representation spaces.

\begin{proof}[End of the proof of Lemma \ref{mainlemma:anal}]
(i) is very similar to the proof of Lemma \ref{mainlemma:alg}(i).

(ii)
The coaxial locus in $\wh{\Hom}(\grp_{g,n},\SU_2)$
is the preimage via the map $\forg$ (introduced
in Section \ref{ssc:analytic})
of the coaxial locus in $\Hom(\grp_{g,n},\SU_2)$, and
$\forg$ is a surjective and analytic by 
Lemma \ref{lemma:forgetful}. Hence,
the coaxial locus is closed analytic by Lemma \ref{mainlemma:alg}(ii).
Observe, similarly to the proof of Lemma \ref{mainlemma:alg}(ii), that
a decorated homomorphism $(\rho,\Axis)$ is elementary if and only
if $\IM(\Ad_{M_j}-I)$, $\IM(\Ad_{N_j}-1)$, $X_i^\perp$ do not span $\su_2$
for $(\bm{M},\bm{N},\bm{X})=\wh{\lambda}(\rho,\Axis)$. Such condition
can be expressed in terms of analytic equations in $(\bm{M},\bm{N},\bm{X})$,
and so the elementary locus is closed analytic.
Observe moreover that $(\rho,\Axis)\in\wti{\Sigma}$
if and only if $\rho$ is coaxial, $\|\Axis(\beta_i)\|\in \ZZ_+$ for $i=1,\dots,n$, and $\mathrm{Span}(\Axis(\beta_1),\dots,\Axis(\beta_n))\neq\su_2$: all such conditions can be expressed through analytic equations, so $\wti{\Sigma}$ is closed analytic.

Note that $\wti{\Sigma}_0$ is closed analytic, since the conditions $\rho(\mu_i),\rho(\nu_i),\rho(\beta_j)\in\{\pm I\}$
and $\mathrm{Span}(\Axis(\beta_1),\dots,\Axis(\beta_n))\leq 2$ can be phrased through analytic equations. Hence $\wti{\Sigma}_1$ is open inside $\wti{\Sigma}$.
If $g=0$, then $\wti{\Sigma}_1$ is empty as $(\rho,\Axis)$ with integer-valued $\Axis$
are necessarily central. Assume now $g\geq 1$ and let $(M_1,\dots,N_g,X_1,\dots,X_n)$
be an element of $\wti{\Sigma}_0$. This means that there exists $X\in\Sph$ such that
$X_1,\dots,X_n$ are orthogonal to $X$ and $M_i,N_i\in\{\pm I\}$.
Define $M_i(t)=M_i e^{tX}$ and $N_i(t)=N_i e^{tX}$, so that
$(M_1(t),\dots,N_g(t),X_1,\dots,X_n)_{t\geq 0}$ is a deformation
of $(M_1,\dots,N_g,X_1,\dots,X_n)$ that belongs to $\wti{\Sigma}_1$ for $t>0$: this proves that $\wti{\Sigma}_1$ is dense inside $\wti{\Sigma}$.

(iii)
Since $\wh{\Gcal}_{\bm{\th}}$ is a closed algebraic subset of $\wh{\Gcal}_+$,
it follows from (i) that $\wh{\Hom}_{\bm{\th}}(\grp_{g,n},\SU_2)$ is a closed analytic
subset of $\wh{\Hom}(\grp_{g,n},\SU_2)$.

(iv) follows from (ii) and (iii).
%

(a)
Clearly, an elementary $(\rho,\Axis)$ is coaxial.
Vice versa, pick a coaxial decorated homomorphism $(\rho,\Axis)$.
By definition, there exists a $1$-parameter subgroup $H$ of $\SU_2$
that contains the image of $\rho$.
Since $\th_i\notin\ZZ$, the element $\ee(\Axis(\beta_i))\in\SU_2$
is different from $\pm I$ and it belongs to $H$ for all $i$.
It follows that all $\Axis(\beta_i)$ belong to the Lie algebra of $H$
and so $(\rho,\Axis)$ is elementary.

(b) By Lemma \ref{lemma:forgetful} with $k=0$
we have a real-analytic isomorphism
$\wh{\Hom}_{\bm{\th}}(\grp_{g,n},\SU_2)\rar \Hom_{\bm{\th}}(\grp_{g,n},\SU_2)$.
Restricting to the non-coaxial locus and taking the quotient by $\PSU_2$,
we obtain the wished isomorphism.

Finally, note that the real-analytic maps in (o) and (b)
are $\PSU_2$-equivariant and so descend to decorated representation spaces.
Similarly, all subets involved in (ii-iii-iv) are $\PSU_2$-invariant
and so analogous statements hold for decorated representation spaces.
\end{proof}

\begin{example}[Semi-analytic nature of decorated $\SU_2$-representation spaces]\label{example:non-analytic}
Let $g=1$, $n=2$ and $\th_1=\th_2=t$ for some fixed $t\in (0,1)$.
Since no $\th_i$ is integer, 
the map $\forg_{\bm{\th}}:\wh{\Hom}_{\bm{\th}}(\grp_{1,2},\SU_2)\rar\Hom_{\bm{\th}}(\grp_{1,2},\SU_2)$ that forgets the decoration is an isomorphism of real analytic
spaces.
Proceeding as in Example \ref{example:non-algebraic},
it is possible to show that 
$\wh{\Rep}_{\bm{\th}}(\grp_{1,2},\SU_2)$ is not analytic (but only semi-analytic). As a consequence, so is $\wh{\Rep}(\grp_{1,2},\SU_2)$.
\end{example}

As a further simple application of the forgetful map $\forg_{\bm{\th}}$,
we have the following.

\begin{corollary}[Space of decorations of a fixed $\rho$]\label{cor:decorations}
Consider the fibration $\forg_{\bm{\th}}:\wh{\Hom}_{\bm{\th}}(\grp_{g,n},\SU_2)\rar\Hom_{\bm{\th}}(\grp_{g,n},\SU_2)$ in Section \ref{ssc:analytic}.
\begin{itemize}
\item[(i)]
If $\rho$ is coaxial, then
the locus of elementary homomorphisms inside
$\forg_{\bm{\th}}^{-1}(\rho)$ is closed analytic.
Moreover, such locus  has dimension $2$ if $\rho$ is central, and it has dimension $0$ if $\rho$ is non-central.
\item[(ii)]
Inside each fiber $\forg_{\bm{\th}}^{-1}(\rho)$ the non-elementary locus is open, dense and connected,
except if $k=0$ and $\rho$ is coaxial, or if $n=1$ and $\rho$ is central.
\end{itemize}
\end{corollary}
\begin{proof}
(i) The elementary locus is closed analytic by Lemma \ref{mainlemma:anal}(ii).
Identify now $\forg_{\bm{\th}}^{-1}(\rho)$ to $(\Sph)^k$ via the trivialization $s$
described in the proof of Lemma \ref{lemma:forgetful}.
If $\rho$ is central (and so $k=n>0$), then
the elementary locus corresponds to the subset of 
collinear $k$-tuples in $(\Sph)^k$, which is a closed
analytic subset of dimension $2$.
If $\rho$ is not central, with image inside $\exp(\hfrak)$ for some
line $\hfrak\subset\su_2$, then
the elementary locus corresponds to $(\Sph\cap\hfrak)^k$, which
consists of $2^k$ points.

(ii) For $k=0$ a decorated homomorphism is elementary if and only if it is coaxial
by Lemma \ref{mainlemma:anal}(a), and so the conclusion is immediate. So suppose $k>0$.
If $\rho$ is non-coaxial, then the whole $f^{-1}(\rho)$ is non-elementary.
If $\rho$ is coaxial but not central, then $f^{-1}(\rho)$ has dimension $2k\geq 2$
and the elementary locus therein has dimension $0$.
If $\rho$ is central, there are two cases: if $n=1$, then
$\forg_{\bm{\th}}^{-1}(\rho)$ consists entirely of elementary homomorphisms; if $n>1$,
then $\forg_{\bm{\th}}^{-1}(\rho)$ 
has dimension $2n\geq 4$ and the elementary locus therein has
dimension $2$. In all cases, the conclusion follows.
\end{proof}

\subsection{First-order computations: the maps $\wh{R}$, $\wh{R}_{\bm{\th}}$}

As for homomorphism spaces,
smoothness of $\wh{\Hom}(\grp_{g,n},\SU_2)$ and
$\wh{\Hom}_{\bm{\th}}(\grp_{g,n},\SU_2)$ are
investigated by studying
the differentials of the maps $\wh{R}$ and $\wh{R}_{\bm{\th}}$ introduced in Section \ref{sec:analytic}.
We will denote by $\SPAN{\Axis}\subseteq\su_2$ the smallest
vector subspace that contains the (possibly infinite) image of 
$\Axis:\Bcal\rar\su_2$.

\begin{proposition}[Differentials of $\wh{R}$ and $\wh{R}_{\bm{\th}}$]\label{prop:IFT}
For every $(\rho,\Axis)\in\wh{\Hom}_{\bm{\th}}(\grp_{g,n},\SU_2)$,
the images of the differentials
of $\wh{R}:\wh{\Gcal}\rar\SU_2$ and $\wh{R}_{\bm{\th}}:\wh{\Gcal}_{\bm{\th}}\rar\SU_2$
at $\wh{\lambda}(\rho)$ are
\begin{itemize}
\item[(i)] 
$\ \mathrm{Im}(d\wh{R}_{\bm{\th}})_{\wh{\lambda}(\rho,\Axis)}
=\Zfrak(\rho)^\perp$;
\item[(ii)]
$\ \displaystyle
\mathrm{Im}(d\wh{R})_{\wh{\lambda}(\rho,\Axis)}=
\begin{cases}
\su_2 & \text{if some $\th_i$ is non-integral}\\
\Zfrak(\rho)^\perp+\SPAN{\Axis} & 
\text{if all $\th_i$ are integral.}
\end{cases}
$
\end{itemize}
%
%
\end{proposition}

A preliminary observation concerns the differential of the map
$\EE:(\su_2\setminus\{0\})^n\rar \SU_2$ defined by
\[
\EE(X_1,\dots,X_n):=\ee(X_1)\cdots \ee(X_n).
\]

For every $\th>0$,
we denote by $\clal_\th$\index{$\clal_\th$}
the subset of $\su_2$ consisting of elements of norm $\th$
and by $\ee_\th$ the restriction of the map $\ee(X)$
to $\clal_{\th}$.
Given an angle vector $\bm{\th}$,
we let $\clal_{\bm{\th}}:=\prod_{i=1}^n \clal_{\th_i}$ and
we call $\EE_{\bm{\th}}$ the restriction of $\EE$ to $\clal_{\bm{\th}}$.


\begin{lemma}[Differential of $\EE$]\label{sub:E}
The exponential maps defined above satisfy the following properties.
\begin{itemize}
\item[(i)]
The differentials of $\ee$ and $\ee_\th$ at a point $X\in \clal_\th$
have image
\[
\begin{array}{ccc}
\dis\IM(d\ee_{X})=
\begin{cases}
\su_2 & \text{if $\th\notin\ZZ$}\\
\mathrm{Span}(X) & \text{if $\th\in\ZZ$,}
\end{cases}
&\phantom{XXX}&
\dis\IM(d\ee_{\th})_X=
\begin{cases}
X^\perp & \text{if $\th\notin\ZZ$}\\
\{0\} & \text{if $\th\in\ZZ$.}
\end{cases}
\end{array}
\]
\item[(ii)]
The differential of $\EE$ satisfies
\[
d\EE_{\bm{X}}(\bm{\dot{X}}):=d\ee_{X_1}(\dot{X}_1)+\sum_{i=2}^n \Ad_{\ee(X_1)\cdots \ee(X_{i-1})}d\ee_{X_i}(\dot{X}_i).
\]
\item[(iii)]
The differentials of $\EE$ and $\EE_{\bm{\th}}$ at a point $\bm{X}\in \clal_{\bm{\th}}$
have image
\[
\IM(d\EE)_{\bm{X}}=
\begin{cases}
\su_2 & \text{if some $\th_i\notin\ZZ$}\\
\mathrm{Span}(\bm{X}) & \text{if all $\th_i\in\ZZ$,}
\end{cases}
\]
\[
\IM(d\EE_{\bm{\th}})_{\bm{X}}=
\begin{cases}
\su_2 & \text{if there exist $i\neq j$ with $X_i,X_j$ linearly independent
and $\th_i,\th_j\notin\ZZ$}\\
\hfrak^\perp & \text{if all $X_i$ with $\th_i\notin\ZZ$ belong to a line $\hfrak\subset\su_2$}\\
\{0\} & \text{if all $\th_i\in\ZZ$.}
\end{cases}
\]
\end{itemize}
\end{lemma}
\begin{proof}
Part (i) is a straightforward
and part (ii) follows from Lemma \ref{sub:P}(ii).
The proof of part (iii) uses (i) and (ii), and is analogous to
the proof of Lemma \ref{sub:P}(iii).
\end{proof}

Similarly to Section \ref{sec:first-order},
upon identifying $\wh{\Hom}(\grp_{g,n},\SU_2)$
with the image of $\wh{\lambda}$, we have
\[
\wh{R}(\bm{M},\bm{N},\bm{X})=\Comm(\bm{M},\bm{N})\cdot \EE(\bm{X}).
\]

\begin{proof}[Proof of Proposition \ref{prop:IFT}]
We follow the same steps as in the proof of Proposition \ref{prop:IFT-undec},
with a few variations.
%
%
Quite as in the proof of Proposition \ref{prop:IFT-undec},
Lemma \ref{sub:P}(i) gives
\begin{equation}\label{eq:dR-image}
\IM(d\wh{R}_{(\bm{M},\bm{N},\bm{X})})=\IM(d\Comm_{(\bm{M},\bm{N})})+
\Ad_{\Comm(\bm{M},\bm{N})}(\IM(d\EE_{\bm{X}})).
\end{equation}
and
\begin{equation}\label{eq:dR-image-th}
\IM(d\wh{R}_{\bm{\th}})_{(\bm{M},\bm{N},\bm{X})}=\IM(d\Comm_{(\bm{M},\bm{N})})+
\Ad_{\Comm(\bm{M},\bm{N})}(\IM(d\EE_{\bm{\th}})_{\bm{X}}).
\end{equation}

Let now $(\bm{M},\bm{N},\bm{X})=\wh{\lambda}(\rho,\Axis)$.

(i) In order to compute the image of $d\wh{R}_{\bm{\th}}$, we separately consider three cases.

If $\dim\,\Zfrak(\bm{M},\bm{N})=0$, then 
$d\Comm_{(\bm{M},\bm{N})}=\su_2$ by Lemma \ref{sub:c} and so
$\IM(dR_{\bm{\th}})_{(\bm{M},\bm{N},\bm{X})}=\su_2$.

If $\Zfrak(\bm{M},\bm{N})=\mathfrak{h}$ is $1$-dimensional,
then Lemma \ref{sub:E}(iii) implies that
the image $(d\wh{R}_{\bm{\th}})_{(\bm{M},\bm{N},\bm{X})}$
is $\mathfrak{h}^\perp$ if all the non-integral $X_i$ lie in $\mathfrak{h}$, and it is $\su_2$ otherwise.

Suppose now that $\dim\, \Zfrak(\bm{M},\bm{N})=3$, and so all $M_i,N_i$ are $\pm I$.
By Lemma \ref{sub:E}(iii), the image of
$(d\wh{R}_{\bm{\th}})_{(\bm{M},\bm{N},\bm{X})}$
is: $\{0\}$, if all $\th_i$ are integer;
$\mathfrak{h}^\perp$, if all non-integral $X_i$ span
the same line $\mathfrak{h}$ inside $\su_2$;
$\su_2$, otherwise.

(ii) We now compute the image of $d\wh{R}$ and we consider Equation \eqref{eq:dR-image}.

If some $\th_i$ is non-integral, then $d\EE_{\bm{X}}$ is surjective
by Lemma \ref{sub:E}(iii) and so $d\wh{R}_{(\bm{M},\bm{N},\bm{X})}$ is too.

Suppose now that all $\th_i$ are integral, and so $e(X_i)=\pm I$ for all $i$.
%
By Lemma \ref{sub:E}(iii) the image of $d\EE_{\bm{X}}$ is equal to $\mathrm{Span}(\bm{X})$.
On the other hand, $\IM(d\Comm)_{(\bm{M},\bm{N})}$ is exactly 
$\Zfrak(\rho)^\perp=\Zfrak(\bm{M},\bm{N})^\perp$ by Lemma \ref{sub:c}.

Thus, if $\rho$ is non-coaxial, then
$\Zfrak(\rho)=\{0\}$ and $\IM(dR)_{(\bm{M},\bm{N},\bm{X})}=\su_2$.

If $\rho$ is coaxial, then $\Comm(\bm{M},\bm{N})=I$. It follows
that the second summand $\Ad_{\Comm(\bm{M},\bm{N})}(\IM(d\EE_{\bm{X}}))$
of \eqref{eq:dR-image}
contains $\mathrm{Span}(\Axis({\beta}_1),\dots,\Axis({\beta}_n))$
and is contained inside $\SPAN{\Axis}$.
Thus,
\[
\Zfrak(\rho)^\perp+\mathrm{Span}(\Axis({\beta}_1),\dots,\Axis({\beta}_n))
\subseteq \IM(d\wh{R})_{(\bm{M},\bm{N},\bm{X})}
\subseteq \Zfrak(\rho)^\perp+\SPAN{\Axis}.
\]
In order to show that the above inclusions are equalities,
it is enough to show that the generators of $\SPAN{\Axis}$,
namely $\Axis(\beta)$ for all $\beta\in\Bcal$,
belong to $\Zfrak(\rho)^\perp+\mathrm{Span}(\Axis({\beta}_1),\dots,\Axis({\beta}_n))$.
Indeed, for every $i$ and every $\beta'_i\in\Bcal_i$,
we can write $\beta'_i=\gamma\beta_i\gamma^{-1}$ for some $\gamma\in\grp_{g,n}$
and so $\Axis(\beta'_i)=\Ad_{\rho(\gamma)}\Axis(\beta_i)$.
Then the difference $\Axis(\beta'_i)-\Axis({\beta}_i)=
(\Ad_{\rho(\gamma)}-I)\Axis(\beta_i)$ belongs to 
$\Zfrak(\rho(\gamma))^\perp\subseteq \Zfrak(\rho)^\perp$,
and so $\Axis(\beta'_i)$ belongs to $\Zfrak(\rho)^\perp+\mathrm{Span}(\Axis({\beta}_1),\dots,\Axis({\beta}_n))$.
%
%
%
%
%
%
%
%
\end{proof}

%
%
%

\subsection{Tangent spaces to decorated homomorphism spaces}

Using Proposition \ref{prop:IFT} we can compute the dimensions of the tangent spaces to
our decorated homomorphism spaces. Recall that in the introduction
(Definition \ref{def:sigma}) we defined
the locus $\wti{\Sigma}=\wti{\Sigma}_0\cup\wti{\Sigma}_1$, where 
\begin{align*}
\wti{\Sigma}_0 & :=\left\{(\rho,\Axis)\ \big|\ \text{$\rho$ central, and $\SPAN{\Axis}\neq\su_2$}\right\}\\
\wti{\Sigma}_1 & :=\left\{(\rho,\Axis)\ \Big|\ 
\begin{array}{c}
\text{$\IM(\rho)\subset\exp(\hfrak)$ non-central},
\ \rho(\Bcal)\subseteq \{\pm I\},
\ \SPAN{\Axis}\subset\hfrak^\perp\\
\text{for some $1$-dimensional $\hfrak\subset\su_2$}
\end{array}\right\}
\end{align*}
inside $\wh{\Hom}(\grp_{g,n},\SU_2)$, and we denote by $\Sigma=\Sigma_0\cup\Sigma_1$
the corresponding loci in $\wh{\Rep}(\grp_{g,n},\SU_2)$.

\begin{corollary}[Tangent spaces to decorated homomorphisms spaces]\label{cor:rep}
Fix $\bm{\th}$ and let $(\rho,\Axis)\in\wh{\Hom}_{\bm{\th}}(\grp_{g,n},\SU_2)$.
Then
\begin{itemize}
\item[(i)]
The Zariski tangent space to the relative decorated homomorphisms space satisfies
\[
\dim\, T_{(\rho,\Axis)}\wh{\Hom}_{\bm{\th}}(\grp_{g,n},\SU_2) = 6g-3+2n+\dim(Z(\rho)).
\]
\item[(ii)]
The non-coaxial locus
$\wh{\Hom}_{\bm{\th}}^{nc}(\grp_{g,n},\SU_2)$
is an oriented manifold of real dimension $6g-3+2n$.
\item[(iii)]
The restriction of $\bm{\Theta}$ to $\wh{\Hom}^{nc}(\grp_{g,n},\SU_2)$ is real-analytic and submersive.
\item[(iv)]
The Zariski tangent space to the decorated homomorphisms space satisfies
\[
\dim\, T_{(\rho,\Axis)}\wh{\Hom}(\grp_{g,n},\SU_2) 
 = \left\{
 \begin{array}{lccr}
\dis 6g+3n-\dim\ \SPAN{\Axis} & \dis\text{if $(\rho,\Axis)\in\wti{\Sigma}_0$} &\phantom{XX} & \text{(a)\phantom{.}}
\\[\smallskipamount]
\dis 6g+3n-2 & \dis \text{if $(\rho,\Axis)\in\wti{\Sigma}_1$} && \text{(b)\phantom{.}}
\\[\smallskipamount]
\dis 6g+3n-3 & \dis \text{otherwise} && \text{(c).} 
\end{array}
\right.
\]
Away from the locus of decorated homomorphisms in $\wti{\Sigma}$, the space
$\wh{\Hom}(\grp_{g,n},\SU_2)$ is an oriented manifold of dimension $6g+3n-3$.
\end{itemize}
\end{corollary}

\begin{proof}
(i)
We recall that  the image $\wh{\lambda}(\wh{\Hom}_{\bm{\th}}(\grp_{g,n},\SU_2))$ is defined
by $\wh{R}_{\bm{\th}}=I$ inside the smooth, oriented variety
$\wh{\Gcal}_{\bm{\th}}$ of dimension $6g+2n$.
Thus the tangent space
to $\wh{\lambda}(\wh{\Hom}_{\bm{\th}}(\grp_{g,n},\SU_2))$
at $\wh{\lambda}(\rho,\Axis)=(\bm{M},\bm{N},\bm{X})$ is the kernel of
$(d\wh{R}_{\bm{\th}})_{(\bm{M},\bm{N},\bm{X})}$.
Since the image of $(d\wh{R}_{\bm{\th}})_{(\bm{M},\bm{N},\bm{X})}$
has dimension $\dim(\Zfrak(\rho)^\perp)=3-\dim(Z(\rho))$ by Proposition \ref{prop:IFT}(i),
it follows that the Zariski tangent space to
$\wh{\Hom}_{\bm{\th}}(\grp_{g,n},\SU_2)$ at $(\rho,\Axis)$
has dimension $6g+2n-3+\dim(Z(\rho))$.

(ii) If $(\rho,\Axis)\in\wh{\Hom}_{\bm{\th}}(\grp_{g,n},\SU_2)$ is non-coaxial, then
$Z(\rho)=\{\pm I\}$ by Lemma \ref{lemma:no-auto-undec} and
$\wh{R}_{\bm{\th}}$ is submersive at $\wh{\lambda}(\rho,\Axis)$
by Proposition \ref{prop:IFT}(i).
Since both $\wh{\Gcal}_{\bm{\th}}$ and $\SU_2$ are oriented,
the conclusion follows applying the implicit function theorem.

(iii) Let $(\rho,\Axis)\in\wh{\Hom}^{nc}_{\bm{\th}}(\grp_{g,n},\Axis)$.
To show that $\bm{\Theta}$ is real-analytic and submersive at $\wh{\lambda}(\rho,\Axis)=(\bm{M'},\bm{N'},\bm{X'})$
it is enough to show that the map $\wh{\Gcal}_+\rar \SU_2\times\RR^n_+$
that sends $(\bm{M},\bm{N},\bm{X})$ to $(\wh{R}(\bm{M},\bm{N},\bm{X}),\|X_1\|,\dots,\|X_n\|)$
is real-analytic submersive at $(\bm{M'},\bm{N'},\bm{X'})$.
As in the proof of Corollary \ref{cor:rep-undec}(iii),
this follows from the fact that $\wh{R}_{\bm{\th}}$ is submersive
at $(\bm{M'},\bm{N'},\bm{X'})$ by Proposition \ref{prop:IFT}(i), and that
the norm map $\|\cdot\|:\su_2\rar \RR_+$
is real-analytic and submersive away from $0$.

(iv)
Since $\wh{\lambda}(\wh{\Hom}(\grp_{g,n},\SU_2))$ is defined
by $\wh{R}=I$ inside $\wh{\Gcal}_+$, the tangent
space to $\wh{\lambda}(\wh{\Hom}(\grp_{g,n},\SU_2))$ at $\wh{\lambda}(\rho,\Axis)$
is the kernel of $d\wh{R}_{\wh{\lambda}(\rho,\Axis)}$.
Hence, the dimension of $T_{(\rho,\Axis)}\wh{\Hom}(\grp_{g,n},\SU_2)$
is $6g+2n-\mathrm{rk}(d\wh{R})_{\wh{\lambda}(\rho,\Axis)}$.
\begin{itemize}
\item[(a)]
Suppose that $(\rho,\Axis)\in\wti{\Sigma}_0$, namely $\rho$ takes values in $\pm I$
and the image of $\Axis$ does not span $\su_2$.
Then Proposition \ref{prop:IFT}(ii)
gives $\IM(d\wh{R})_{\wh{\lambda}(\rho,\Axis)}=\SPAN{\Axis}$.
\item[(b)]
Suppose that $(\rho,\Axis)\in\wti{\Sigma}_1$, namely $\rho$ is non-central and takes values in the subgroup generated by a $1$-dimensional subalgebra $\hfrak\subset\su_2$,
and $\Axis$ takes values in $\hfrak^\perp$ of norm in $2\pi\ZZ$.
Then 
Proposition \ref{prop:IFT}(ii)
gives $\IM(d\wh{R})_{\wh{\lambda}(\rho,\Axis)}=\hfrak^\perp$, which has dimension $2$.
\item[(c)]
Suppose that $(\rho,\Axis)\notin\wti{\Sigma}$.
If $\rho$ is non-coaxial or if some $\rho(\beta_i)\neq\pm I$, then 
Proposition \ref{prop:IFT}(ii) implies that $d\wh{R}_{\wh{\lambda}(\rho,\Axis)}$ has 
rank $3$.
If $\rho$ takes values in $\pm I$, then $\SPAN{\Axis}$ must be the whole $\su_2$
(otherwise $(\rho,\Axis)$ would belong to $\wti{\Sigma}_0$),
and so again $d\wh{R}_{\wh{\lambda}(\rho,\Axis)}$ has 
rank $3$ by Proposition \ref{prop:IFT}(ii).
If there exists a $1$-dimensional $\hfrak\subset\su_2$
such that $\rho$ takes values in $\hfrak$ but is not central and
$\rho(\Bcal)\subseteq\{\pm I\}$, then
$\Axis$ cannot take values inside $\hfrak^\perp$ (otherwise $(\rho,\Axis)$ would belong to $\wti{\Sigma}_1$) and so $\hfrak^\perp+\SPAN{\Axis}=\su_2$,
which implies that $d\wh{R}_{\wh{\lambda}(\rho,\Axis)}$ has 
rank $3$ by Proposition \ref{prop:IFT}(ii).
\end{itemize}
The last claim is a consequence of the implicit function theorem,
since $\wh{\Gcal}$ and $\SU_2$ are oriented manifolds
and $\wh{R}$ is submersive at $(\rho,\Axis)\notin\wti{\Sigma}$.
\end{proof}

\begin{lemma}[Coaxial and elementary decorated homomorphisms]\label{lemma:coax-elem}
Let $g\geq 0$ and $n>0$ so that $2g-2+n>0$.
\begin{itemize}
\item[(i)]
Inside $\wh{\Hom}(\grp_{g,n},\SU_2)$
the elementary locus has pure dimension $2g+n+1$,
the singular locus agrees with $\wti{\Sigma}$ and has pure dimension $2g+2n+2$,
the components of the coaxial locus have dimensions in
$[2g+n+1,2g+2n-1]\cup\{2g+2n+2\}$ if $g>0$, and $[n+1,2n]$ if $g=0$.
\end{itemize}
Let $\bm{\th}\in\RR^n_{>0}$ and let $k$ be the number of integer entries of $\bm{\th}$.
\begin{itemize}
\item[(ii)]
Inside $\wh{\Hom}_{\bm{\th}}(\grp_{g,n},\SU_2)$
the elementary and the coaxial loci are non-empty if and only if $\sum_i(\pm(\th_i-1))\in 2\ZZ$ for some choice of the signs.
The elementary locus is connected, irreducible, of dimension $2g+2$,
the coaxial locus is connected, irreducible, of dimension $2n$ if $(g,n)=(0,k)$,
and $2g+2k+2$ otherwise. Moreover, if $k>0$, then the coaxial non-elementary locus
is dense and connected inside the coaxial locus.
\end{itemize}
\end{lemma}
\begin{proof}
(i)
The elementary locus in $\wh{\Hom}(\grp_{g,n},\SU_2)$ is the image of
\[
\xymatrix@R=0in{
\Sph\times\RR^{2g}\times\RR^n_0\times\ZZ\ar[rr] && \wh{\Hom}(\grp_{g,n},\SU_2)\\
(X,s_1,t_1,\dots,t_g,\th_1,\dots,\th_n)\ar@{|->}[rr] &&
(\bm{\ee}(s_1 X,t_1 X,\dots,t_g X),\th_1 X,\dots,\th_n X)
}
\]
where $\RR^n_0=\{\bm{\th}\in\RR^n\,|\,\text{$\sum_i(\th_i-1)\in 2\ZZ$, and $\th_i\neq 0$ for all $i$}\}$. The fiber of such map is discrete, and so the elementary locus
has dimension $2g+n+1$.

Corollary \ref{cor:rep}(iv) implies that 
the singular locus inside $\wh{\Hom}(\grp_{g,n},\SU_2)$
is $\wti{\Sigma}$.

For $g=0$ we have $\wti{\Sigma}=\wti{\Sigma}_0$ by Lemma \ref{mainlemma:anal}(ii)
and $\wti{\Sigma}_0$ can be described as the locus of coplanar $(X_1,\dots,X_n)$
in $\su_2^{\oplus n}$ such that $\bm{\th}\in\ZZ_+$, 
which is analytic of dimension $2n+2$.

Let now $g\geq 1$ and consider the $(2g+2n+2)$-dimensional analytic subset 
of $\Sph\times(\RR^{2g}\setminus \ZZ^{2g})\times \su_2^{\oplus n}$ defined as
\[
\wti{\mathcal{S}}_1:=
\left\{(X,s_1,t_1,\dots,t_g,X_1,\dots,X_n)\ \big|\ 
\sum_j(\|X_j\|-1)\in 2\ZZ,\ \text{and}
\ \text{$0\neq X_j\perp X$ for all $j$}
\right\}.
\]
The map $\wti{\mathcal{S}}_1\rar\wti{\Sigma}_1$
that sends $(X,s_1,\dots,X_n)$ to $(\bm{\ee}(s_1 X,\dots,t_g X),X_1,\dots,X_n)$
is real-analytic, surjective, with discrete generic fiber: it follows that
$\wti{\Sigma}_1$ has pure dimension $2g+2n+2$.
By Lemma \ref{mainlemma:anal}(ii), $\wti{\Sigma}_1$ is dense in $\wti{\Sigma}$
and so $\wti{\Sigma}$ has pure dimension $2g+2n+2$.

The coaxial locus is the preimage of the coaxial locus
in $\Hom(\grp_{g,n},\SU_2)$ via the forgetful map $\forg$.
Such locus can be decomposed into subloci indexed by the number $k$
of integer entries of $\bm{\th}$.
Indeed, consider the $(2g+n+k+2)$-dimensional
analytic subset $\wti{\mathcal{C}}_k$
of $\Sph\times\RR^{2g}\times\RR_{>0}^{n-k}\times 
(\Sph\times\ZZ_{\neq 0})^k$ consisting of 
$(X,s_1,\dots,t_g,\th_1,\dots,\th_{n-k},X_{n-k+1},\dots,X_n)$
such that $\sum_{i=1}^{n-k}(\th_i-1)+\sum_{j=n-k+1}^n (\|X_j\|-1)\in 2\ZZ$.
Then the image of the map
$\wti{\mathcal{C}}_k\rar \wh{\Hom}(\grp_{g,n},\SU_2)$
that sends $(X,s_1,\dots,t_g,\th_1,\dots,\th_{n-k},X_{n-k+1},\dots,X_n)$
to $(\bm{\ee}(s_1 X,\dots,t_g X),\th_1 X,\dots,\th_{n-k}X,X_{n-k+1},\dots,X_n)$
corresponds to the subset of coaxial $(\rho,\Axis)$
such that $\th_1,\dots,\th_{n-k}\notin\ZZ$ and $\th_{n-k+1},\dots,\th_n\in\ZZ$.
If $g=0$ and $k=n$, then such map has 
$(2n+2)$-dimensional domain and $2$-dimensional fiber, and so
its image has dimension $2n$. 
If $g=0$ and $k\leq n-2$, then such map has 
$(n+k+1)$-dimensional domain and discrete fiber,
and so its image has dimension $n+k+1\in[n+1,2n-1]$.
If $g>0$ and $k=n$, then the map has $(2g+2n+2)$-dimensional domain and discrete fiber,
and so its image has dimension $2g+2n+2$.
If $g>0$ and $k\leq n-2$, 
then the map has $(2g+n+k+1)$-dimensional domain and discrete fiber,
and so its image has dimension $2g+n+k+1\in[2g+n+1,2g+2n-1]$.

(ii)
The non-emptiness claim for the coaxial locus
follows from Proposition \ref{prop:coaxial}(ii).
Assume now that $\sum_i\e_i(\th_i-1)\in 2\ZZ$ for certain $\e_1,\dots,\e_n\in\{\pm 1\}$.
The following analytic map
\[
\xymatrix@R=0in{
\Sph\times \RR^{2g}\ar[rr] && \wh{\Hom}_{\bm{\th}}(\grp_{g,n},\SU_2)\\
(X,s_1,t_1,\dots,t_g)\ar@{|->}[rr] && (\bm{\ee}(s_1 X,\dots,t_g X),\th_1 X,\dots,\th_n X)
}
\]
has discrete fibers, and its image is the elementary locus: hence, the elementary locus is connected, irreducible, of dimension $2g+2$.
As for the coaxial locus, suppose that $\th_1,\dots,\th_k\in\ZZ$ and $\th_{k+1},\dots,\th_n\notin\ZZ$, and consider the following analytic map
\[
\xymatrix@R=0in{
\Sph\times \RR^{2g}\times(\Sph)^k\ar[rr] && \wh{\Hom}_{\bm{\th}}(\grp_{g,n},\SU_2)\\
(X,s_1,t_1,\dots,t_g,\hat{X}_1,\dots,\hat{X}_k)\ar@{|->}[rr] && (\bm{\ee}(s_1 X,\dots,t_g X),\th_1 \hat{X}_1,\dots,\th_k\hat{X}_k,\th_{k+1}X,\dots,\th_n X)
}
\]
Its image is the coaxial locus, which is thus connected and irreducible.
It $g\geq 1$ or if $k<n$, then the fibers are discrete and so the coaxial locus has dimension $2g+2k+2$.
If $g=0$ and $k=n$, then the fibers are $2$-dimensional and so the coaxial locus has dimension $2n$.

Note finally that the elementary locus is the image via the latter map above 
of the subset $(X,s_1,\dots,t_g,X,\dots,X)$ inside $\Sph\times \RR^{2g}\times(\Sph)^k$:
hence, the non-elementary locus is dense and connected inside the coaxial locus if $k>0$.
\end{proof}


\subsection{Density and connectedness of the non-coaxial and non-elementary decorated loci}\label{sec:connectedness-dec}

In this short section we investigate the non-coaxial and the non-elementary loci
of decorated homomorphism spaces, first in the relative case and then in the absolute case.

The following statement relies on the properties of the map
$\forg_{\bm{\th}}:\wh{\Hom}_{\bm{\th}}(\grp_{g,n},\SU_2)\rar\Hom_{\bm{\th}}(\grp_{g,n},\SU_2)$, introduced in Section \ref{ssc:analytic},
and on density and connectedness of the non-coaxial locus
in undecorated relative homomorphism spaces proven in Theorem \ref{thm:connected-rel}.

\begin{proposition}[Density and connectedness of non-coaxial and non-elementary relative decorated loci]\label{prop:dens-conn-nc-ne}
Let $g\geq 0$ and $\bm{\th}\in\RR^n_{>0}$.
\begin{itemize}
\item[(i)]
If $(g,n,\bm{\th})$ is non-special, then the non-coaxial locus
in $\wh{\Hom}_{\bm{\th}}(\grp_{g,n},\SU_2)$ is non-empty, dense and connected;
and so $\wh{\Hom}_{\bm{\th}}(\grp_{g,n},\SU_2)$ has pure dimension $6g-3+2n$.
If $(g,n,\bm{\th})$ is special, then the non-coaxial locus is empty.
\item[(ii)]
The non-elementary locus inside $\wh{\Hom}_{\bm{\th}}(\grp_{g,n},\SU_2)$ is dense and connected.
\end{itemize}
Analogous claims hold in $\wh{\Rep}_{\bm{\th}}(\grp_{g,n},\SU_2)$.
\end{proposition}
\begin{proof}
Recall that, by Lemma \ref{lemma:forgetful}, the map
$\forg_{\bm{\th}}$
is a $(\Sph)^k$-bundle that preserves the coaxial loci, where $k$
is the number of integer entries of $\bm{\th}$.
Clearly, it will be enough to prove
(i) and (ii) for $\wh{\Hom}_{\bm{\th}}(\grp_{g,n},\SU_2)$,
and the conclusion will hold for $\wh{\Rep}_{\bm{\th}}(\grp_{g,n},\SU_2)$ too.

(i) Consider first the non-special case.
Recall that undecorated
non-special $\Hom_{\bm{\th}}^{nc}(\grp_{g,n},\SU_2)$ are smooth of dimension $6g-3+2(n-k)$, dense and connected by
Theorem \ref{thm:connected-rel}. 
By the above considerations on $\forg_{\bm{\th}}$, the decorated
non-special $\wh{\Hom}_{\bm{\th}}^{nc}(\grp_{g,n},\SU_2)$ are smooth of dimension $6g-3+2n$, dense and connected;
so $\wh{\Hom}_{\bm{\th}}(\grp_{g,n},\SU_2)$ has pure dimension $6g-3+2n$.

Consider now the special cases.
Recall that the undecorated special $\Hom_{\bm{\th}}^{nc}(\grp_{g,n},\SU_2)$
are empty by Theorem \ref{mainthm:special-undec}. 
By the above considerations 
on $\forg$, the same holds in the decorated case.

(ii) In non-special cases, the conclusion follows from (i).
In the special cases, $k=n>0$ and all homomorphisms are coaxial. So it is enough to show that the subset of non-elementary coaxial homomorphisms are dense inside the subset of coaxial homomorphisms: this is proven in Lemma \ref{lemma:coax-elem}(ii).
\end{proof}

In order to prove density and connectedness for absolute decorated homomorphism spaces,
we want to view $\wh{\Hom}^{nc}(\grp_{g,n},\SU_2)$ as the total space of a fibration,
whose fibers are the spaces $\wh{\Hom}^{nc}_{\bm{\th}}(\grp_{g,n},\SU_2)$,
and then exploit Proposition \ref{prop:dens-conn-nc-ne}.
Thus, we first need to understand the base space of such fibration.

We denote by $\bm{\Theta}^{nc}$ the restriction
of $\bm{\Theta}:\wh{\Hom}(\grp_{g,n},\SU_2)\rar\RR^n$ the the non-coaxial locus
and let $\ANGL_{g,n}$ and $\ANGL^{nc}_{g,n}$ be the images of $\bm{\Theta}$ and $\bm{\Theta}^{nc}$ respectively.

\begin{lemma}[Angles realized by decorated homomorphisms]\label{lemma:angles}
For $g\geq 0$ and $n>0$ such that $2g-2+n>0$
\[
\ANGL_{g,n}=\begin{cases}
\RR^n_{>0} & \text{if $g\geq 1$}\\
\{\bm{\th}\in\RR^n_{>0}\,|\,d_1(\bm{\th}-\bm{1},\ZZ^n_o)\geq 1\} & \text{if $g=0$}
\end{cases}
\]
and it is always connected; on the other hand,
\[
\ANGL^{nc}_{g,n}=\begin{cases}
\RR^n_{>0} & \text{if $g\geq 2$}\\
\RR^n_{>0}\setminus\{\bm{\th}\in\ZZ^n\,|\,\sum_1(\th_i-1)\in 2\ZZ\} & \text{if $g=1$}\\
\{\bm{\th}\in\RR^n_{>0}\,|\,d_1(\bm{\th}-\bm{1},\ZZ^n_o)>1\} & \text{if $g=0$}
\end{cases}
\]
and it is connected if and only if $(g,n)\neq (0,3),(1,1)$.
\end{lemma}
\begin{proof}
The characterization of $\ANGL_{g,n}$ and $\ANGL^{nc}_{g,n}$
follows from Proposition \ref{prop:conn-genus0} for genus $0$,
from Proposition \ref{prop:genus1}(ii) in the special cases of genus $1$,
and from Theorem \ref{thm:density-non-dec} in non-special cases.

The connectedness of $\ANGL_{g,n}$ and $\ANGL^{nc}_{g,n}$ for genus $0$
is analyzed in \cite[Lemma B]{MP1}; the case of genus $1$ is straightforward.
\end{proof}

Now we can treat the non-coaxial locus inside absolute
decorated homomorphism spaces.

\begin{proposition}[Density and connectedness of non-coaxial and non-elementary decorated loci]\label{prop:dens-conn-nc-ne-abs}
Let $g\geq 0$ and $n>0$ with $2g-2+n>0$. Inside $\wh{\Hom}(\grp_{g,n},\SU_2)$
\begin{itemize}
\item[(i)]
the non-coaxial locus is dense;
moreover it is connected if and only if $(g,n)\neq (0,3),(1,1)$;
\item[(ii)]
the non-elementary locus is dense and connected;
\item[(iii)]
$\wh{\Hom}(\grp_{g,n},\SU_2)$ has pure dimension $6g-3+3n$.
\end{itemize}
The same holds for $\wh{\Rep}(\grp_{g,n},\SU_2)$.
\end{proposition}
\begin{proof}
Again, the conclusion for $\wh{\Rep}(\grp_{g,n},\SU_2)$ will immediately follow
once claims (i) and (ii) for decorated homomorphism spaces are proven.

(i)
Observe non-special points are dense in $\wh{\Hom}(\grp_{g,n},\SU_2)$.
By Proposition \ref{prop:dens-conn-nc-ne}(i) non-coaxial points are dense inside the subset of non-special points. It follows that the non-coaxial locus is dense in
$\wh{\Hom}(\grp_{g,n},\SU_2)$.
Recall also that the map 
$\bm{\Theta}^{nc}:\wh{\Hom}^{nc}(\grp_{g,n},\SU_2)\rar\ANGL^{nc}_{g,n}$
is surjective and submersive (and so open) by Corollary \ref{cor:rep}(iii),
and that its fibers are connected by Proposition \ref{prop:dens-conn-nc-ne}(i).
By Lemma \ref{lemma:connected}(i) it follows that the connected components of
$\wh{\Hom}^{nc}(\grp_{g,n},\SU_2)$ are exactly the inverse images via $\bm{\Theta}^{nc}$
of the connected components of $\ANGL^{nc}_{g,n}$.
The conclusion follows by Lemma \ref{lemma:angles}.

(ii) Density of the non-elementary locus follows from the density of the non-coaxial locus in (i). 
As for the connectedness, if $(g,n)\neq (0,3),(1,1)$, then again it follows from the connectedness of the non-coaxial locus in (i). 

Consider the case $(g,n)=(1,1)$. The subset $(\bm{\Theta}^{nc})^{-1}(2m-1,2m+1)$ is connected for every integer $m\geq 1$. Thus we only need to connect a non-elementary point
in $\bm{\Theta}^{-1}(2m+1)$ to a point in $\bm{\Theta}^{-1}(2m+1-\e)$ and to a point in
$\bm{\Theta}^{-1}(2m+1+\e)$. By Corollary \ref{cor:comm}(i) every non-elementary
point in $\bm{\Theta}^{-1}(2m+1)$ can be connected inside $\wh{\Hom}^{ne}_{2m+1}(\grp_{1,1},\SU_2)$ to the point $(M(0),N(0),X(0))=(I,D,(2m+1)R)$, where
$D=\begin{pmatrix}i & 0\\ 0 & -i\end{pmatrix}$ and $R=\frac{1}{2}\begin{pmatrix}0 & -1\\ 1 & 0\end{pmatrix}$. So we can for example use the path
$(M(t),N(t),X(t))=(\exp(\pi t R), D,(2m+1+t)R)$ with $t\in(-\e,\e)$.

Consider finally the case $(g,n)=(0,3)$.
By Lemma \ref{lemma:angles} the set $\ANGL_{0,3}$ is connected.
Indeed, if $\mathpzc{C}=[a,a+1]\times[b,b+1]\times[c,c+1]$ with $a,b,c\in\ZZ_{\geq 0}$ and $(a,b,c)\neq (0,0,0)$, then $\ANGL_{0,3}\cap \mathpzc{C}$ is a closed tetrahedron with vertices
in the even vertices of $\mathpzc{C}$. Thus, up to reordering the labels,
it is enough to show that
a non-elementary point $(X_1(0),X_2(0),X_3(0))$ 
in $\wh{\Hom}_{\bm{\th}}(\grp_{0,3},\SU_2)$
with $\th_3\in\ZZ$ and $\th_1,\th_2\notin\ZZ$ 
is part of a path $(X_1(t),X_2(t),X_3(t))$ such that 
$t\mapsto \|X_3(t)\|$ is strictly increasing. We choose $X_3(t)=(1+t)X_3(0)$ and $X_2(t)=X_2(0)$.
Since $\th_1\notin\ZZ$, for $t\in(-\e,\e)$ there exists a unique continuous path $X_1(t)$ with given $X_1(0)$ and such that $\ee(X_1(t))\ee(X_2(t))\ee(X_3(t))=I$. The conclusion follows.

(iii) follows from the fact that the non-coaxial locus inside
$\wh{\Hom}(\grp_{g,n},\SU_2)$ is dense by (i) and
is smooth of dimension $6g-3+3n$ by Corollary \ref{cor:rep}(iv).
\end{proof}

\subsection{Absolute, non-special and special decorated homomorphism spaces}

In this section we prove our main results on decorated absolute and relative homomorphism and representation spaces.

For absolute decorated spaces, all the work has already been done.

\begin{proof}[Proof of Theorem \ref{mainthm:rep}]
(i) is the content of Lemma \ref{lemma:no-auto}.

(ii-a,b,c) follow from Lemma \ref{lemma:coax-elem}(i).

(ii-d) It follows from Corollary \ref{cor:rep}(iv) it is an
oriented manifold of dimension $6g-3+3n$.

(ii-e) follows from (ii-a), since elementary decorated homomorphisms have
$1$-dimensional stabilizer.

(ii-f) follows from (ii-b) by Lemma \ref{lemma:no-auto}.

(ii-g) follows from (ii-c) and Remark \ref{rmk:orbi} and Lemma \ref{lemma:no-auto}.

(iii) is proven in Proposition \ref{prop:dens-conn-nc-ne-abs}(iii).
\end{proof}

Also for relative decorated spaces
we have already proven almost everything.

\begin{proof}[Proof of Theorem \ref{mainthm:rep-rel}]
(i) follows from Lemma \ref{lemma:coax-elem}(ii).

(ii) follows from Corollary \ref{cor:rep}(i-ii). 

(iii) follows from (i), since elementary decorated homomorphisms have $1$-dimensional stabilizer.

(iv)
By Lemma \ref{mainlemma:anal}(a), if $k=0$, then coaxial is equivalent to elementary: in this case, coaxial decorated homomorphisms have $1$-dimensional stabilizer.
If $k>0$, then the general coaxial decorated homomorphism
is non-elementary by Lemma \ref{lemma:coax-elem}(ii): in this case, the general coaxial decorated homomorphism has trivial stabilizer by Lemma \ref{lemma:no-auto}.
The conclusion then follows from (i).

(v) follows from Remark \ref{rmk:orbi} and Lemma \ref{lemma:no-auto}.

(vi) is proven in Proposition \ref{prop:dens-conn-nc-ne}(ii).
\end{proof}

The wished conclusion for the non-coaxial loci of non-special relative
decorated spaces follows.

\begin{proof}[Proof of Theorem \ref{mainthm:rep-nonsp}]
Recall that smoothness of the non-coaxial locus is proven in
Theorem \ref{mainthm:rep-rel}(ii-v) together with the determination of its dimension. 

(iii-iv) Non-emptiness, density and connectedness of the non-coaxial locus 
is proven in Proposition \ref{prop:dens-conn-nc-ne}(i).

(i) Pure dimensionality of the relative decorated homomorphism and representation spaces follows from (iii)-(iv).

(ii) Because of pure-dimensionality, Corollary \ref{cor:rep}(i-ii) also shows that the smooth locus of $\wh{\Hom}_{\bm{\th}}(\grp_{g,n},\SU_2)$ coincides with the non-coaxial locus. It follows that
$\wh{\Hom}_{\bm{\th}}(\grp_{g,n},\SU_2)$ is reduced and irreducible.

Concerning the relative decorated representation space, it is reduced and irreducible and
its smooth locus consists of $[\rho,\Axis]$ at which
$\mathrm{dim}(T_{(\rho,\Axis)}\wh{\Hom}_{\bm{\th}}(\grp_{g,n},\SU_2))+\mathrm{dim}(Z(\rho,\Axis))$ is minimum by Remark \ref{rmk:X/G}. Such minimum is achieved at the non-coaxial locus by
Corollary \ref{cor:rep}(i), which is non-empty by (iii).
\end{proof}

The analysis of special decorated spaces is more explicit.

\begin{proof}[Proof of Theorem \ref{mainthm:special}]
%
%
Preliminarly observe that the map
\[
\xymatrix@R=0in{
\Sph\times\{\pm 1\}^{m-1}\ar[rr] && \COAX((\Sph)^m)\\
(X,\e_2,\dots,\e_m) \ar@{|->}[rr] && (X,\e_2 X,\dots,\e_m X)
}
\]
is an isomorphism.
For $m=1$, clearly $\COAX(\Sph)=\Sph$. For $m\geq 2$,
the coaxial locus in $(\Sph)^m$ is a compact submanifold of codimension $2m-2$:
hence the non-coaxial locus in $(\Sph)^m$ is open, dense and connected.

(i-a)
The central locus in $\wh{\Hom}_{\bm{\th}}(\grp_{1,n},\SU_2)$
consists of $(M_1,N_1,X_1,\dots,X_n)$ such that $M_1,N_1\in\{\pm I\}$
and $|X_i\|=\th_i$ for all $i$. As a consequence, it is isomorphic to $\{\pm I\}^2\times(\Sph)^n$. Inside it, the elementary locus 
of $\wh{\Hom}_{\bm{\th}}(\grp_{1,n},\SU_2)$
corresponds to $\{\pm I\}\times\COAX((\Sph)^n)$. 
The conclusion follows from (o).

(i-b) The proof of Proposition \ref{prop:genus1}(ii) shows that
the non-central locus
in $\Hom_{\bm{\th}}(\grp_{1,n},\SU_2)$ is isomorphic to
the quotient of $\Sph\times (\RR^2\setminus\ZZ^2)/(2\ZZ)^2$ by $\{\pm 1\}$,
which is an $\Sph$-bundle over $S^2\setminus\{\text{4 points}\}$.
By Lemma \ref{lemma:forgetful},
the non-central locus of $\wh{\Hom}_{\bm{\th}}(\grp_{1,n},\SU_2)$
is isomorphic to the quotient of $(\Sph)^{1+n}\times (\RR^2\setminus\ZZ^2)/(2\ZZ)^2$ by $\{\pm 1\}$, which is an $(\Sph)^{1+n}$-bundle over 
$S^2\setminus\{\text{4 points}\}$.
The elementary non-central locus corresponds to 
the $\COAX((\Sph)^{1+n})$-subbundle over $S^2\setminus\{\text{4 points}\}$.

(i-c) The claim for $\wh{\Hom}_{\bm{\th}}(\grp_{1,n},\SU_2)$ is a consequence of (i-b) and (o). The corresponding conclusion for $\wh{\Hom}_{\bm{\th}}(\grp_{1,n},\SU_2)$ follows from Remark \ref{rmk:orbi} and Lemma \ref{lemma:no-auto}.

(ii) is a consequence of Theorem \ref{mainthm:special-undec}(ii).

(iii) By Theorem \ref{mainthm:special-undec}(iii),
$\Hom_{\bm{\th}}(\grp_{0,n},\SU_2)$ consists of a single conjugacy class,
and it is isomorphic to $\Sph$ if $k\leq n-2$, or to a point if $k=n$.

(iii-a) follows from Lemma \ref{mainlemma:anal}(b).

(iii-b) Assume $\th_1,\dots,\th_k\in\ZZ$ and $\th_{k+1},\dots,\th_n\notin\ZZ$.
By Theorem \ref{mainthm:special-undec}(iii) the map
$\Hom_{\bm{\th}}(\grp_{0,n},\SU_2)\rar \Sph$ that sends
$(X_1,\dots,X_n)$ to $\hat{X}_{k+1}$ is an isomorphism.
As a consequence,
\[
\xymatrix@R=0in{
\wh{\Hom}_{\bm{\th}}(\grp_{0,n},\SU_2) \ar[rr] && (\Sph)^{k+1}\\
(X_1,\dots,X_n) \ar@{|->}[rr] && (\hat{X}_1,\dots,\hat{X}_{k+1})
}
\]
is an isomorphism.
Clearly, the elementary locus corresponds to $\COAX((\Sph)^{k+1})$.
The conclusion then follows from (o), Remark \ref{rmk:orbi} and Lemma \ref{lemma:no-auto}.

(iii-c) By Lemma \ref{lemma:forgetful} the map
\[
\xymatrix@R=0in{
\wh{\Hom}_{\bm{\th}}(\grp_{0,n},\SU_2) \ar[rr] && (\Sph)^n\\
(X_1,\dots,X_n) \ar@{|->}[rr] && (\hat{X}_1,\dots,\hat{X}_n)
}
\]
is an isomorphism. The elementary locus corresponds
to $\COAX((\Sph)^n)$.
The conclusion for $\wh{\Rep}^{ne}_{\bm{\th}}(\grp_{0,n},\SU_2)$
follows by Remark \ref{rmk:orbi} and Lemma \ref{lemma:no-auto}.

(iii-d)
Since $\wh{\Hom}_{\bm{\th}}(\grp_{0,n},\SU_2)$ consists of a single conjugacy class,
$\wh{\Rep}_{\bm{\th}}(\grp_{0,n},\SU_2)$ consists of one point.
By Lemma \ref{lemma:forgetful} the analytic structure on
$\wh{\Hom}_{\bm{\th}}(\grp_{0,n},\SU_2)$
is reduced if and only if the algebraic structure on
$\Hom_{\bm{\th}}(\grp_{0,n},\SU_2)$ is reduced.
Hence, the conclusion follows from Theorem \ref{mainthm:special-undec}(iii-c).
\end{proof}

In \cite[Theorem 1]{goldman-millson} it was shown that, for $n=0$, the singularities of the homomorphisms
spaces are cut by quadrics in suitable analytic coordinates.
The singularities of relative representation spaces
and of relative decorated representations spaces, and in particular
of their non-elementary locus, are worthwhile further investigations.

%
%
%

\subsection{Symplectic structures}\label{sec:symplectic}

%
%

We begin by recalling the definition of Goldman's symplectic structure on $\Rep^{nc}_{\bm{\th}}(\grp_{g,n},\SU_2)$, following Guruprasad-Huebschmann-Jeffrey-Weinstein
\cite{guruprasad}. We refer to \cite{guruprasad} for full treatment of the topic (see also \cite[Section 1.4]{goldman:poisson}).


The classical Riemann-Hilbert correspondence establishes 
a bijective correspondence between flat $\su_2$-vector bundles $E$ on $\dot{S}$ with monodromy in $\PSU_2$ (up to isomorphism)
and representations $\bar{\rho}:\pi_1(\dot{S})\rar \PSU_2$,
and it works as follows.
For every $\bar{\rho}$ we define $E_{\bar{\rho}}:=\su_2\times\wti{\dot{S}}/\pi_1(\dot{S})$, where
$\pi_1(\dot{S})$ acts on $\su_2$ via $\mathrm{Ad}\circ\bar{\rho}:\pi_1(\dot{S})\rar\mathrm{Aut}(\su_2)\cong \PSU_2$.
Vice versa, given a flat $E$ we define $\bar{\rho}_E$ to be the monodromy representation of $E$.

Since $\PSU_2$ is compact,
first-order deformations of $[\bar{\rho}]$ are parametrized by the de Rham cohomology group $H^1(\dot{S},E_{\bar{\rho}})$.
Moreover, first-order deformations that do not change the conjugacy class of the image of the ${\beta}_i$'s correspond to elements of the {\it{parabolic cohomology}} group $H^1_{\mathrm{par}}(\dot{S},E_{\bar{\rho}})$,
namely classes in $H^1(\dot{S},E_{\bar{\rho}})$ that can be represented by compactly supported $1$-forms on $\dot{S}$ (see \cite[Section 6]{weil}).

Now, recall that $\grp_{g,n}$ is the fundamental group of $\dot{S}$ with a chosen basepoint.
If $\bar{\rho}:\grp_{g,n}\rar\PSU_2$ is induced by
some
$\rho:\grp_{g,n}\rar\SU_2$, then first-order deformations of $\rho$ bijectively correspond to first-order deformation of $\bar{\rho}$, namely
$T_{[\rho]}\Rep_{\bm{\th}}(\grp_{g,n},\SU_2)\cong
T_{[\ol{\rho}]}\Rep_{\bm{\th}}(\grp_{g,n},\PSU_2)$.
Hence,
\[
T_{[\rho]}\Rep_{\bm{\th}}(\grp_{g,n},\SU_2)\cong H^1_{\mathrm{par}}(\dot{S},E_{\bar\rho})
\subset
H^1(\dot{S},E_{\bar\rho})\cong T_{[\rho]}\Rep(\grp_{g,n},\SU_2).
\]
The Goldman 2-vector field $\Lambda$ on $\Rep(\grp_{g,n},\SU_2)$ defined as
\[
\Lambda_{[\rho]}:T^*_{[\rho]}\Rep(\grp_{g,n},\SU_2)\cong H^1(\dot{S},E_{\bar\rho})\rar H^1_c(\dot{S},E_{\bar\rho})\cong T_{[\rho]}\Rep(\grp_{g,n},\SU_2)
\]
determines a Poisson structure.
Now, the space $\Rep(\grp_{g,n},\SU_2)$ is in general singular,
but its non-coaxial locus is smooth by Theorem \ref{mainthm:rep-undec-rel}(iv).

In \cite[Section 9]{guruprasad} it is also shown that $\Rep^{nc}_{\bm{\th}}(\grp_{g,n},\SU_2)$ is a symplectic leaf for such Poisson structure.
In fact, if we denote by $\Omega$ the symplectic form induced by $\Lambda$
on $\Rep^{nc}_{\bm{\th}}(\grp_{g,n},\SU_2)$, 
then 
\[
\Omega_{[\rho]}:T_{\rho}\Rep^{nc}_{\bm{\th}}(\grp_{g,n},\SU_2)\otimes
T_{\rho}\Rep^{nc}_{\bm{\th}}(\grp_{g,n},\SU_2)\lra\RR
\]
at a non-coaxial representation $[\rho]$
can be identified to the alternate pairing
\[
\xymatrix@R=0in{
H^1_{\mathrm{par}}(\dot{S},E_{\bar\rho})\times H^1_{\mathrm{par}}(\dot{S},E_{\bar\rho})\ar[rr] && \RR\\
([\phi],[\psi])\ar@{|->}[rr] && \int_S \Kill(\phi\wedge\psi)
}
\]
where the representatives $\phi,\psi$ are chosen to 
vanish in a neighbourhood of $\bm{x}$. Note that such alternate pairing 
is non-degenerate by Poincar\'e duality and by the non-degeneracy of the Killing form $\Kill$.

The above considerations still hold when replacing $\SU_2$
by $\SL_2(\CC)$: in this case, we obtain a holomorphic
symplectic form $\Omega_{\CC}$ on $\Rep^{nc}_{\bm{\th}}(\grp_{g,n},\SL_2(\CC))$.

\begin{proof}[Proof of Proposition \ref{mainprop:sympl}]
(i) It is enough to observe that 
the inclusion $\Hom^{nc}_{\bm{\th}}(\grp_{g,n},\SU_2)\rar \Hom^{nc}_{\bm{\th}}(\grp_{g,n},\SL_2(\CC))$ is a real-analytic map
and that both $\Rep^{nc}_{\bm{\th}}(\grp_{g,n},\SU_2)$ and $\Rep^{nc}_{\bm{\th}}(\grp_{g,n},\SL_2(\CC))$
are manifolds by Theorem \ref{mainthm:rep-undec-rel}(ii-iv).

(ii)
The compatibility between $\Omega_\CC$ and $\Omega$ is clear, since the Killing form on $\SU_2$ is the restriction of the Killing form on $\SL_2(\CC)$.
Since $\Omega_{\CC}$ is holomorphic, its restriction $\Omega$ is real-analytic.
The non-degeneracy of $\Omega$ and $\Omega_{\CC}$ discussed above is proven
in \cite{guruprasad}.
\end{proof}

Functoriality of cohomology implies that  $\Omega$ and $\Omega_{\CC}$ are 
$\MCG_{g,n}$-invariant. As a consequence, the symplectic
form on the decorated representation space mentioned in
Corollary \ref{maincor:sympl-dec} is mapping class group invariant too.
%
%

\subsection{Decorated monodromy of spherical surfaces}

%
%
%

In this section we discuss general properties of the decorated monodromy representations of spherical surfaces with conical points and prove Theorem \ref{mainthm:mon-sph}.
Before doing that, we will show how to restrict a decorated homomorphism to a finite-index subgroup.

\subsubsection{Restriction of a decorated homomorphism}
Given a finite-index subgroup $\grp'\subset\grp_{g,n}$, 
we define the peripheral set
of $\grp'$ to be the set $\Bcal'$ of 
elements $\beta'\in\grp'$ that are indivisible in $\grp'$
and such that
$\beta'=\beta^k$ for some $\beta\in\Bcal$ and some $k\neq 0$.

\begin{remark}\label{rmk:cover-dec}
It is easy to see that, if $\dot{S}$ is a punctured surface with a base point $\ast$
endowed with an isomorphism $\grp_{g,n}\cong \pi_1(\dot{S},\base)$,
then $\grp'$ corresponds to a finite unbranched cover $\dot{S}'\rar\dot{S}$
and $\Bcal'$ corresponds to the set of peripheral elements in $\pi_1(\dot{S}')$.
\end{remark}

\begin{definition}[Restriction of a decorated homomorphism]
Let $(\rho,\Axis)$ be a decorated homomorphism with
$\rho:\grp_{g,n}\rar\SU_2$ and $\Axis:\Bcal\rar\su_2\setminus\{0\}$,
and let $\grp'\subseteq\grp_{g,n}$ be a finite-index subgroup
and $\Bcal'$ its peripheral set.
The {\it{decorated homomorphism induced by $(\rho,\Axis)$ on $\grp'$}}
is the couple $(\rho',\Axis')$, where
$\rho':\grp'\rar\SU_2$ is simply the restriction of $\rho$,
and $\Axis':\Bcal'\rar\su_2\setminus\{0\}$ 
is defined by requiring that 
$\Axis'([\beta']):=k\cdot\Axis(\beta)$,
whenever $\beta'=\beta^k$ with $\beta\in\Bcal$.
\end{definition}

Observe that $(\rho',\Axis')$ is a decorated homomorphism of $\grp'$
and that the image of $\hat{\Axis}'$ coincides with the image of $\hat{\Axis}$.
Indeed, if $[\beta']\in\Bcal'$,
then $\beta'=\beta^k$ for some $\beta\in\Bcal$ and some $k\neq 0$
and so $\hat{\Axis}'(\beta')=\hat{\Axis}(\beta)$ belongs to the image of $\hat{\Axis}$: this shows that
$\mathrm{Im}(\hat{\Axis}')\subseteq \mathrm{Im}(\hat{\Axis})$. 
Given $[\beta]\in\Bcal$, there exists a $k\neq 0$ such that $[\beta^k]\in\Bcal'$
because $\grp'$ has finite index in $\grp_{g,n}$. 
It follows that $\hat{\Axis}(\beta)=\hat{\Axis}'(\beta^k)$ belongs to the image
of $\hat{\Axis}'$: this shows that $\mathrm{Im}(\hat{\Axis})\subseteq\mathrm{Im}(\hat{\Axis}')$.

We also have the following dichotomy.

\begin{lemma}\label{lemma:finite-index-dec}
Let $g\geq 0$ and $n>0$ so that $2g-2+n>0$, and
let $(\rho,\Axis)$ be a decorated homomorphism with
$\rho:\grp_{g,n}\rar\SU_2$ and $\Axis:\Bcal\rar\su_2\setminus\{0\}$.
Suppose that all decorated homomorphisms induced on finite-index subgroups of $\grp_{g,n}$ are non-elementary.
Then either $\hat{\Axis}$ achieves infinitely many values
or $\rho$ has finite image.
\end{lemma}
\begin{proof}
Since $n>0$, the map $\hat{\Axis}$ achieves at least one value.
Suppose that $\mathrm{Im}(\hat{\Axis})$ is finite.
We want to show that $\mathrm{Im}(\rho)$ is finite.

Since $\mathrm{Im}(\hat{\Axis})$ is $\mathrm{Im}(\rho)$-invariant, there exists
a finite-index subgroup $\grp'$ of $\grp_{g,n}$
such that $\rho(\grp')$ fixes $\mathrm{Im}(\hat{\Axis})$ pointwise.
Denote by $(\rho',\Axis')$ the decorated homomorphisms induced
by $\grp'\subset\grp_{g,n}$. 

Now, the images of $\hat{\Axis}'$ and of $\hat{\Axis}$ coincide by the above observation
and $\rho'$ fixes them pointwise, thus $\rho'$ is coaxial.
If $\rho'$ were non-central, then $\mathrm{Im}(\hat{\Axis}')$ 
would contain either one point
or two antipodal points: in either case, $(\rho',\Axis')$ would be elementary,
against our hypothesis.
Hence, $\rho'$ is central and so $\rho$ has finite image.
\end{proof}

\subsubsection{Spin structures}

Let $S$ be a compact Riemann surface.
By \cite{atiyah-spin} a spin structure on $S$ corresponds to
a holomorphic line bundle $L$ on $S$ such that $L^{\otimes 2}=K$.
Since $H^1(S;\ZZ/2)$ can be identified to the group of holomorphic line bundles on $S$
of order two,
the set of spin structures is simply transitively acted on by $H^1(S;\ZZ/2)\cong(\ZZ/2)^{2g}$.

For the Riemann surface $\CC\PP^1$, there is a unique spin structure up to isomorphism.
Indeed, the square of $\mathbb{L}=\mathcal{O}_{\CC\PP^1}(-1)\rar\CC\PP^1$ is isomorphic
to $K_{\CC\PP^1}$. Note also that the total space of $\mathbb{L}$
can be identified to $\CC^2\setminus\{0\}$ and that 
the usual action of $\PSL_2(\CC)$ 
on $\CC\PP^1$ lifts to the standard action of
$\SL_2(\CC)$ on $\mathbb{L}=\CC^2\setminus\{0\}$.

Note that, upon identifying the center $Z(\SU_2)=\{\pm I\}$ with $\ZZ/2$,
an element of $H^1(S;\ZZ/2)$ can be viewed as a homomorphism
$\sigma: \grp_{g,n}\cong\pi_1(\dot{S})\rar\{\pm I\}$ 
that sends every peripheral element to $I$.
Thus, $H^1(S;\ZZ/2)$ acts on $\wh{\Hom}(\grp_{g,n},\SU_2)$ by sending
$(\rho,\Axis)$ to $(\sigma\cdot\rho,\Axis)$, where $\sigma\cdot\rho$ simply denotes the product in $\SU_2$: in particular, $H^1(S;\ZZ/2)$ simply transitively acts
on the set of $\SU_2$-liftings of a given $(\ol{\rho},\Axis)$ in $\wh{\Hom}(\grp_{g,n},\SO_3(\RR))$.

\begin{proof}[Proof of Theorem \ref{mainthm:mon-sph}(o)]
Let $h$ be a spherical metric on $(S,\bm{x})$ and let $\iota:\wti{\dot{S}}\rar\Sph$ be a developing map for $h$, with associated decorated monodromy homomorphism
$(\ol{\rho},\Axis)$.

The existence of an $\SU_2$-lifting for $(\ol{\rho},\Axis)$
was shown in \cite[Proposition A.1]{EMP}. Since such liftings can be put 
in bijection with $H^1(S;\ZZ/2)$, there are exactly $2^{2g}$ of such.

In order to see that an $\SU_2$-lifting of $(\ol{\rho},\Axis)$ bijectively corresponds to a choice of $L$ satisfying $L^{\otimes 2}\cong K_S$, it is enough to construct an $H^1(S;\ZZ/2)$-equivariant
map
\[
\mathcal{L}:\{\text{$\SU_2$-liftings of $(\rho,\Axis)$}\}\lra\{L\,|\,L^{\otimes 2}\cong K_S\}/
\text{isom.}
\]
Consider the action of $\grp_{g,n}\cong \pi_1(\dot{S})$
by deck transformations on $\wti{\dot{S}}$.
Given an $\SU_2$-lifting $\rho$ of $\ol{\rho}$,
we can lift the $\grp_{g,n}$-action from $\wti{\dot{S}}$ to
\[
\iota^*\mathbb{L}:=\{(\tilde{x},\iota(\tilde{x}),v)\in \wti{\dot{S}}\times\CC\PP^1\times(\CC^2\setminus\{0\})\ |\ [v]=\iota(\tilde{x})\}.
\]
The quotient determines a line bundle $\dot{L}$ on $\dot{S}$:
the isomorphism $\mathcal{O}_{\CC\PP^1}(-1)^{\otimes 2}\cong K_{\CC\PP^1}$
induces the isomorphism $\dot{L}^{\otimes 2}\cong K_{\dot{S}}$ on $\dot{S}$,
since $\iota$ is a local biholomorphism.

In order to extend the isomorphism $\dot{L}^{\otimes 2}\cong K_{\dot{S}}$ over
the puncture $x_i$ for all $i$, choose a coordinate $z$ in a disk neighbourhood
$V_i\subset S$ of $x_i$
and a local coordinate $w$ on $\CC\PP^1$ in such a way that the developing map
on $\dot{V}_i:=V_i\setminus\{x_i\}$ can be written as a multi-valued function $\hat{\iota}(z)=z^{\th_i}$.
This way the section $\iota^*\sqrt{dw}$ of $\dot{L}\big|_{\dot{V}_i}$ can be multi-valued.
Since the pull-back of $dw$ via $z\mapsto z^{\th}=w$ is
$\th z^{\th-1}dz$, the section $\sigma_i=z^{\frac{1-\th_i}{2}}(\iota^*\sqrt{dw})$ 
of $\dot{L}\big|_{\dot{V}_i}$ is single-valued. 
So we define $L\rar S$ to be the unique extension
of $\dot{L}$ whose local sections on $V_i$ are generated by $\sigma_i$.
It follows that the isomorphism $\dot{L}^{\otimes 2}\cong K_{\dot{S}}$,
whose restriction to $\dot{V}_i$ sends $\sigma_i^{\otimes 2}$ to $\th_i dz$,
extends over $V_i$ to an isomorphism $L^{\otimes 2}\cong K_S$.
Then we set $\mathcal{L}(\rho,\Axis):=L$.
It is easy to check that $\mathcal{L}$ is $H^1(S;\ZZ/2)$-equivariant.
\end{proof}

\subsubsection{Properties of decorated monodromy homomorphisms}

By Theorem \ref{mainthm:mon-sph}(o) proved above, we know that
decorated monodromy homomorphisms admit $\SU_2$-liftings.
Using this piece of information, we are ready to complete the proof of our last main result.

\begin{proof}[End of the proof of Theorem \ref{mainthm:mon-sph}]
As in the previous part of the proof,
let $h$ be a spherical metric on $(S,\bm{x})$
and let $\iota:\wti{\dot{S}}\rar\Sph$ be a developing map for $h$, with associated decorated monodromy homomorphism $(\ol{\rho},\Axis)$.  Also, fix an $\SU_2$-lifting $(\rho,\Axis)$ of $(\ol{\rho},\Axis)$.

(i) By contradiction, suppose that $(\rho,\Axis)$ is elementary. Since $n\geq 1$, 
the image of $\Axis$ is non-trivial and so the infinitesimal centralizer
$\hfrak=\Zfrak(\rho,\Axis)\subset\su_2$ is $1$-dimensional. Thus the decoration $\Axis$ takes values in $\hfrak$
and $\rho$ takes values in $H=\exp(\hfrak)\subset \SU_2$.
Fix now an element $0\neq X\in \hfrak$. It induces a nontrivial Killing vector field on $\Sph$ that vanishes at $\Sph\cap\hfrak$
(remember that we are viewing $\Sph$ as the unit sphere inside $\su_2$).
Its pull-back $\wti{V}$ on $\wti{\dot{S}}$ is $\pi_1(\dot{S})$-invariant holomorphic vector field, which descends to a nontrivial holomorphic
vector field $V$ on $\dot{S}$.
Note that the developing map extends to $\wh{S}$ by sending $\pa\wh{S}$ to $\Sph\cap\hfrak$, and so
the vector field $V$ vanishes at $\bm{x}$.
Since $\deg(T_S(-\bm{x}))=2-2g-n<0$, it follows that $V$ must vanish. This contradiction proves that $(\rho,\Axis)$ must be non-elementary.

As for the image of $\hat{\Axis}$ in $\Sph$, 
note first that every finite-index subgroup $\grp'\subset\grp_{g,n}$
corresponds to a finite unbranched cover $\dot{S}'\rar\dot{S}$ 
and that the decorated homomorphism $(\rho',\Axis')$ induced by $(\rho,\Axis)$
corresponds to the monodromy of the spherical metric $h'$ induced on $\dot{S}'$
by pull-back as in Remark \ref{rmk:cover-dec}.
At the beginning of the proof of (i) we have already shown 
that $(\rho',\Axis')$ cannot be elementary.
Hence, by Lemma \ref{lemma:finite-index-dec} there are two possibilities: either
$\hat{\Axis}$ achieves infinitely many values, or the image of $\rho$ is finite.
In the former case, we have already achieved the wished conclusion.
In the latter case, we consider the cover $\dot{S}'\rar\dot{S}$
associated to $\grp'=\mathrm{ker}(\rho)$, which has finite-index in $\grp_{g,n}$.
Since $\rho'$ is trivial, the developing map $\iota$ descends to 
$\dot{S}'$ and extends to a finite cover 
$\ol{\iota}':S'\rar\Sph$. Such cover $\ol{\iota}'$ 
ramifies at the conical points of $(S',h')$
and its branching locus is contained inside the image of $\hat{\Axis}'$.
Since $\chi(\dot{S}')<0$, it follows that the branching locus of $\ol{\iota}'$
contains at least three points and so
$\hat{\Axis}'$ must achieve at least three values.
%
%
%
%
%
%
%

(ii) It follows from (i) by Lemma \ref{mainlemma:anal}(a).
It is also in \cite[Theorem 5]{dey}.

(iii) 
Since $(\rho,\Axis)\in\wti{\Sigma}$, the image of $\Axis$ sits on a 
$\rho$-invariant plane $P$. The pull-back $\iota^{-1}(P\cap\Sph)$ descends to a graph
in $S$ that passes through the conical points. Let $\{\HEM_i\}$ be the connected components
of the complement of such graph. We want to show that each $\HEM_i$ is a hemisphere.

Let $\wti{\HEM}_i$ be a connected component of the preimage of $\HEM_i$ inside $\wh{S}$.
Since each isometry of $\wti{\HEM}_i$ fixes a point in $\wti{\HEM}_i$, the stabilizer of $\wti{\HEM}_i$
inside $\pi_1(\dot{S})$ is trivial and so $\wti{\HEM}_i\cong \HEM_i$.
Moreover $\iota$ properly maps $\wti{\HEM}_i$ to $\Sph\setminus P$ as a local isometry.
It follows that $\iota$ is an isometry of $\wti{\HEM}_i$ onto a connected component of $\Sph\setminus P$,
and so $\wti{\HEM}_i\cong \HEM_i$ is a hemisphere.

(iv) Since $\rho$ is central, then the developing map $\iota$ descends
$\dot{S}$ and extends to a cover $\overline{\iota}:S\rar\Sph$,
whose branch locus is contained inside the image of $\hat{\Axis}$.
Since $(\rho,\Axis)$ belongs to $\wti{\Sigma}_0$,
the decoration $\Axis$ takes values in a plane $P\subset\su_2$,
and so the branch locus is contained in the maximal circle $P\cap\Sph$.
\end{proof}

\printindex
\printbibliography


\makeatletter
\renewcommand{\@seccntformat}[1]{%
  \ifcsname prefix@#1\endcsname
    \csname prefix@#1\endcsname
  \else
    \csname the#1\endcsname\quad
  \fi}
\newcommand\prefix@section{Appendix \thesection. }
\makeatother

\appendix

\begin{refsection}

\section{Constraints on edge lengths of spherical polygons}\label{app:mamaev}

\begin{center}
by Daniil Mamaev
\end{center}


\begin{abstract}
	We give a short proof of a well-known criterion for the existence of polygons in $\mathbb S^3$ with given edge lengths. 
\end{abstract}

\subsection{Introduction}\label{app:intro}

We describe a solution to a classical problem on standard collections of unitary matrices, and restate it in the language of spherical polygons.

\begin{definition}A collection of elements $U_1,\ldots, U_n\in \mathrm {SU}(2)$ is called \emph{standard} if $U_1\cdot \ldots\cdot U_n=1$.
\end{definition}

\begin{question}[Existence of standard collections]\label{standardcol} For which $\bm{l}=(l_1,\ldots, l_n)\in [0,1]^n$ there exist a standard collection of matrices $U_1,\ldots, U_n\in \mathrm{SU}(2)$ 
such that the eigenvalues of $U_k$ are $e^{\pm i\pi\cdot l_k}$?
\end{question}

Recall that $\mathrm{SU}(2)$ can be naturally identified with the unit sphere $\mathbb S^3$, and as a result we can rephrase Question~\ref{standardcol} as follows.

Let $U_1,\ldots,U_n$ be a standard collection such that the eigenvalues of $U_k$ are $e^{\pm i\pi \cdot l_k}$ with $l_k \in [0, 1]$. Consider the set of partial products $p_i=U_1\cdot \ldots \cdot U_{i-1}$ so that
$$1=p_1, \quad
U_1=p_2,\dots, \quad
U_1\cdot \ldots \cdot U_{n-1}=p_n\in \mathrm{SU}(2).$$
Then for any two consecutive points $p_{k-1}, p_{k}$
the distance $d_{\mathbb S^3}(p_{k-1},p_{k})$ is $\pi\cdot l_{k-1}$. Note that in the case $l_{k-1}\notin \{0,1\}$ there exists a unique geodesic $p_{k-1}p_k$ of length $\pi\cdot l_{k-1}$ that joins $p_{k-1}$ and $p_k$. 

\begin{definition} A \emph{spherical polygon} in $\mathbb S^3$ with \emph{edge lengths} $\pi\cdot (l_1,\ldots,l_n)$ is a collection of points $p_1,\ldots, p_n$ such that $d_{\mathbb S^3}(p_{k-1},p_{k})=\pi\cdot l_{k-1}$
and $d_{\mathbb S^3}(p_n,p_1)=\pi\cdot l_n$.
A spherical polygon is called \emph{coaxial} if its vertices lie on one great circle.

For $\bm{l}=(l_1,\ldots, l_n)\in [0,1]^n$ we denote by $\mathpzc{Pol}(\bm{l})\subset (\mathbb S^3)^n$ the space of all spherical polygons with edge lengths $\pi\cdot\bm{l}$.
\end{definition}


\begin{question}[Existence of spherical polygons]\label{polyquestion} For which $\bm{l}\in [0,1]^n$ there exists a spherical polygon in $\mathbb S^3$ with edge lengths $(\pi l_1,\ldots ,\pi l_n)$?
In other words, for which $\bm{l}$ is the space $\mathpzc{Pol}(\bm{l})$ non-empty?
\end{question}

In order to state the answer we need one more definition.

\begin{definition} 
A vertex of the unit $n$-cube $[0,1]^n$ is \emph{odd} (resp. \emph{even}) if its coordinates add up to an odd (resp. even) integer. The 
{\it{$n$-demicube $DC_n$}} is the convex hull of all even vertices of $[0,1]^n$.
\end{definition}

We recall that the standard $\ell^1$-distance on $\RR^n$ is defined as
$d_1(\bm{x},\bm{y}):=\sum_{i=1}^n |x_i-y_i|$.

\begin{remark}\label{demicubechar} The $n$-demicube can be obtained from the unit $n$-cube by cutting out $2^{n-1}$ tetrahedra each containing one odd vertex and $n$ adjacent even vertices. Denote the set of odd vertices of $[0,1]^n$ by $V_{\odd}$. Then the $n$-demicube $DC_n$ coincides with the subset of $[0,1]^n$ consisting of points at $\ell^1$-distance at least $1$ from $V_{\odd}$.
\end{remark}

The main result of the appendix is the following theorem, which is essentially contained in the works of Galitzer~\cite{G97} and Biswas~\cite{B98}. 

\begin{theorem}\label{polygonsexistence} 
For $\bm{l}=(l_1,\ldots, l_n)\in [0,1]^n$ a spherical polygon with edges of lengths $\pi\cdot l_i$ exists if and only if $\bm{l}\in DC_n$,
namely, if and only if $d_1(\bm{l},V_{\odd})\geq 1$.
\end{theorem}

We conclude with a slight refinement of Theorem \ref{polygonsexistence}.

\begin{corollary}\label{coaxialunique} Let $\bm{l}\in [0,1]^n$.
\begin{itemize}
\item[(i)]
If $d_1(\bm{l},V_{\odd})=1$, then a polygon with edge lengths $\pi\cdot\bm{l}$ is unique up to an isometry of $\mathbb S^3$ and is coaxial.
\item[(ii)]
If $d_1(\bm{l},V_{\odd})>1$, then there exists a non-coaxial polygon with edge lengths
$\pi\cdot\bm{l}$.
\end{itemize} 
\end{corollary}

\subsection{Proof of Theorem \ref{polygonsexistence}}

The proof of Theorem \ref{polygonsexistence} is organized as follows.  In Lemma \ref{generatorsDn} we introduce the group $\mathrm{Iso}(DC_n)$ of self-isometries of the $n$-demicube $DC_n$. By Lemma  \ref{isobetweenMP} the spaces of polygons $\mathpzc{Pol}(\bm{l})$ 
and $\mathpzc{Pol}(\bm{l'})$
are isomorphic whenever $\bm{l},\bm{l'}$ belong to the same
$\mathrm{Iso}(DC_n)$-orbit.
Hence it is sufficient to determine a fundamental domain $F_n\subset[0,1]^n$
for the action of $\mathrm{Iso}(DC_n)$ on $[0,1]^n$
and prove that, for $\bm{l}\in F_n$, the space $\mathpzc{Pol}(\bm{l})$ is non-empty if and only if $\bm{l}\in DC_n\cap F_n$. This last step is achieved in Proposition \ref{prop: criterion for reduced vectors}.

\begin{lemma}\label{generatorsDn} The group $\mathrm {Iso}(DC_n)$ of self-isometries of the $n$-demicube is generated by the subgroup (isomorphic to $S_n$) of linear trasformations that permute the $n$ coordinates and by the symmetry $\tau_{12}$ defined as
$$\tau_{12}(l_1,l_2,l_3,\ldots, l_n):=(1-l_1, 1-l_2,l_3,\ldots, l_n).$$ 
\end{lemma}
\begin{proof}
Any isometry of the $n$-demicube is induced by an isometry of the unit $n$-cube that sends even vertices to even. Hence,  $\mathrm {Iso}(DC_n)$ is an index $2$ subgroup in the group of isometries of $[0,1]^n$. 
The group $\mathrm{Lin}(DC_n)$ of linear transformations that permute the coordinates is clearly isomorphic to $S_n$ and acts by isometries on $DC_n$.
It is easy to see that the subgroup $G$ of isometries
generated by $\mathrm{Lin}(DC_n)$ and
$\tau_{12}$ acts transitively on the set of vertices of $DC_n$. 
Now, $DC_n$ has $2^{n-1}$ even vertices and
the stabilizer of $(0,\ldots, 0)$ inside $G$
is exactly $\mathrm{Lin}(DC_n)$, which has order $n!$.
It follows that $G$ has order $n!\cdot 2^{n-1}$. 
We conclude that $G$ is a subgroup of $\mathrm{Iso}([0,1]^n)$ of index $2$,
contained inside $\mathrm{Iso}(DC_n)$, and so
$G=\mathrm{Iso}(DC_n)$.
\end{proof}

From now on we identify $S_n$ with $\mathrm{Lin}(DC_n)$.

\begin{remark}\label{l1perserve} 
The group $\mathrm{Iso}(DC_n)$ is acting on the unit cube sending its odd vertices to odd vertices. As a result, the action of $\mathrm{Iso}(DC_n)$ preserves the $\ell^1$-distance from points of $DC_n$ to $V_{\odd}$. 
Hence all $2^{n-1}$ simplicial faces of $DC_n$ that are at $\ell^1$-distance $1$ from $V_{\odd}$ are permuted by the action of $\mathrm {Iso}(DC_n)$.
\end{remark}

\begin{lemma}\label{isobetweenMP} (i) For any $\bm{l}\in [0,1]^n$
and any $g\in \mathrm{Iso}(DC_n)$ there is an isomorphism ${g_*:\mathpzc{Pol}(\bm{l})\to \mathpzc{Pol}(g \cdot \bm{l})}$. In particular, $\mathpzc{Pol}(\bm{l})$ is empty if and only if $\mathpzc{Pol}(g \cdot \bm{l})$ is.

(ii) Furthermore, the image of a coaxial polygon in $\mathpzc{Pol}(\bm{l})$ under $g_*$ is coaxial and if two polygons in $\mathpzc{Pol}(\bm{l})$ are congruent, then so are their images in $\mathpzc{Pol}(g \cdot \bm{l})$.
\end{lemma}
\begin{proof}\emph {(i).} Since $S_n$ is generated by transpositions $(i,i+1)$, according to Lemma \ref{generatorsDn} it is enough to prove the statement for all $g=(i,i+1)$ and $g=\tau_{12}$. 

Suppose $g=(1,2)$ (the case $g=(i,i+1)$ is identical). If $l_1=l_2$, there is nothing to prove. Otherwise, let $P\in \mathpzc{Pol}(\bm{l})$ be a polygon with vertices $p_1,\ldots, p_n$. Since $l_1\ne l_2$ we have $p_1\ne p_3$. Hence, there is a unique geodesic sphere $\mathbb S^2\subset \mathbb S^3$ equidistant from $p_1$ and $p_3$ (this $2$-sphere depends algebraically on positions of $p_1$ and $p_3$). Let $\sigma$ be the reflection of $\mathbb S^3$ with respect to such $\mathbb S^2$ and set $p_2'=\sigma(p_2)$. Then define $g_*(P)$ to be the polygon with vertices $p_1,p_2',\ldots, p_n$ 
We obtained the desired isomorphism $\mathpzc{Pol}(l_1,l_2,\ldots,l_n)\cong \mathpzc{Pol}(l_2,l_1,\ldots,l_n)$.

Suppose now $g=\tau_{12}$. Given a polygon $P\in \mathpzc{Pol}(\bm{l})$ with vertices $p_1,\ldots,p_n$, we let $g_*(P)$ be the polygon with vertices $p_1,-p_2,p_3,\ldots,p_n$, where $-p_2$ is the point of $\mathbb S^3$ antipodal to $p_2$.


\emph{(ii)} It is clear that both $g=(i,i+1)$ and $g=\tau_{12}$ send coaxial polygons 
to coaxial polygons, and congruent polygons to congruent polygons.
\end{proof}

Lemma \ref{isobetweenMP} implies that in order to understand the spaces  $\mathpzc{Pol}(\bm{l})$ of $n$-gons in $\mathbb S^3$ with arbitrary  edge lengths it will be enough to consider $\mathpzc{Pol}(\bm{l})$ with $\bm{l}\in F_n$, where:
%
\begin{equation*} 
	F_n: = \left\{(l_1, \ldots, l_n) \in [0,1]^n \ |\ l_1+l_2\leq 1,
	\ \text{and}\quad
	l_1 \ge l_2 \ge \ldots \ge l_n\right\}.
\end{equation*}

Indeed, we have the following.

\begin{lemma} \label{l: reduced vectors generate R^n}
$F_n$ is a fundamental domain for the action of $\mathrm{Iso}(DC_n)$ on $[0,1]^n$.
\end{lemma}

In order to motivate the above claim, endow
$\RR^n$ with the standard Euclidean product $(\cdot,\cdot)$
and
recall from \cite[pag.5]{H90} that the elements $\alpha_1,\dots,\alpha_n\in\RR^n$ defined as
\[
\alpha_1=-e_1-e_2,\qquad \alpha_i=e_{i-1}-e_i\ \text{for $i=2,\dots,n$}
\]
form a simple root system of type $D_n$.
The reflections $v\mapsto v-\frac{2(v,\alpha_i)}{(\alpha_i,\alpha_i)}\alpha_i$ generate the Weyl group $W(D_n)$
and the chamber
\begin{align*}
\Delta_{D_n} & =\{v\in\RR^n\,|\,(v,\alpha_j)\geq 0\quad\text{for $j=1,\dots,n$}\}=\\
&=\{v\in\RR^n\,|\,v_1+v_2\leq 0,\ \text{and}\quad v_{i-1}\geq v_i\quad\text{for $j=1,\dots,n$}\}
\end{align*}
is a fundamental domain for the action of $W(D_n)$ on $\RR^n$.

Note that the translation operator $T(v):=v+(\frac{1}{2},\dots,\frac{1}{2})$
conjugates the actions of $W(D_n)$ and of $\mathrm{Iso}(DC_n)$, namely
$T\cdot W(D_n)\cdot T^{-1}=\mathrm{Iso}(DC_n)$.

\begin{proof}[Proof of Lemma \ref{l: reduced vectors generate R^n}]
By the above discussion,
$T(\Delta_{D_n})\cap [0,1]^n$ is a fundamental domain
for the action of $\mathrm{Iso}(DC_n)$ on $[0,1]^n$.
We conclude by observing that $F_n=T(\Delta_{D_n})\cap [0,1]^n$.
\end{proof}

Next we determine for which $\bm{l}$ in a certain subdomain of $[0,1]^n$ 
the space $\mathpzc{Pol}(\bm{l})$ is non-empty.

\begin{proposition} \label{prop: criterion for reduced vectors}
Let $n \ge 2$ and $l_1, \ldots, l_n$ be real numbers satisfying ${1 \ge l_1 \ge l_2 \ge \ldots \ge l_n \ge 0}$ and ${1/2 \ge l_2}$.
Then $\mathpzc{Pol}(l_1, \ldots, l_n)$ is non-empty if and only if 
$l_1 \le l_2 + \ldots + l_n$.
\end{proposition}
\begin{proof}
	The `only if' part of this proposition is just the triangle inequality in $\mathbb S^3$. We prove the `if' part by induction on $n$. 

	If $n = 2$, then we have $1/2 \ge l_1 = l_2 \ge 0$ and all polygons in $\mathpzc{Pol}(\bm{l})$ consist of two overlapping edges, joining two points at distance $\pi\cdot l_1$ from each other.


	Suppose $n = 3$ and $(l_1, l_2, l_3)$ satisfies $1 \ge l_1 \ge l_2 \ge l_3 \ge 0$ and $l_1 \le l_2 + l_3$. Then there exist two degenerate triangles $T_-$ and $T_+$ lying on a great circle with two edges of lengths $\pi\cdot (l_2,l_3)$. The third edge of $T_-$ and $T_+$ has length $\pi\cdot (l_2-l_3)$ and $\pi\cdot (l_2+l_3)$. Since we can continuously deform $T_-$ to $T_+$ so that the length of the first two edges stays equal to $l_2$ and $l_3$, the third edge can achieve any length in $ \pi\cdot[l_2-l_3, l_2+l_3]$, and so in particular it can achieve length $\pi\cdot l_1$.
Note that in this case $n=3$ we are not using the hypothesis $l_2\leq\frac{1}{2}$.	

	Now let $n \ge 4$ and $(l_1, \ldots, l_n)$ be an $n$-tuple satisfying $1 \ge l_1 \ge \ldots \ge l_n \ge 0$, $l_{2} \le 1/2$ and $l_1 \le l_2 + \ldots +l_n$. There are two cases.
	\begin{enumerate}
		\item 
		$l_1 \le l_2 + \ldots + l_{n - 2} + (l_{n - 1} - l_n)$.\\ 
		In this case, by induction there exists a polygon with $(n - 1)$ vertices and edge lengths ${\pi\cdot (l_1, \ldots, l_{n - 2}, l_{n - 1} - l_n)}$. Replacing the last edge by two edges of lengths $\pi\cdot l_{n - 1}$ and $\pi\cdot l_n$ lying on the same great circle, we obtain the desired polygon with edge lengths $\pi\cdot (l_1, \ldots, l_n)$.
		\item 
		$l_1 > l_2 + \ldots + l_{n - 2} + (l_{n - 1} - l_n)$. \\
		 In this case we construct the desired polygon from scratch. The first edge is a segment of length $\pi\cdot l_1$ on a great circle $C$ starting at a point $p_1$. The edges from the second to the $(n - 2)$-th lie on the great circle $C$ and go in the direction opposite to the first one. Let $p_{n - 1}$ be the end of the $(n - 2)$-th edge. We have $\frac{1}{\pi}\cdot|p_1, p_{n - 1}| = l_1 - l_2- \ldots -l_{n + 2}$ and therefore 
		$$
			0 \le l_{n - 1} - l_n < \frac{1}{\pi}\cdot|p_1, p_{n - 1}| \le (l_{n - 1} + l_n) \le 1.
		$$
		We have seen in the case $n=3$ that there exists a triangle with edge lengths $\pi\cdot l_{n - 1}, \pi\cdot l_n, |p_1, p_{n - 1}|$. Thus there exists a point $p_n\in \mathbb S^3$ such that $|p_{n-1} p_n|=\pi\cdot l_{n-1}$ and $|p_n p_1|=\pi\cdot l_n$. Thus the points $p_1,\ldots,p_n$ are the vertices of a polygon in $\mathpzc{Pol}(\bm{l})$.
	\end{enumerate}
\end{proof}

Another advantage of working with $F_n$ instead of $[0,1]^n$ is the following.

\begin{lemma}\label{lemma:d1-achieved}
Let ${\bm{l} = (l_1, \ldots, l_n) \in F_n}$.
Then 
$$
d_1(\bm{l},V_{\odd})=d_1(\bm{l}, e_1) = (1 - l_1) + l_2 + \ldots + l_n	.
$$
\end{lemma}
\begin{proof}
It is clearly enough to show that $d_1(\bm{l},V_{\odd})$ is achieved
at $e_1=(1,0,\ldots,0)$.

Assume that the distance $d_1(\bm{l},  V_{\odd})$ 
is achieved at the point $\bm{x} = (x_1, \ldots, x_n) \in V_{\odd}$.
	If $x_1=0$ then for some $i\ge 2$ we have $x_i=1$, but since $l_1\ge l_i$, replacing $(x_1,x_i)$ by $(1,0)$ will not increase the distance $d_1(\bm{l},\bm{x})$. So we can assume $x_1=1$. Now, if for some $i>j\ge 2$ we have $x_i=x_j=1$, then
	replacing $x_i$ and $x_j$ by $0$ does not increase the distance $d_1(\bm{l},\bm{x})$, since ${l_i,l_j\le l_2 \le (l_1 + l_2)/2 \le 1/2}$. 
	We conclude that $e_1=(1,0,\ldots, 0)$ is indeed  closest to $\bm{l}$.
\end{proof}

Proposition \ref{prop: criterion for reduced vectors} motivates the introduction of the following set: 
\begin{align*}
	T_n &= \{\bm{l} \in F_n \ |\  l_1 \le l_2 + \ldots + l_n\}=\\
	&=\left\{\bm{l} \in [0,1]^n \ |\ l_1 \ge \ldots \ge l_n  \text{, } l_1+l_2 \le 1 \text{, and } l_1 \le l_2 + \ldots + l_n\right\}.
\end{align*}

\begin{lemma} \label{l: description of good vectors}  $T_n=DC_n\cap F_n$.
\end{lemma}
\begin{proof}
	Recall that $DC_n=\{\bm{l}\in[0,1]^n\ |\ d_1(\bm{l},V_{\odd})\geq 1\}$
by	Remark \ref{demicubechar}.  
	By Lemma \ref{lemma:d1-achieved} we conclude that
	\begin{align*}
	DC_n\cap F_n & =\{\bm{l} \in F_n \ |\ d_1(\bm{l},   V_{\odd}) \ge 1\} = \\
	&=\{\bm{l} \in F_n \ |\ (1 - l_1) + l_2 + \ldots + l_n \ge 1\} = T_n.
	\end{align*}
\end{proof}

\begin{proof}[Proof of Theorem \ref{polygonsexistence}] 
Let $E_n$ be the subset of $\bm{l}\in [0,1]^n$ for which $\mathpzc{Pol}(\bm{l})$ is not empty. We want to show that $E_n=DC_n$.

Now $DC_n$ is clearly $\mathrm{Iso}(DC_n)$-invariant
and $E_n$ is $\mathrm{Iso}(DC_n)$-invariant by
Lemma \ref{isobetweenMP}.
Since $F_n$ is a fundamental domain (Lemma \ref{l: reduced vectors generate R^n}), it is enough to show that $DC_n\cap F_n=E_n\cap F_n$.

Since every $\bm{l}\in F_n$ satisfies the hypotheses of Proposition
\ref{prop: criterion for reduced vectors},
we obtain $T_n=E_n\cap F_n$.
The proof is complete, as $T_n=DC_n\cap F_n$ by
Lemma \ref{l: description of good vectors}.
\end{proof}

\begin{proof}[Proof of Corollary \ref{coaxialunique}]
Recall that $F_n$ is the fundamental domain of the action of $\mathrm{Iso}(DC_n)$ on the unit cube (Lemma \ref{l: reduced vectors generate R^n}), and that such action preserves the $\ell^1$-distance from $V_{\odd}$ (Remark \ref{l1perserve}). Moreover every element $g\in\mathrm{Iso}(DC_n)$
determines an isomorphism between $\mathpzc{Pol}(\bm{l})$
and $\mathpzc{Pol}(g\cdot \bm{l})$ that
sends congruent polygons to congruent polygons and coaxial polygons to coaxial polygons (Lemma \ref{isobetweenMP}). So, it is enough to prove the claims for $\bm{l}\in F_n$. 

Let now $\bm{l}\in F_n$. 
By Lemma \ref{lemma:d1-achieved}, 
we have $d_1(\bm{l},V_{\odd})=d_1(\bm{l},e_1)=(1-l_1)+l_2+\dots+l_n$. 

(i) Assume $d_1(\bm{l},V_{\odd})=1$, so that $l_1=l_2+\ldots+l_n$.
By the triangular inequality in $\mathbb{S}^3$,
a polygon $P\in \mathpzc{Pol}(\bm{l})$ must be a segment in $\mathbb S^3$ of length $\pi\cdot l_1$ traced twice, so it is coaxial and unique up to isometry of $\mathbb S^3$.

(ii) Assume $d_1(\bm{l},V_{\odd})>1$, so that $l_1<l_2+\ldots+l_n$.
We first observe that, for $n=3$, no such polygon in $\mathpzc{Pol}(\bm{l})$ is coaxial. So we can assume $n\geq 4$. 
Note that, if $l_n=0$, we can reduce to studying the case of $(n-1)$-gons
with edge lengths $(l_1,\dots,l_{n-1})$. Thus we can assume $l_n>0$.

The below argument consists just on a closer inspection of
the proof of Proposition \ref{prop: criterion for reduced vectors}.

If $l_1\leq  l_2+\dots+l_{n-2}+(l_{n-1}-l_n)$, then
by induction there exists a non-coaxial polygon with $(n - 1)$ vertices and edge lengths ${\pi\cdot (l_1, \ldots, l_{n - 2}, l_{n - 1} - l_n+\eta)}$
with $\eta=l_n/2>0$.
Since $l_{n-1}-l_n+\eta \in (l_{n-1}-l_n,l_{n-1}+l_n)$, 
the last edge can then be replaced by two edges of lengths 
$\pi\cdot l_{n - 1}$ and $\pi\cdot l_n$, and
we obtain the desired non-coaxial polygon with edge lengths $\pi\cdot \bm{l}$.

If $l_1>l_2+\dots+l_{n-2}+(l_{n-1}-l_n)$, then
the last two edges of the polygon constructed in the proof of 
Proposition \ref{prop: criterion for reduced vectors}
do not lie on the same great circle, because
$$
			0\leq l_{n - 1} - l_n < \frac{1}{\pi}\cdot|p_1, p_{n - 1}| < (l_{n - 1} + l_n) \le 1.
		$$
		It follows that such polygon has edge length $\pi\cdot\bm{l}$
		and is not coaxial.	
\end{proof}


\printbibliography[heading=subbibintoc]
\end{refsection}

\end{document}